\documentclass[11pt]{amsart}

\usepackage[T1]{fontenc}

\usepackage[letterpaper,margin=1in]{geometry}
\usepackage{amssymb}
\usepackage{enumitem}
\usepackage{mathtools}
\usepackage{microtype}
\usepackage{tikz}
\usetikzlibrary{matrix,arrows}
\usepackage{tikz-3dplot}
\usepackage{verbatim}

\usepackage[colorlinks,
pagebackref=true,
pdftitle={Homology of Vietoris-Rips complexes of hypercube graphs via group actions},
pdfauthor={Federico Galetto, Jonathan Montaño, Zoe Wellner}
]{hyperref}

\usepackage{cleveref}

\theoremstyle{definition}
\newtheorem{definition}{Definition}[section]
\newtheorem{lemma}[definition]{Lemma}
\newtheorem{proposition}[definition]{Proposition}
\newtheorem{theorem}[definition]{Theorem}
\newtheorem{example}[definition]{Example}
\newtheorem{remark}[definition]{Remark}

\newtheorem{problem}[definition]{Problem}
\newtheorem*{headthm}{Main Theorem}

\numberwithin{equation}{section}

\renewcommand*\backref[1]{\ifx#1\relax \else (Cited on p.~#1) \fi}

\begin{document}

\title{Homology of Vietoris-Rips complexes of hypercube graphs via group actions}

\author{Federico Galetto}
\address{Federico Galetto\\ Department of Mathematics and Statistics, Cleveland State University, 
  2121 Euclid Avenue, RT 1515, Cleveland, OH 44115-2215, USA.}
\email{\href{mailto:f.galetto@csuohio.edu}{f.galetto@csuohio.edu}}
\urladdr{\href{https://math.galetto.org}{math.galetto.org}}

\author{Jonathan Montaño}
\address{Jonathan
  Montaño\\School of Mathematical and Statistical Sciences, Arizona
  State University, P.O. Box 871804, Tempe, AZ 85287-18041, USA.}
\email{\href{mailto:montano@asu.edu}{montano@asu.edu}}
\urladdr{\href{https://j-montano.github.io}{j-montano.github.io}}

\author{Zoe Wellner}
\address{Zoe Wellner\\School of
  Mathematical and Statistical Sciences, Arizona State University,
  P.O. Box 871804, Tempe, AZ 85287-18041, USA.}
\email{\href{mailto:zwellner@asu.edu}{zwellner@asu.edu}}
\urladdr{\href{https://zoewellner.com}{zoewellner.com}}

\begin{abstract}
  The Vietoris-Rips complex of a metric space is the simplicial
  complex whose faces are the subsets of points with pairwise distance
  bounded above by a given scale $r$. In this paper, we study
  Vietoris-Rips complexes on the vertex set of the $n$-dimensional
  hypercube equipped with the Hamming distance. These complexes are
  stable under the action of the automorphism group of the hypercube
  graph, also known as the hyperoctahedral group, which therefore acts
  on their homology groups. Our results completely describe the
  decomposition of these homology groups into irreducible
  representations of the hyperoctahedral group at scales
  $r\leqslant 3$ and $r=n-1$.
\end{abstract}

\keywords{hypercube graph, Hamming distance, Vietoris-Rips complex,
  homology, hyperoctahedral group}

\subjclass[2020]{Primary: 05E45, 05E18. Secondary: 55N31, 55T25,
  55U10.}

\maketitle

\section{Introduction}

A common technique in topological data analysis is to consider a data
set $X$ in a metric space with distance $d$ and attempt to recover
topological information about an underlying geometric object, often
modeled by a manifold $M$. To this end, for each
$r\in \mathbb{R}_{\geqslant 0}$ one considers the simplicial complex
\[
  \operatorname{VR}(X;r)=\{Y\subseteq X \mid d(x,y)\leqslant r \text{
    for all } x,y\in Y\},
\]
known as the \emph{Vietoris-Rips complex of $X$ at scale $r$}.
Considering the collection of these complexes over all values of $r$
gives rise to the notion of \emph{persistence}. The guiding
principle is that topological features that persist over a large range
of values of $r$ are more likely to reflect meaningful geometric
structure in the underlying data rather than noise.

Vietoris-Rips complexes, together with their connections to persistent
homology, have been extensively studied in a variety of settings,
including planar point sets \cite{zbMATH05723613}, geodesic spaces
\cite{zbMATH07037772,zbMATH07184351}, and equivariant spaces
\cite{zbMATH07729434}. More recently,
\cite{zbMATH07553291,zbMATH07819241,zbMATH07725058,zbMATH08181349}
studied the case in which $X=X^n$ is the vertex set of the
$n$-dimensional hypercube $Q_n$ equipped with the \emph{Hamming
  distance}, where the distance between two vertices, represented by
binary sequences of length $n$, is the number of coordinates in which
they differ. In these works, the authors determined the homotopy type
of the complexes $X^{n,r}:=\operatorname{VR}(X^n;r)$ for
$r\leqslant 3$ and $r=n-1$, and provided lower bounds for their Betti
numbers in the general case. We observe that when $r\geqslant n$, the
complex $X^{n,r}$ is contractible. As shown in \cite[Proposition
3.3]{zbMATH07819241}, the inclusion
$\operatorname{VR}(X^n;r)\hookrightarrow \operatorname{VR}(X^n;r+1)$
induces trivial maps in homology, so questions about persistence boil
down to the study of the individual homology groups.

In \cite[Question 8.7]{zbMATH07819241}, Adams and Virk asked how the
\emph{hyperoctahedral group} $\mathfrak{H}_n$, the automorphism group
of $Q_n$, acts on the homology of the complexes $X^{n,r}$. The main
goal of this paper is to address this question. Our results completely
describe the decomposition into irreducible representations of the
homology groups of $X^{n,r}$ for $r\leqslant 3$ and $r=n-1$. Our
approach combines methods from representation theory, algebraic
topology, and commutative algebra, together with extensive computer
calculations. It is worth noting that we do not make use of previous
characterizations of the homotopy type of $X^{n,r}$. In fact, our
methods provide an independent way to obtain the Betti numbers of
$X^{n,r}$ that may further generalize to other scales.

The action of $\mathfrak{H}_n$ on the hypercube has a long history in
enumerative and representation-theoretic problems. For example, Pólya
\cite{zbMATH03038319} studied this action in his computation of the
number of distinguishable Boolean functions on $n \leqslant 4$ Boolean
variables, a problem later connected by Shannon \cite{MR29860} to
switching circuits and resolved by Slepian \cite{zbMATH03081690} using
the representation theory of $\mathfrak{H}_n$; see also
\cite{zbMATH03193728}. In this paper, we exploit this
representation-theoretic perspective in a topological setting by
studying the action of $\mathfrak{H}_n$ on the homology of the
complexes $X^{n,r}$.

Throughout the paper, we work over an arbitrary field $\Bbbk$ of
characteristic zero. All homology groups have coefficients in $\Bbbk$,
and we omit $\Bbbk$ from the notation for homology, writing simply
$H_\bullet (X^{n,r})$ instead of $H_\bullet (X^{n,r}; \Bbbk)$.  In
characteristic zero, the irreducible representations of
$\mathfrak{H}_n$ are in bijection with bipartitions of $n$, that is,
pairs $(\lambda;\mu)$ where $\lambda$ is a partition of $n_1$ and
$\mu$ is a partition of $n_2$ such that $n_1+n_2=n$ \cite[Corollary
II.4]{zbMATH03590472}. We denote by $\{\lambda;\mu\}$ the irreducible
representation indexed by the bipartition $(\lambda;\mu)$; see
\Cref{sec:constr-dimens} for a brief summary of its construction.  The
parts $\lambda_i$ of a partition $\lambda$ are written in
non-increasing order. When a part $\lambda_i = m$ is repeated $n$
times, we write $m^n$ for a more compact notation.

\begin{headthm}[\Cref{thm:main0,thm:main1,thm:1,thm:2,thm:mainr-1}]
  For $n>r$, the following are the only nonzero homology groups of
  $X^{n,r}$. Moreover, the corresponding isomorphisms hold as
  representations of $\mathfrak{H}_n$.

  \begin{enumerate}[label=(\arabic*)]
  \item Scale $r=0$.
    \[
      H_0(X^{n,0}) \cong \bigoplus_{i=0}^n \{n-i;i\}.
    \]
  \item Scale $r=1$.
    \[
      H_0(X^{n,1}) \cong \{n;0\},
      \qquad
      H_1(X^{n,1}) \cong \bigoplus_{i=2}^n \{n-i;(i-1,1)\}.
    \]
  \item Scale $r=2$.
    \begin{gather*}
      H_0(X^{n,2}) \cong \{n;0\},\\
      H_3(X^{n,2}) \cong
      \left(
      \bigoplus_{j=1}^{n-3} \{n-3-j;(j,1^3)\}
      \right)
      \oplus
      \left(
      \bigoplus_{j=0}^{1}\;
      \bigoplus_{i=0}^{n-3-j}
      \{(n-2-i-j,1^{2+j});i\}
      \right).
    \end{gather*}
  \item Scale $r=3$.
    \begin{gather*}
      H_0(X^{n,3}) \cong \{n;0\},
      \qquad
      H_4(X^{n,3})
      \cong
      \bigoplus_{j=1}^{n-4}
      \{n-4-j;(j,1^4)\},\\
      H_7(X^{n,3})
      \cong
      \bigoplus_{i=0}^{n-4}
      \bigoplus_{j=0}^{\min\{4,n-4-i\}}
      \{(n-i-j,j);i\}.
    \end{gather*}
  \item Scale $r=n-1$, with $n\geqslant 4$.
    \[
      H_0(X^{n,n-1})
      \cong
      H_{2^{n-1}-1}(X^{n,n-1})
      \cong
      \{n;0\}.
    \]
  \end{enumerate}
\end{headthm}

Adams and Virk's original question may be interpreted in different
ways. At a more basic level, one may simply try to obtain a basis for
$H_\bullet (X^{n,r})$ and explain how $\mathfrak{H}_n$ acts on this
basis. This kind of description is built into our process and appears
typically in the form of induced representations from subgroups of
$\mathfrak{H}_n$ that fix certain homology generators up to
orientation. From there, one can abstract to a ``coordinate-free''
description such as finding the character of $\mathfrak{H}_n$ on
$H_\bullet (X^{n,r})$ (see, for example \cite[Section
1.8]{zbMATH01601795}). Decomposing into irreducible representations is
the next step, and a task that is not always
manageable. Interestingly, our results show that at scales
$r\leqslant 3$ and $r=n-1$ the representations $H_\bullet (X^{n,r})$
are \emph{multiplicity-free}, meaning that no irreducible
representation of $\mathfrak{H}_n$ appears as a direct summand more
than once.

The cases of scale $r=0$ and $r=1$, discussed in
\Cref{sec:scale-zero,sec:scale-one}, follow by directly analyzing the
action of $\mathfrak{H}_n$ on the chain complex of $X^{n,r}$. Although
these cases are easily understood from the point of view of algebraic
topology, we take them as an opportunity to introduce some concepts of
representation theory that may be less familiar to the reader.  In
\Cref{sec:suppl-filtr}, we introduce a filtration of the complex
$X^{n,r}$ whose associated spectral sequence is our main tool to
compute homology at higher scales. This filtration is designed to play
well with the $\mathfrak{H}_n$-action, and it allows us to take
advantage of the highly symmetric structure of $X^{n,r}$ to reduce
several computations to smaller cases. These observations are applied
in \Cref{sec:scale-two,sec:scale-three}, where we present the cases of
scale $r=2$ and $r=3$, respectively. We note that the scale three case
relies on purpose-built software solutions; our code is described in
\Cref{sec:comp-meth} and is available in our Github repository at
\url{https://github.com/galettof/VietorisRipsHypercube}. The case
$r=n-1$, which we refer to as the ``submaximal'' scale, is treated in
\Cref{sec:submaximal-scale}. We settle this case using tools of
commutative algebra, namely Hochster's Formula (see \cite[Theorem
5.5.1]{zbMATH01194481}) together with an explicit free resolution of
the Stanley-Reisner ideal of the simplicial complex $X^{n,n-1}$. We
conclude with some open problems in \Cref{sec:open-problems}.

\subsection*{Acknowledgments}

The authors are grateful for helpful discussions with many
mathematicians, including Henry Adams, Alessio D'Alì, Graham Denham,
Sara Faridi, Ziqin Feng, Ewgenij Gawrilow, Selvi Kara, Allen Knutson,
Gregory Lupton, Susan Morey, Claudiu Raicu, and Žiga Virk.

The first author was partially funded by NSF Grant DMS \#2200844. The
second author was partially funded by NSF Grant DMS \#2401522. The
Viking Cluster at Cleveland State University was used to perform many
of the initial computations that led to the results of this paper.

Part of the research in this article was carried out at the Fields
Institute for Research in Mathematical Sciences while the first author
was in residence during the Thematic Program in Commutative Algebra
and Applications. The authors thank the institute and the program's
organizers for their hospitality and for providing excellent
conditions for conducting research.

\section{Scale zero}
\label{sec:scale-zero}

The complex $X^{n,0}$ consists of $2^n$ disconnected points
corresponding to the vertices of the hypercube $Q_n$. As such,
$H_0 (X^{n,0})$ is a $\Bbbk$-vector space of dimension $2^n$.

\subsection{Chains as a representation}
\label{sec:chains-as-repr}

The vertices of $Q_n$ can be represented by binary sequences of length
$n$. We identify such sequences $u=(u_1,\dots,u_n)$ with the elements
of the group $\mathbb{Z}_2^n$. The $0$-dimensional simplicial chains
in $X^{n,0}$ are the singletons $[u]$ with $u\in \mathbb{Z}_2^n$.

The hyperoctahedral group $\mathfrak{H}_n$, also known as the Weyl
group of type B, can be realized as the semidirect product
$\mathbb{Z}_2^n \rtimes \mathfrak{S}_n$, with
$\sigma \in \mathfrak{S}_n$ acting on $u \in \mathbb{Z}_2^n$ by
\begin{equation*}
  \sigma (u_1,\dots,u_n) = (u_{\sigma^{-1} (1)}, \dots, u_{\sigma^{-1} (n)});
\end{equation*}
see \cite[Section 1.1]{zbMATH00047598}, as well as \cite[Section
5.5]{zbMATH01970438} for the basics of semidirect products.  The
operation in $\mathfrak{H}_n$ is given by
\begin{equation*}
  (s, \sigma) (u, \tau) = (s+\sigma u, \sigma \tau).
\end{equation*}
Thus, the action of $\mathfrak{S}_n$ on $\mathbb{Z}_2^n$ extends to
the following action of $\mathfrak{H}_n$ on $\mathbb{Z}_2^n$:
\begin{equation*}
  (s, \sigma) u= s+ \sigma u.
\end{equation*}
This naturally induces an action on the simplices containing these
sequences:
\begin{equation*}
  (s,\sigma) [u] = [(s,\sigma) u],
\end{equation*}
and on the vector spaces spanned by these simplices:
\begin{equation*}
  (s,\sigma) \langle [u] \rangle = \langle (s,\sigma) [u] \rangle,
\end{equation*}
which we enclose in angled brackets.  When writing an element of
$\mathbb{Z}_2^n$ as a vertex of a simplex, we may remove parentheses
to simplify the notation; for example, $[(0^n)]=[0^n]$. Therefore, we
can write the space of $0$-dimensional simplicial chains in $X^{n,0}$
as
\begin{equation*}
  C_0 (X^{n,0}) = \bigoplus_{u\in \mathbb{Z}_2^n} \langle [u] \rangle
  = \bigoplus_{u\in \mathbb{Z}_2^n} (u, 1_{\mathfrak{S}_n}) \langle [0^n] \rangle.
\end{equation*}
We note that $C_0 (X^{n,0})$ is a representation of $\mathfrak{H}_n$,
i.e., a vector space with a linear $\mathfrak{H}_n$-action.

\subsection{Chains as an induced representation}
\label{sec:chains-as-an}

The group algebra (or group ring) of $\mathfrak{H}_n$ is the
$\Bbbk$-vector space
\begin{equation*}
  \Bbbk [\mathfrak{H}_n] =
  \bigoplus_{(s,\sigma) \in \mathfrak{H}_n} \langle (s,\sigma)\rangle
\end{equation*}
spanned by the elements of $\mathfrak{H}_n$ with multiplication
defined by extending $\Bbbk$-linearly the group law of
$\mathfrak{H}_n$ \cite[Section 18.1]{zbMATH01970438}. Since
$\mathfrak{S}_n$ is a subgroup of $\mathfrak{H}_n$, the elements
$(0,\sigma)$ with $\sigma\in\mathfrak{S}_n$ span a subring of
$\Bbbk [\mathfrak{H}_n]$ isomorphic to the group algebra
$\Bbbk [\mathfrak{S}_n]$ of $\mathfrak{S}_n$. For every element of
$\mathfrak{H}_n$ we have an equality
$\langle (s,\sigma) \rangle = (s,1_{\mathfrak{S}_n}) \langle
(0,\sigma) \rangle$, from which we get
\begin{equation}\label{eq:1}
  \Bbbk [\mathfrak{H}_n] = \bigoplus_{s \in \mathbb{Z}_2^n}
  (s,1_{\mathfrak{S}_n}) \bigoplus_{\sigma \in \mathfrak{S}_n}
  \langle (0,\sigma)\rangle \cong \bigoplus_{s \in \mathbb{Z}_2^n}
  (s,1_{\mathfrak{S}_n}) \Bbbk [\mathfrak{S}_n].
\end{equation}
Now, since $(0^n)$ is fixed by permutations in $\mathfrak{S}_n$, the
space $\langle [0^n]\rangle$ is a trivial left module over
$\Bbbk [\mathfrak{S}_n]$. Therefore, using properties of tensor
products \cite[Section 10.4, Theorem 17, Corollary
18]{zbMATH01970438}, we get an isomorphism of $\Bbbk [\mathfrak{H}_n]$
modules
\begin{equation}\label{eq:2}
  \begin{split}
    \Bbbk [\mathfrak{H}_n] \otimes_{\Bbbk [\mathfrak{S}_n]} \langle [0^n]\rangle
    &\cong \bigoplus_{s \in \mathbb{Z}_2^n}
    (s,1_{\mathfrak{H}_n}) \left( \Bbbk [\mathfrak{S}_n]
    \otimes_{\Bbbk [\mathfrak{S}_n]} \langle [0^n]\rangle \right)\\
    &\cong \bigoplus_{s \in \mathbb{Z}_2^n}
    (s,1_{\mathfrak{H}_n}) \langle [0^n]\rangle
    \cong C_0 (X^{n,0}).
  \end{split}
\end{equation}
The tensor
$\Bbbk [\mathfrak{H}_n] \otimes_{\Bbbk [\mathfrak{S}_n]} \langle
[0^n]\rangle$ is the $\Bbbk [\mathfrak{H}_n]$-module obtained by
extension of scalars from the $\Bbbk [\mathfrak{S}_n]$-module
$\langle [0^n]\rangle$ \cite[Sections 10.4, 19.3]{zbMATH01970438}. In
the language of representation theory,
$\Bbbk [\mathfrak{H}_n] \otimes_{\Bbbk [\mathfrak{S}_n]} \langle
[0^n]\rangle$ is the $\mathfrak{H}_n$-representation induced by the
$\mathfrak{S}_n$-representation $\langle [0^n]\rangle$
\cite[Definition 1.12.2]{zbMATH01601795}. While the tensor product
makes certain algebraic properties more transparent, we prefer working
in the language of representation theory. The isomorphism in
\Cref{eq:2} can then be written as
\begin{equation*}
  C_0 (X^{n,0}) \cong \operatorname{Ind}^{\mathfrak{H}_n}_{\mathfrak{S}_{n}}
  \langle [0^n]\rangle.
\end{equation*}

\subsection{Decomposition into irreducibles}
\label{sec:decomp-into-irred}

A representation of a group is called irreducible (or simple) if it
has no proper nontrivial subrepresentation \cite[Definition
1.4.5]{zbMATH01601795}. Over the complex numbers, every
finite-dimensional representation of a finite group is semisimple,
i.e., isomorphic to a direct sum of irreducible subrepresentations;
this result is known as Maschke's Theorem \cite[Theorem
1.5.3]{zbMATH01601795}. The same result holds over a field whose
characteristic does not divide the order of the group \cite[Theorem
4.1.1]{zbMATH05943514}; in particular, this holds over any field of
characteristic zero.

The irreducible representations of the symmetric group
$\mathfrak{S}_n$ are in bijection with the partitions of $n$
\cite[Theorem 2.4.6]{zbMATH01601795}; this is true over any field of
characteristic zero \cite[Corollary 5.12.4]{zbMATH05943514}. By
partition of $n$, we mean a sequence
$\lambda = (\lambda_1,\dots,\lambda_n)$ with
$\lambda_1 \geqslant \lambda_2 \geqslant \dots\geqslant \lambda_n
\geqslant 0$ and $\lambda_1 + \lambda_2 +\dots + \lambda_n = n$. The
entries $\lambda_i$ are referred to as the parts of $\lambda$. We also
write $\lambda = (\dots i^{m_i} \dots 1^{m_1} 0^{m_0})$ to indicate
that the part $i$ is repeated $m_i$ times in $\lambda$. For example,
\begin{equation*}
  (1^k 0^{n-k}) = (\underbrace{1,\dots,1}_k ,\underbrace{0,\dots,0}_{n-k}).
\end{equation*}
For convenience, we may occasionally omit the trailing zeros in the
sequence $\lambda$.

The irreducible representation of $\mathfrak{S}_n$ indexed by the
partition $\lambda$, also known as a Specht module, is often denoted
in the literature as $S_\lambda$ or $V_\lambda$; we adopt the more
compact notation $\{\lambda\}$. We may also omit parentheses when
expanding a partition; so
\begin{equation*}
  \{\lambda\} = \{(\lambda_1,\dots,\lambda_r)\}
  = \{\lambda_1,\dots,\lambda_r\}
  = \{(\dots i^{m_i} \dots 1^{m_1} 0^{m_0})\}
  = \{\dots i^{m_i} \dots 1^{m_1} 0^{m_0}\}
\end{equation*}
are different notations we may use for the irreducible representation
of $\mathfrak{S}_n$ indexed by the partition
$\lambda = (\lambda_1,\dots,\lambda_r) = (\dots i^{m_i} \dots 1^{m_1}
0^{m_0})$ of $n$. An explicit construction for these representations
can be found in \cite[Section 2.3]{zbMATH01601795}. For now, we only need
to understand two irreducible representations. The representation
$\{n\}$, corresponding to the partition with a single part equal to
$n$, is a 1-dimensional trivial representation, i.e., for all
$\sigma \in \mathfrak{S}_n$ and $u\in \{n\}$, we have $\sigma u = u$.
The representation $\{1^n\}$, corresponding to the partition with $n$
parts equal to $1$, is a 1-dimensional sign (or alternating)
representation, i.e., for all $\sigma \in \mathfrak{S}_n$ and
$u\in \{1^n\}$, we have $\sigma u = \operatorname{sgn} (\sigma) u$
where $\operatorname{sgn} (\sigma) $ denotes the sign of the
permutation $\sigma$. As previously observed, the 0-dimensional
simplex $[0^n]$ spans a 1-dimensional trivial representation of
$\mathfrak{S}_n$; therefore, we have an isomorphism
$\langle [0^n]\rangle \cong \{n\}$ of representations of
$\mathfrak{S}_n$. This induces the isomorphism of representations of
$\mathfrak{H}_n$
\begin{equation*}
  C_0 (X^{n,0}) \cong \operatorname{Ind}^{\mathfrak{H}_n}_{\mathfrak{S}_{n}} \{n\},
\end{equation*}
which gives a purely representation-theoretic description of
$C_0 (X^{n,0})$.

The irreducible representations of the hyperoctahedral group
$\mathfrak{H}_n$ are in bijection with bipartitions of $n$, i.e.,
pairs $(\lambda;\mu)$ with $\lambda$ a partition of $n_1$ and $\mu$ a
partition of $n_2$ such that $n_1+n_2=n$ \cite[Corollary
II.4]{zbMATH03590472}; this is true over any field of characteristic
zero \cite[Theorem 5.5.6]{zbMATH01516227}. Note that $\lambda$ or
$\mu$ could be the empty partition, i.e., the only partition of 0,
which we denote simply by $0$.  We denote by $\{\lambda;\mu\}$ the
irreducible representation indexed by the bipartition $(\lambda;\mu)$;
we will summarize its construction and provide a dimension formula in
\Cref{sec:constr-dimens}. An explicit description of these irreducible
representations is given by Geissinger and Kinch in \cite[Section
II]{zbMATH03590472}; however, their notation swaps partitions, i.e.,
our $\{\lambda;\mu\}$ is isomorphic to their $\{\mu;\lambda\}$.  For
this section, we only need the irreducibles $\{\lambda;\mu\}$ with
both partitions having at most one part, i.e.,
$\{\lambda;\mu\} = \{n-i;i\}$ where $i$ is an integer satisfying
$0\leqslant i\leqslant n$; we give more information about these
representations in \Cref{exa:2}.  Conveniently, there is an
isomorphism of representations of $\mathfrak{H}_n$
\begin{equation}\label{eq:3}
  \operatorname{Ind}^{\mathfrak{H}_n}_{\mathfrak{S}_{n}} \{n\}
  \cong \bigoplus_{i=0}^n \{n-i;i\},
\end{equation}
presented, for example, in \cite[Lemma III.6]{zbMATH03590472}. Thus,
we have
\begin{equation}\label{eq:4}
  C_0 (X^{n,0}) \cong \bigoplus_{i=0}^n \{n-i;i\}.
\end{equation}

\subsection{Homology as a representation}
\label{sec:homol-as-repr}

We summarize our work so far.
\begin{enumerate}
\item The simplices $[u]$ with $u$ a vertex of $Q_n$ form a
  $\Bbbk$-basis of $C_0 (X^{n,0})$.
\item All vertices of $Q_n$ are in the $\mathfrak{H}_n$-orbit of
  $(0^n)$.
\item The stabilizer of $(0^n)$ is $\mathfrak{S}_n$.
\item The $\mathfrak{S}_n$-action on $\langle [0^n]\rangle$ is
  trivial, hence $\langle [0^n]\rangle \cong \{n\}$.
\item As a representation of $\mathfrak{H}_n$,
  $C_0 (X^{n,0}) \cong
  \operatorname{Ind}^{\mathfrak{H}_n}_{\mathfrak{S}_{n}} \langle [0^n]\rangle
  \cong \operatorname{Ind}^{\mathfrak{H}_n}_{\mathfrak{S}_{n}} \{n\}$.
\item Its decomposition into irreducibles is
  $C_0 (X^{n,0}) \cong \bigoplus_{i=0}^n \{n-i;i\}$.
\end{enumerate}
Since $X^{n,0}$ has dimension zero, we have
$H_0 (X^{n,0})\cong C_0 (X^{n,0})$ and the next result follows from
our previous discussion.
\begin{theorem}\label{thm:main0}
  For every integer $n\geqslant 0$, the only nonzero homology group of
  $X^{n,0}$ is
  \begin{equation*}
    H_0 (X^{n,0}) \cong \bigoplus_{i=0}^n \{n-i;i\}
  \end{equation*}
  as a representation of $\mathfrak{H}_n$.
\end{theorem}
It follows from \Cref{eq:29} that
$\dim_\Bbbk \{n-i;i\} = \binom{n}{i}$. Hence, we have:
\begin{equation*}
  \dim_\Bbbk \left(H_0 (X^{n,0}) \right) = \sum_{i=0}^n \dim_\Bbbk \{n-i;i\}
  = \sum_{i=0}^n \binom{n}{i} = 2^n,
\end{equation*}
which matches the number of vertices of $Q_n$.

\begin{remark}
  To compute the reduced homology of $X^{n,0}$, we augment the
  simplicial chain complex with the map
  $\partial_0\colon C_0 (X^{n,0}) \to C_{-1} (X^{n,0})$, where
  $C_{-1} (X^{n,0}) = \langle \varnothing\rangle$ is the $\Bbbk$-span
  of the empty face. The action of $\mathfrak{H}_n$ on the empty face
  is trivial, hence $C_{-1} (X^{n,0}) \cong \{n;0\}$, the irreducible
  trivial representation of $\mathfrak{H}_n$. Schur's Lemma states
  that every map of representations between two irreducibles is either
  zero or an isomorphism (see \cite[Proposition
  2.3.9]{zbMATH05943514}). Since $\partial_0$ is surjective, it is
  nonzero when restricted to $\{n;0\}$ and zero when restricted to all
  other irreducible representations in $C_0 (X^{n,0})$. Therefore, we
  have
  \begin{equation*}
    \widetilde{H}_0 (X^{n,0}) = \ker (\partial_0)
    \cong \bigoplus_{i=1}^n \{n-i;i\}
  \end{equation*}
  as a representation of $\mathfrak{H}_n$.
\end{remark}

\section{Scale one}
\label{sec:scale-one}

The complex $X^{n,1}$ is the graph $Q_n$, which is connected with
$2^n$ vertices and $n 2^{n-1}$ edges. Fix a spanning tree $T$ of
$Q_n$. Since $T$ is contractible, $X^{n,1}$ is homotopy equivalent to
the quotient space $Q_n / T$, which consists of a single point and a
loop for every edge of $Q_n$ not belonging to $T$. Since $T$ is a tree
on $2^n$ vertices, it has $2^n -1$ edges. Therefore, $X^{n,1}$ is
homotopy equivalent to a wedge sum of
\begin{equation*}
  n 2^{n-1} - (2^n -1) = (n-2)2^{n-1} +1
\end{equation*}
circles. It follows that $H_0 (X^{n,1}) \cong \Bbbk$ and
$H_1 (X^{n,1}) \cong \Bbbk^{(n-2)2^{n-1} +1}$ are the only nonzero
homology groups of $X^{n,1}$. For a source on this material, the
reader may consult \cite[Section 8.3]{zbMATH00049719}.

\subsection{The chain complex}
\label{sec:chain-complex}

To compute the homology of $X^{n,1}$ we construct the complex of
simplicial chains
\begin{equation*}
  C_0 (X^{n,1}) \xleftarrow{\partial_1} C_1 (X^{n,1})
\end{equation*}
where $C_0 (X^{n,0})$ is spanned by the vertices of $Q_n$, so it is
described as the representation of $\mathfrak{H}_n$ in \Cref{eq:4}. To
describe $C_1 (X^{n,1})$, we follow the outline in
\Cref{sec:homol-as-repr}.

\begin{enumerate}
\item The simplices $[u,v]$ with $u,v$ vertices of $Q_n$ such that
  $d(u,v)= 1$ form a $\Bbbk$-basis of $C_1 (X^{n,1})$.
\item All these 1-simplices are in the $\mathfrak{H}_n$-orbit of
  $[0^n,0^{n-1}1]$.
\item The stabilizer of $[0^n,0^{n-1}1]$ is the subgroup of
  $\mathfrak{H}_n$ consisting of the elements $(s,\sigma)$ such that
  $s\in \{0\}^{n-1} \times \mathbb{Z}_2$ and
  $\sigma (\{1,\dots,n-1\}) = \{1,\dots,n-1\}$. Hence, we have
  $\operatorname{Stab} ([0^n,0^{n-1}1]) \cong \mathfrak{S}_{n-1}
  \times \mathfrak{H}_1$.
\item As a 1-dimensional representation of
  $\mathfrak{S}_{n-1} \times \mathfrak{H}_1$,
  $\langle [0^n,0^{n-1}1] \rangle \cong U\otimes V$\footnote{Tensor
    products without a ring explicitly indicated as a subscript are
    taken over $\Bbbk$.}, with $U$ a 1-dimensional representation of
  $\mathfrak{S}_{n-1}$ and $V$ a 1-dimensional representation of
  $\mathfrak{H}_1$ (see \cite[Theorem 3.10.2]{zbMATH05943514}). The
  factor $\mathfrak{S}_{n-1}$ acts trivially on $[0^n,0^{n-1}1]$ so
  $U$ is the trivial representation $\{n-1\}$. The factor
  $\mathfrak{H}_1$ has only two irreducible representations, both
  one-dimensional: the trivial representation $\{1;0\}$, and the
  alternating representation $\{0;1\}$ on which the element of order
  two acts by a change of sign. The element of order two in
  $\mathfrak{H}_1 \subset \operatorname{Stab} ([0^{n-1}1,0^n])$
  corresponds to
  $((0,\dots,0,1),\operatorname{id}) \in \mathfrak{H}_n$, which acts
  on $[0^n,0^{n-1}1]$ by
  \begin{equation*}
    ((0,\dots,0,1),\operatorname{id}) [0^n,0^{n-1}1] = [0^{n-1}1,0^n]
    = -[0^n,0^{n-1}1]
  \end{equation*}
  when taking into account the orientation of the simplex. As a
  result, $V\cong \{0;1\}$. Therefore, as a representation of
  $\mathfrak{S}_{n-1} \times \mathfrak{H}_1$, we have
  $\langle [0^n,0^{n-1}1] \rangle \cong \{n-1\} \otimes \{0;1\}$.
\item As a representation of $\mathfrak{H}_n$,
  \begin{equation}\label{eq:5}
    C_1 (X^{n,1}) \cong \operatorname{Ind}^{\mathfrak{H}_n}_{\mathfrak{S}_{n-1} \times \mathfrak{H}_1}
    \langle [0^n,0^{n-1}1] \rangle \cong
    \operatorname{Ind}^{\mathfrak{H}_n}_{\mathfrak{S}_{n-1} \times \mathfrak{H}_1}
    ( \{n-1\} \otimes \{0;1\}).
  \end{equation}
\end{enumerate}

Since induction is transitive, we can induce from
$\mathfrak{S}_{n-1} \times \mathfrak{H}_1$ to
$\mathfrak{H}_{n-1} \times \mathfrak{H}_1$ on the way to
$\mathfrak{H}_n$. This allows us to distribute the induction over the
tensor product.
\begin{equation*}
  \begin{split}
    &\operatorname{Ind}^{\mathfrak{H}_n}_{\mathfrak{S}_{n-1} \times \mathfrak{H}_1}
      (\{n-1\} \otimes \{0;1\})
    \cong
    \operatorname{Ind}^{\mathfrak{H}_n}_{\mathfrak{H}_{n-1} \times \mathfrak{H}_1}
    \operatorname{Ind}^{\mathfrak{H}_{n-1} \times \mathfrak{H}_1}_{\mathfrak{S}_{n-1} \times \mathfrak{H}_1}
      (\{n-1\} \otimes \{0;1\})\\
    \cong{}
    &\operatorname{Ind}^{\mathfrak{H}_n}_{\mathfrak{H}_{n-1} \times \mathfrak{H}_1} \left(
    \operatorname{Ind}^{\mathfrak{H}_{n-1}}_{\mathfrak{S}_{n-1}} \{n-1\} 
    \otimes \operatorname{Ind}^{\mathfrak{H}_1}_{\mathfrak{H}_1} \{0;1\} \right)
    \cong
    \operatorname{Ind}^{\mathfrak{H}_n}_{\mathfrak{H}_{n-1} \times \mathfrak{H}_1} \left(
    \operatorname{Ind}^{\mathfrak{H}_{n-1}}_{\mathfrak{S}_{n-1}} \{n-1\}
    \otimes \{0;1\}
      \right)
  \end{split}
\end{equation*}
The inner induction can be dealt with as in \Cref{eq:3}. Then, using
the fact that both tensor products and induction commute with direct
sums, we get the following isomorphism.
\begin{equation*}
  \operatorname{Ind}^{\mathfrak{H}_n}_{\mathfrak{S}_{n-1} \times \mathfrak{H}_1}
  (\{n-1\} \otimes \{0;1\})
  \cong
  \bigoplus_{i=0}^{n-1} \operatorname{Ind}^{\mathfrak{H}_n}_{\mathfrak{H}_{n-1} \times \mathfrak{H}_1}
  \left( \{n-1-i;i\} \otimes \{0;1\} \right)
\end{equation*}
Finally, we can use a generalization of the Littlewood-Richardson rule
\cite[Section 4.9]{zbMATH01601795} to expand these inductions (see
\cite[Theorem III.2]{zbMATH03590472}, with the reminder that the
authors write $\{\mu;\lambda\}$ for our irreducible representation
$\{\lambda;\mu\}$). In fact, since we are tensoring with $\{0;1\}$, we
are looking at a special case, known as the Pieri rule \cite[Section
2.2]{zbMATH01001729}. Essentially, when inducing
$\{n-1-i;i\} \otimes \{0;1\}$ from
$\mathfrak{H}_{n-1} \times \mathfrak{H}_1$ to $\mathfrak{H}_n$ we can
either add the $1$ and the $i$ after the semicolons to obtain
$\{n-1-i;i+1\}$ or, if $i\geqslant 1$, we can combine them to obtain
$\{n-1-i;(i,1)\}$. After replacing $i$ with $i-1$, we have
\begin{equation*}
  C_1 (X^{n,1}) \cong
  \left(\bigoplus_{i=1}^{n} \{n-i;i\}\right) \oplus
  \left(\bigoplus_{i=2}^{n} \{n-i;(i-1,1)\} \right)
\end{equation*}
as a representation of $\mathfrak{H}_n$.

\subsection{Computing homology}
\label{sec:computing-homology}

We compute the homology of the following complex.
\begin{equation*}
  0\leftarrow C_0 (X^{n,1}) \xleftarrow{\partial_1} C_1 (X^{n,1}) \leftarrow 0
\end{equation*}
By definition,
$H_0 (X^{n,1}) \cong C_0 (X^{n,1}) / \operatorname{im} \partial_1$.
Schur's Lemma states that every map of representations between two
irreducibles is either zero or an isomorphism (see \cite[Proposition
2.3.9]{zbMATH05943514}).  The 1-dimensional irreducible representation
$\{n;0\}$ appears in $C_0 (X^{n,1})$ but not in $C_1 (X^{n,1})$,
therefore $\{n;0\}$ is not in $\operatorname{im} \partial_1$ by
Schur's Lemma. Because $X^{n,1}$ is connected, it follows that
$H_0 (X^{n,1}) \cong \{n;0\}$ and
\begin{equation}\label{eq:6}
  \operatorname{im} \partial_1 = \bigoplus_{i=1}^n \{n-i;i\}.
\end{equation}
Now, the irreducible representations appearing in $C_1 (X^{n,1})$ but
not in $\operatorname{im} \partial_1$ must map to zero by Schur's
Lemma.  Therefore, we have
\begin{equation}\label{eq:7}
  H_1 (X^{n,1}) \cong \ker \partial_1 = \bigoplus_{i=2}^n \{n-i;(i-1,1)\}.
\end{equation}

We can summarize the discussion above into the following result.
\begin{theorem}\label{thm:main1}
  For every integer $n\geqslant 1$, the nonzero homology groups of
  $X^{n,1}$ are
  \begin{equation*}
    H_0 (X^{n,1}) \cong \{n;0\}, \qquad
    H_1 (X^{n,1}) \cong \bigoplus_{i=2}^n \{n-i;(i-1,1)\},
  \end{equation*}
  as representations of $\mathfrak{H}_n$.
\end{theorem}

As a sanity check, let us compute the dimension of $H_1 (X^{n,1})$. We
will use the fact that the dimension of the Specht module $\{i-1,1\}$
is $i-1$ (see \cite[Theorem 3.10.2]{zbMATH01601795}) and the binomial
identity $\sum_{i=0}^n i\binom{n}{i} = n 2^{n-1}$. By \Cref{eq:29}, we
have
\begin{equation*}
  \dim_\Bbbk H_1 (X^{n,1})
  = \sum_{i=2}^n \dim_\Bbbk \{n-i;(i-1,1)\}
  = \sum_{i=2}^n \binom{n}{i} (i-1)
  = n 2^{n-1} -2^n +1
\end{equation*}
as recorded at the beginning of this section.

\section{A supplementary filtration and its spectral sequence}
\label{sec:suppl-filtr}

In \cite[Section 2.4]{zbMATH07819241}, the authors define the
\emph{cubic hull} of a simplex $\Delta$ in $X^{n,r}$ as the smallest
isometric copy of a cube $Q_p$ containing $\Delta$. Based on this
notion, we give the following.

\begin{definition}
  The \emph{cubic dimension} of a simplex $\Delta$ in $X^{n,r}$ is the
  nonnegative integer $p$ such that the cubic hull of $\Delta$ is
  isometric to $Q_p$.
\end{definition}

\begin{example}\label{exa:1}
  In $X^{n,2}$, the edge $[0^n,0^{n-1}1]$ can be embedded in $Q_1$ but
  not $Q_0$; therefore, it has cubic dimension 1.  Similarly, the edge
  $[0^n,0^{n-2}1^2]$ can be embedded in $Q_2$ but not $Q_1$;
  therefore, it has cubic dimension 2.
\end{example}

The following observations follow directly from the definition.
\begin{itemize}
\item The cubic dimension of a simplex in $X^{n,r}$ is between $0$ and
  $n$ included.
\item The cubic dimension of $\Delta$ is the number of nonconstant
  entries across the binary sequences corresponding to the vertices of
  $\Delta$.
\item Since the $\mathfrak{H}_n$-action preserves the Hamming
  distance, it also preserves cubic dimension. In particular,
  simplices belonging to the same $\mathfrak{H}_n$-orbit have the same
  cubic dimension.
\item Every face of a simplex $\Delta$ in $X^{n,r}$ has cubic
  dimension less than or equal to that of $\Delta$.
\end{itemize}

Let $X^{n,r}_p$ be the collection of all simplices in $X^{n,r}$ that
have cubic dimension at most $p$. By the last observation above, $X^{n,r}_p$ is a
subcomplex of $X^{n,r}$. Since $X^{n,r}_p \subseteq X^{n,r}_{p+1}$, we
have a finite increasing filtration $\{X^{n,r}_p\}_{p\geqslant 0}$ of
$X^{n,r}$ by simplicial complexes. The idea is to compute the homology
of $X^{n,r}$ using a spectral sequence associated to this filtration;
for this purpose, we follow the setup in \cite[Section
5.4]{zbMATH00595200}. The terms in the zero page are given by
\begin{equation*}
  E^0_{p,q} = C_{p+q} (X^{n,r}_p) / C_{p+q} (X^{n,r}_{p-1});
\end{equation*}
the differentials are inherited from the simplicial chain complex
$C_\bullet (X^{n,r})$. By the properties of cubic dimension, this is a
bounded spectral sequence (see \cite[Definition
5.4.2]{zbMATH00595200}).  The terms in the first page are given by
\begin{equation*}
  E^1_{p,q} = H_{p+q} (E^0_{p,\bullet}).
\end{equation*}
By \cite[Theorem 5.5.1]{zbMATH00595200}, the spectral sequence
converges, resulting in a filtration of $H_\bullet (X^{n,r})$. Since
our objects are vector spaces, the associated graded object of this
filtration on $H_\bullet (X^{n,r})$ is isomorphic to
$H_\bullet (X^{n,r})$. In practice, we have an isomorphism:
\begin{equation}\label{eq:8}
  H_i (X^{n,r}) \cong \bigoplus_{p+q=i} E^\infty_{p,q},
\end{equation}
where $E^\infty_{p,q} = E^j_{p,q}$ for a sufficiently large
$j\geqslant 0$.  Moreover, because the $\mathfrak{H}_n$-action on
$X^{n,r}$ preserves cubic dimension, the filtration
$\{X^{n,r}_p\}_{p\geqslant 0}$ is stable under the
$\mathfrak{H}_n$-action. This implies that $\mathfrak{H}_n$ acts on
all terms of the spectral sequence and \Cref{eq:8} is an isomorphism
of $\mathfrak{H}_n$-representations. We will explicitly use this
spectral sequence to describe the homology of $X^{n,r}$ as a
representation of $\mathfrak{H}_n$ in the later sections. The rest of
this section is dedicated to general facts about our spectral
sequence.

\subsection{Zeroth page}
\label{sec:zeroth-page}

Recall that $C_{p+q} (X^{n,r}_p)$ is the $\Bbbk$-vector space with
basis given by all simplices of dimension $p+q$ in $X^{n,r}_p$, i.e.,
all simplices of dimension $p+q$ and cubic dimension at most $p$ in
$X^{n,r}$.  Let $C^{n,r}_{p+q,p}$ denote the subspace of
$C_{p+q} (X^{n,r}_p)$ spanned by the simplices with cubic dimension
exactly $p$. With this notation, we have
$C_{p+q} (X^{n,r}_p) = C^{n,r}_{p+q,p} + C_{p+q}
(X^{n,r}_{p-1})$. Notice also that
$C^{n,r}_{p+q,p} \cap C_{p+q} (X^{n,r}_{p-1}) = 0$. By the second
isomorphism theorem, we get
\begin{equation}\label{eq:9}
  E^0_{p,q} = \frac{C^{n,r}_{p+q,p} + C_{p+q} (X^{n,r}_{p-1})}{C_{p+q} (X^{n,r}_{p-1})} \cong
  \frac{C^{n,r}_{p+q,p}}{C^{n,r}_{p+q,p} \cap C_{p+q} (X^{n,r}_{p-1})} \cong C^{n,r}_{p+q,p}.
\end{equation}
With this identification, the differential
$d^0_{p,q} \colon E^0_{p,q} \to E^0_{p,q-1}$ in the spectral sequence
is the restriction of the boundary map
$\partial_{p+q} \colon C_{p+q} (X^{n,r}) \to C_{p+q-1} (X^{n,r})$ to
the subspaces spanned by simplices with cubic dimension exactly $p$.

Fix integers $i,p$ with $i\geqslant 0$ and $0\leqslant p\leqslant
n$. The space $C^{n,r}_{i,p}$ can be described with same approach we
used for the terms of the chain complexes in scale zero and one.
\begin{enumerate}
\item By construction, the $i$-simplices of $X^{n,r}$ with cubic
  dimension exactly $p$ form a $\Bbbk$-basis of $C^{n,r}_{i,p}$.
\item By definition of cubic dimension, every simplex in the
  $\Bbbk$-basis of $C^{n,r}_{i,p}$ is in the $\mathfrak{H}_n$-orbit of
  a simplex of cubic dimension $p$ whose vertices all have their first
  $n-p$ coordinates equal to zero. More formally, the graph embedding
  $Q_p\hookrightarrow Q_n$ defined by sending a vertex
  $(u_1,\dots,u_p)$ to $(0,\dots,0,u_1,\dots,u_p)$ induces an
  embedding $X^{p,r} \hookrightarrow X^{n,r}$ of simplicial complexes
  with image $\{0\}^{n-p} \times X^{p,r}$. The $i$-simplices of cubic
  dimension $p$ in $\{0\}^{n-p} \times X^{p,r}$ span a subspace $V$ of
  $C^{n,r}_{i,p}$, and every simplex in the $\Bbbk$-basis of
  $C^{n,r}_{i,p}$ is in the $\mathfrak{H}_n$-orbit of a simplex in
  this $\Bbbk$-basis of $V$.
\item An element $(s,\sigma) \in \mathfrak{H}_n$ stabilizes the set of
  $i$-simplices of cubic dimension $p$ in $\{0\}^{n-p} \times X^{p,r}$
  if and only if it stabilizes the set of vertices
  $\{0\}^{n-p}\times \{0,1\}^p$ of this subcomplex.  The stabilizer of
  $\{0\}^{n-p}\times \{0,1\}^p$ consists of the elements
  $(s,\sigma) \in \mathfrak{H}_n$ such that
  $s\in \{0\}^{n-p} \times \mathbb{Z}_2^p$ and
  $\sigma (\{1,\dots,n-p\}) = \{1,\dots,n-p\}$. It is isomorphic to
  $\mathfrak{S}_{n-p} \times \mathfrak{H}_p$.
\item As a representation of
  $\mathfrak{S}_{n-p} \times \mathfrak{H}_p$,
  $V \cong \{n-p\} \otimes C^{p,r}_{i,p}$.
\item As a representation of $\mathfrak{H}_n$,
  \begin{equation}\label{eq:10}
    C^{n,r}_{i,p} \cong
    \operatorname{Ind}^{\mathfrak{H}_n}_{\mathfrak{S}_{n-p} \times \mathfrak{H}_p}
    \left( \{n-p\} \otimes C^{p,r}_{i,p} \right).
  \end{equation}
\end{enumerate}

\begin{example}
  The 1-simplices of cubic dimension 1 in $X^{n,2}$ are in the
  $\mathfrak{H}_n$-orbit of $[0^n,0^{n-1}1]$, which comes from the
  image of the embedding $Q_1 = [0,1] \hookrightarrow Q_n$ defined by
  $(u)\mapsto (0,\dots,0,u)$. Thus, we have
  $C^{1,2}_{1,1} = \langle [0,1]\rangle$ and
  \begin{equation*}
    C^{n,2}_{1,1} \cong
    \operatorname{Ind}^{\mathfrak{H}_n}_{\mathfrak{S}_{n-1} \times \mathfrak{H}_1}
    \langle [0^n,0^{n-1}1]\rangle \cong
    \operatorname{Ind}^{\mathfrak{H}_n}_{\mathfrak{S}_{n-1} \times \mathfrak{H}_1}
    \left( \{n-1\} \otimes \langle [0,1]\rangle \right).
  \end{equation*}
  Note that this is also the description of $C_1 (X^{n,1})$ we gave in
  \Cref{eq:5}, where all 1-simplices have cubic dimension exactly 1.

  The 1-simplices of cubic dimension 2 in $X^{n,2}$ are in the
  $\mathfrak{H}_n$-orbit of $[0^n,0^{n-2}11]$ or
  $[0^{n-1}1,0^{n-2}10]$, which are the images of the diagonals of the
  square under the embedding $Q_2 \hookrightarrow Q_n$ defined by
  $(u_1,u_2)\mapsto (0,\dots,0,u_1,u_2)$. Thus, we have
  $C^{2,2}_{1,2} = \langle [00,11], [01,10]\rangle$ and
  \begin{equation*}
    C^{n,2}_{1,2} \cong
    \operatorname{Ind}^{\mathfrak{H}_n}_{\mathfrak{S}_{n-2} \times \mathfrak{H}_2}
    \langle [0^n,0^{n-2}11], [0^{n-1}1,0^{n-2}10]\rangle \cong
    \operatorname{Ind}^{\mathfrak{H}_n}_{\mathfrak{S}_{n-2} \times \mathfrak{H}_2}
    \left( \{n-2\} \otimes \langle [00,11],[01,10]\rangle \right).
  \end{equation*}

  Since, every 1-simplex in $X^{n,2}$ has cubic dimension 1 or 2, we
  have an isomorphism of $\mathfrak{H}_n$-representations
  $C_1 (X^{n,2}) \cong C^{n,2}_{1,1} \oplus C^{n,2}_{1,2}$.
\end{example}

\subsection{First page}
\label{sec:first-page}

In general, we have an isomorphism of representations of $\mathfrak{H}_n$:
\begin{equation*}
  C_i (X^{n,r}) \cong \bigoplus_{p=0}^n C^{n,r}_{i,p},
\end{equation*}
where the summands on the right hand side may be zero. The summand
$C^{n,r}_{i,p}$ is a representation of $\mathfrak{H}_n$ induced from
the subgroup $\mathfrak{S}_{n-p} \times \mathfrak{H}_p$. This mix of
representations induced from different subgroups makes it harder to
compute the homology of $X^{n,r}$ directly from the complex
$C_\bullet (X^{n,r})$. One advantage of working with our spectral
sequence is that summands induced from different subgroups are
separated into different complexes whose homology is easier to
compute.

\begin{proposition}\label{pro:1}
  With the notation of this section, we have an isomorphism
  \begin{equation*}
    E^1_{p,q} \cong
    \operatorname{Ind}^{\mathfrak{H}_n}_{\mathfrak{S}_{n-p} \times \mathfrak{H}_p}
    \left( \{n-p\} \otimes H_{p+q} (C^{p,r}_{p+\bullet,p}) \right)
  \end{equation*}
  of representations of $\mathfrak{H}_n$ for all values of $p$ and
  $q$.
\end{proposition}

\begin{proof}
  By \Cref{eq:9,eq:10}, we have an isomorphism
  \begin{equation}\label{eq:11}
    E^0_{p,q} \cong
    \operatorname{Ind}^{\mathfrak{H}_n}_{\mathfrak{S}_{n-p} \times \mathfrak{H}_p}
    \left( \{n-p\} \otimes C^{p,r}_{p+q,p} \right)
  \end{equation}
  of representations of $\mathfrak{H}_n$. Using the isomorphism
  between representations and modules over the group algebra (see
  \cite[p.~843]{zbMATH01970438}), the right hand side is the object
  resulting from the application of the functor
  \begin{equation}\label{eq:12}
    \begin{split}
      \Bbbk [\mathfrak{H}_p]\text{-}\mathrm{mod}
      &\longrightarrow \Bbbk [\mathfrak{H}_n] \text{-} \mathrm{mod}\\
      - &\longmapsto \operatorname{Ind}^{\mathfrak{H}_n}_{\mathfrak{S}_{n-p} \times \mathfrak{H}_p}
          \left( \{n-p\} \otimes - \right)
    \end{split}
  \end{equation}
  to $C^{p,r}_{p+q,p}$. As $q$ varies, the representations
  $C^{p,r}_{p+q,p}$ form a chain complex $C^{p,r}_{p+\bullet,p}$ with
  maps coming from the restriction of the boundary map of
  $X^{p,r}$. Since the boundary map commutes with the action of
  $\mathfrak{H}_p$ by definition, it follows that
  $C^{p,r}_{p+\bullet,p}$ is a complex of
  $\Bbbk [\mathfrak{H}_p]$-modules and the isomorphism in \Cref{eq:11}
  extends to an isomorphism of complexes
  \begin{equation*}
    E^0_{p,\bullet} \cong
    \operatorname{Ind}^{\mathfrak{H}_n}_{\mathfrak{S}_{n-p} \times \mathfrak{H}_p}
    \left( \{n-p\} \otimes C^{p,r}_{p+\bullet,p} \right).
  \end{equation*}
  Thinking of induction from
  $\mathfrak{S}_{n-p} \times \mathfrak{H}_p$ to $\mathfrak{H}_n$ as
  the tensor product
  $\Bbbk [\mathfrak{H}_n] \otimes_{\Bbbk [\mathfrak{S}_{n-p} \times
    \mathfrak{H}_p]} -$, we notice that $\Bbbk [\mathfrak{H}_n]$ is a
  free, hence flat, left
  $\Bbbk [\mathfrak{S}_{n-p} \times \mathfrak{H}_p]$-module. This is
  always the case when inducing from a subgroup of a finite group (see
  \cite[Section 10A, Section 10 Exercise 20]{zbMATH03736029}) and a
  concrete example can be seen in \Cref{eq:1}.  Thus, the functor in
  \Cref{eq:12} is exact, so it commutes with homology. Therefore, the
  terms in the first page of the spectral sequence are given by
  \begin{equation*}
    E^1_{p,q} = H_{p+q} (E^0_{p,\bullet}) \cong
    \operatorname{Ind}^{\mathfrak{H}_n}_{\mathfrak{S}_{n-p} \times \mathfrak{H}_p}
    \left( \{n-p\} \otimes H_{p+q} (C^{p,r}_{p+\bullet,p}) \right)
  \end{equation*}
  as desired.
\end{proof}

Conceptually, \Cref{pro:1} says that computing the homology of
$X^{n,r}$ for an arbitrary $n$ reduces to computing the homology of
the complexes $C^{p,r}_{p+\bullet,p}$ for $0\leqslant p\leqslant n$,
which is coming from the homology of the subcomplex
$\{0\}^{n-p} \times X^{p,r}$.

\subsection{Second page}
\label{sec:second-page}

The first page of the spectral sequence contains horizontal complexes
$(E^1_{\bullet, q},d^1_{\bullet,q})$ with $q\in \mathbb{Z}$, and the
second page is formed by computing the homology of these complexes.
As seen in later sections, typically only the horizontal complex with
$q=0$ is nonzero, although there may be other sporadic nonzero
terms. The complex $(E^1_{\bullet, 0},d^1_{\bullet,0})$ is closely
related to a truncation of the cellular chain complex of the
$n$-dimensional unit cube $[0,1]^n$ considered as a regular CW-complex
with cubical cells.

We proceed to describe an action of $\mathfrak{H}_n$ on the cellular
chain complex $(\mathcal{F}_\bullet, \partial_\bullet)$ of
$[0,1]^n$. Ordering the vertices of a simplicial complex gives an
orientation of the simplices. We saw in \Cref{sec:chain-complex} that
the $\mathfrak{H}_n$-action on the simplicial complex $X^{n,1}$ may
stabilize a simplex as a set but change its orientation. In contrast,
the cells of a CW-complex do not come with a predetermined
orientation. In fact, an orientation of a cell $f$ in a CW-complex is
merely a choice between using $f$ and $-f$ as a generator of the group
of cellular chains with $\mathbb{Z}$-coefficients (see \cite[Section
3.B]{zbMATH02103273}). We will show that the $\mathfrak{H}_n$-action
on $[0,1]^n$ affects the orientation of cells in $\mathcal{F}_\bullet$
in a way that is inductively determined by the differentials
$\partial_\bullet$.

\begin{proposition}\label{lem:1}
  The cellular cell complex $(\mathcal{F}_\bullet, \partial_\bullet)$
  of the unit $n$-cube $[0,1]^n$ is a complex of
  $\mathfrak{H}_n$-representations with
  \begin{equation*}
    \begin{split}
      \mathcal{F}_i
      &\cong \operatorname{Ind}^{\mathfrak{H}_n}_{\mathfrak{S}_{n-i} \times \mathfrak{H}_i}
        (\{n-i\} \otimes \{0;1^i\})\\
      &\cong \left(
        \bigoplus_{j=0}^{n-i} \{n-i-j;(j+1,1^{i-1})\}
        \right)
        \oplus
        \left(
        \bigoplus_{j=1}^{n-i} \{n-i-j;(j,1^i)\}
        \right).
    \end{split}
  \end{equation*}
\end{proposition}

\begin{proof}
  Consider the (closed) $i$-cell
  \begin{equation*}
    f = \{0\}^{n-i} \times [0,1]^i.
  \end{equation*}
  For $j\in \{n-i+1,\dots,n\}$, the $j$-th front face of $f$ is
  \begin{equation*}
    f^0_j = \{x\in f \,|\, x_j=0\}
  \end{equation*}
  and the $j$-th back face of $f$ is
  \begin{equation*}
    f^1_j = \{x\in f \,|\, x_j=1\};
  \end{equation*}
  these are both (closed) $(i-1)$-cells in $[0,1]^n$. Following
  \cite[Section 7.2, Definition 2.3]{zbMATH00049719}, the differential
  $\partial_i$ applied to $f$ gives:
  \begin{equation*}
    \partial_i (f) = \sum_{j=n-i+1}^n (-1)^{j-(n-i)} (f^0_j - f^1_j).
  \end{equation*}
  Notice that $\mathfrak{H}_n$ acts by permuting the $i$-cells of
  $[0,1]^n$; in fact, all $i$-cells are in the $\mathfrak{H}_n$-orbit
  of $f$. Thus, $\mathcal{F}_i$ is a representation of
  $\mathfrak{H}_n$ and we can extend the definition of $\partial_i$ to
  other cells by letting
  \begin{equation*}
    \partial_i \left( (s,\sigma) f \right)
    = (s,\sigma) \partial_i (f)
  \end{equation*}
  for every $(s,\sigma) \in \mathfrak{H}_n$, which makes $\partial_i$
  a $\Bbbk [\mathfrak{H}_n]$-linear map.

  Next, we determine the structure of the terms $\mathcal{F}_i$ as
  $\mathfrak{H}_n$-representations.  Let us proceed by induction on
  the dimension of the cells.  The 0-cells of $[0,1]^n$ are also the
  vertices of $X^{n,r}$. As seen in \Cref{sec:scale-zero} for
  $X^{n,0}$, all vertices are in the $\mathfrak{H}_n$-orbit of
  $[0^n]$, which is stabilized by the subgroup
  $\mathfrak{S}_n$. Therefore, given that $\mathfrak{H}_n$ acts simply
  by permuting these vertices, we have
  \begin{equation*}
    \mathcal{F}_0 \cong \operatorname{Ind}^{\mathfrak{H}_n}
    _{\mathfrak{S}_n} \{n\}
    \cong \bigoplus_{i=0}^n \{n-i;i\},
  \end{equation*}
  which covers the base case.

  The $i$-cells of $[0,1]^n$ are in the $\mathfrak{H}_n$-orbit of
  $f = \{0\}^{n-i} \times [0,1]^i$, which is stabilized by the
  subgroup $\mathfrak{S}_{n-i} \times \mathfrak{H}_i$. Applying the
  differential, we get
  \begin{equation}\label{eq:13}
    \partial_i (f) = (-1)^{n-i}\sum_{j=n-i+1}^n (-1)^j (f^0_j - f^1_j).
  \end{equation}
  We show that $\partial_i (f)$ spans a representation of
  $\mathfrak{S}_{n-i} \times \mathfrak{H}_i$ isomorphic to
  $\{n-i\} \otimes \{0;1^i\}$. Since $\partial_i$ is
  $\Bbbk [\mathfrak{H}_n]$-linear, this will imply that $f$ also spans
  a representation isomorphic to $\{n-i\} \otimes \{0;1^i\}$.  Because
  adjacent transpositions generate the symmetric group
  \cite[Proposition 1.5.4]{zbMATH02186170}, we have that, as a
  subgroup of $\mathfrak{H}_n$,
  $\mathfrak{S}_{n-i} \times \mathfrak{H}_i$ is generated by the
  elements
  \begin{enumerate}
  \item $(0, (j\ j+1))$ for $j\in \{1,\dots,n-i-1\}$,
  \item $(0, (j\ j+1))$ for $j\in \{n-i+1,\dots,n-1\}$,
  \item $(e_j,\operatorname{id})$ for $j\in \{n-i+1,\dots,n\}$,
  \end{enumerate}
  where $e_j \in \mathbb{Z}_2^n$ is the $j$-th unit vector.  By
  definition, $\{n-i\} \otimes \{0;1^i\}$ is a 1-dimensional vector
  space on which the generators of
  $\mathfrak{S}_{n-i} \times \mathfrak{H}_i$ of the first type act
  trivially, while the generators of the second and third type act by
  a change of sign (see \Cref{sec:constr-dimens}). We will show that
  this is exactly the case for the span of $\partial_i (f)$ using
  \Cref{eq:13}. The generators of the first type act trivially on
  $\partial_i (f)$ because they only permute the first $n-i$
  coordinates which are all zero. For the second set of generators,
  observe that
  \begin{equation*}
    \begin{split}
      (0, (j\ j+1)) &\left[ (-1)^j (f^0_j - f^1_j) + (-1)^{j+1} (f^0_{j+1} - f^1_{j+1}) \right]\\
      = &\left[ (-1)^j (f^0_{j+1} - f^1_{j+1}) + (-1)^{j+1} (f^0_j - f^1_j) \right]\\
      = &-\left[ (-1)^j (f^0_j - f^1_j) + (-1)^{i+1} (f^0_{i+1} - f^1_{i+1}) \right];
    \end{split}
  \end{equation*}
  meanwhile, for $k\neq j,j+1$, we have
  \begin{equation*}
    (0, (j\ j+1)) (f^0_k - f^1_k) = -f^0_k + f^1_k = -(f^0_k - f^1_k)
  \end{equation*}
  because $(0, (j\ j+1))$ acts by a change of sign on these
  $(i-1)$-cells by inductive hypothesis. Therefore,
  $(0, (j\ j+1)) \partial_i (f) = -\partial_i (f)$.  For the
  generators of the third kind, note that
  \begin{equation*}
    (e_j,\operatorname{id}) (f^0_j - f^1_j) = (f^1_j - f^0_j) = -(f^0_j - f^1_j);
  \end{equation*}
  meanwhile, for $k\neq j$, we have
  \begin{equation*}
    (e_j,\operatorname{id}) (f^0_k - f^1_k) = -f^0_k + f^1_k = -(f^0_k - f^1_k)
  \end{equation*}
  because $(e_j,\operatorname{id})$ acts by a change of sign on these
  $(i-1)$-cells by inductive hypothesis. Therefore,
  $(e_j,\operatorname{id}) \partial_i (f) = -\partial_i (f)$.  This
  concludes the proof of the inductive step, showing that
  $\mathcal{F}_i$ is isomorphic to the induced representation in the
  statement of the proposition.

  Finally, the decomposition of $\mathcal{F}_i$ into irreducible
  representations can be obtained with the same reasoning employed at
  the end of \Cref{sec:chain-complex}. First, we induce from
  $\mathfrak{S}_{n-i}$ to $\mathfrak{H}_{n-i}$ using \cite[Lemma
  III.6]{zbMATH03590472}, which replaces the left factor in the tensor
  product with $\bigoplus_{j=0}^{n-i} \{n-i-j;j\}$. Second, we induce
  the products $\{n-i-j;j\} \otimes \{0;1^i\}$ to $\mathfrak{H}_n$
  using \cite[Theorem III.2]{zbMATH03590472}. In this case, according
  to the Pieri rule \cite[Section 2.2]{zbMATH01001729}, the partitions
  $(j)$ and $(1^i)$ after the semicolons can either be added entrywise
  to give $(j+1,1^{i-1})$ or combined to give $(j,1^i)$.
\end{proof}

In \Cref{pro:5}, we show that the cellular chain complex of the unit
$n$-cube is essentially uniquely determined by the
representation-theoretic structure described in \Cref{lem:1}.

Now, we describe the homology of the truncations of the cellular chain
complex of the unit $n$-cube.  For $i \in \{0,1,\dots,n\}$, denote
$(\mathcal{F}_{\leqslant i})_\bullet$ the truncation of
$(\mathcal{F}_\bullet, \partial_\bullet)$ in homological degree $i$,
i.e., the complex whose terms are defined by
\begin{equation*}
  (\mathcal{F}_{\leqslant i})_j =
  \begin{cases}
    \mathcal{F}_j, & 0\leqslant j\leqslant i,\\
    0, & \text{otherwise},
  \end{cases}
\end{equation*}
and with the same differentials
$\partial_j \colon \mathcal{F}_j \to \mathcal{F}_{j-1}$ between
nonzero terms.

\begin{proposition}\label{pro:2}
  For every $i \in \{1,\dots,n\}$, the nonzero homology groups of
  $(\mathcal{F}_{\leqslant i})_\bullet$ are
  \begin{equation*}
    H_0 ((\mathcal{F}_{\leqslant i})_\bullet) \cong \{n;0\},
    \qquad
    H_i ((\mathcal{F}_{\leqslant i})_\bullet) \cong
    \bigoplus_{j=1}^{n-i} \{n-i-j; (j,1^i)\}
  \end{equation*}
  as representations of $\mathfrak{H}_n$.
\end{proposition}

\begin{proof}
  The unit $n$-cube $[0,1]^n$ is contractible so the homology of its
  cellular chain complex is zero, except for
  $H_0 (\mathcal{F}_\bullet)$ which is 1-dimensional. We can use this
  information to show that
  \begin{equation}\label{eq:14}
    \operatorname{im} \partial_i \cong
    \bigoplus_{j=0}^{n-i} \{n-i-j; (j+1,1^{i-1})\}
  \end{equation}
  as a representation of $\mathfrak{H}_n$. We proceed by decreasing
  induction on $i$. For the base case, let $i=n$. Since
  $H_n (\mathcal{F}_\bullet) = 0$, the differential $\partial_n$ is
  injective. Hence, we have
  \begin{equation*}
    \operatorname{im} \partial_n \cong \mathcal{F}_n \cong \{0;1^n\},
  \end{equation*}
  which matches \Cref{eq:14}. For the inductive step, assume
  \Cref{eq:14} holds for $i+1\leqslant n$; we show it holds for
  $i$. Since $H_i (\mathcal{F}_\bullet) = 0$, the inductive hypothesis
  implies that
  \begin{equation}\label{eq:15}
    \ker \partial_i = \operatorname{im} \partial_{i+1} \cong
    \bigoplus_{j=0}^{n-i-1} \{n-i-1-j; (j+1,1^i)\}
    \cong \bigoplus_{j=1}^{n-i} \{n-i-j; (j,1^i)\}.
  \end{equation}
  Decomposing $\mathcal{F}_i$ into irreducibles representations as in
  \Cref{lem:1}, we deduce that
  \begin{equation*}
    \operatorname{im} \partial_i \cong \mathcal{F}_i / \ker \partial_i
    \cong \bigoplus_{j=0}^{n-i} \{n-i-j;(j+1,1^{i-1})\},
  \end{equation*}
  which proves the isomorphism in \Cref{eq:14}. In particular, since
  \begin{equation*}
    \operatorname{im} \partial_1
    \cong \bigoplus_{j=0}^{n-1} \{n-1-j;j+1\},
  \end{equation*}
  we have an isomorphism
  $H_0 (\mathcal{F}_\bullet) = \mathcal{F}_0 / \operatorname{im}
  \partial_1 \cong \{n;0\}$.

  Now, consider the truncation $(\mathcal{F}_{\leqslant
    i})_\bullet$. The differentials of the truncation are the same as
  the differentials $\partial_j$ of $\mathcal{F}_\bullet$ in the range
  $1\leqslant j\leqslant i$; hence, we have
  $H_0 ((\mathcal{F}_{\leqslant i})_\bullet) \cong \{n;0\}$ and
  $H_j ((\mathcal{F}_{\leqslant i})_\bullet) = 0$ for
  $0<j<i$. Finally, \Cref{eq:15} implies that
  \begin{equation*}
    H_i ((\mathcal{F}_{\leqslant i})_\bullet) \cong \ker \partial_i
    \cong \bigoplus_{j=1}^{n-i} \{n-i-j;(j,1^i)\}.
  \end{equation*}
\end{proof}

\section{Scale two}
\label{sec:scale-two}

We continue with the irreducible decomposition of the homology of
$X^{n,2}$ as an $\mathfrak{H}_n$-representation. As in
\Cref{sec:suppl-filtr}, we separate the computations by the cubic
dimension of the simplices involved.

\subsection{Cubic dimension zero and one}
\label{sec:size-zero-one}

There is a single simplex with cubic dimension $0$ in $X^{0,2}$, i.e., the
$0$-simplex indexed by an empty binary string. Hence, the complex
$C^{0,2}_{\bullet,0}$ has a single nonzero term, namely the
one-dimensional vector space $C^{0,2}_{0,0}$. Note that
$C^{0,2}_{0,0}$ is isomorphic to the irreducible trivial
representation $\{0;0\}$ of the (trivial) group $\mathfrak{H}_0$. It
follows that there is a single nonzero homology group
$H_0 (C^{0,2}_{\bullet,0}) \cong \{0;0\}$.  By \Cref{pro:1}, we
conclude that $E^1_{0,q} = 0$ for all $q\neq 0$, and
\begin{equation*}
  E^1_{0,0} \cong
  \operatorname{Ind}^{\mathfrak{H}_n}_{\mathfrak{S}_{n} \times \mathfrak{H}_0}
  \left( \{n\} \otimes \{0;0\} \right) \cong
  \operatorname{Ind}^{\mathfrak{H}_n}_{\mathfrak{S}_{n}} \{n\}.
\end{equation*}
Using the notation of \Cref{sec:second-page}, we also have
$E^1_{0,0} \cong \mathcal{F}_0$.

Similarly, there is a single simplex with cubic dimension $1$ in
$X^{1,2}$, i.e., the $1$-simplex $[0,1]$. Hence, the complex
$C^{1,2}_{\bullet,1}$ has a single nonzero term, namely the
one-dimensional vector space $C^{1,2}_{1,1}$.  As a representation of
$\mathfrak{H}_1$, $C^{1,2}_{1,1}$ is isomorphic to the irreducible
$\{0;1\}$ because acting on $[0,1]$ with
$(1,\operatorname{id}_{\mathfrak{S}_1})$ sends it to $[1,0]=-[0,1]$.
It follows that there is a single nonzero homology group
$H_1 (C^{1,2}_{\bullet,1}) \cong \{0;1\}$.  By \Cref{pro:1}, we
conclude that $E^1_{1,q} = 0$ for all $q\neq 0$, and
\begin{equation*}
  \begin{split}
    E^1_{1,0}
    \cong
    \operatorname{Ind}^{\mathfrak{H}_n}_{\mathfrak{S}_{n-1} \times \mathfrak{H}_1}
    \left( \{n-1\} \otimes \{0;1\} \right).
  \end{split}
\end{equation*}
Using the notation of \Cref{sec:second-page}, we also have
$E^1_{1,0} \cong \mathcal{F}_1$.

\subsection{Cubic dimension two}
\label{sec:size-two}

The simplices of cubic dimension $2$ in $X^{2,2}$, shown in
\Cref{fig:1}, are as follows.
\begin{itemize}
\item In dimension 1, we have the two diagonals of the square $Q_2$, so
  \begin{equation*}
    C^{2,2}_{1,2} = \langle [00,11], [01,10]\rangle.
  \end{equation*}
\item In dimension 2, we have four triangles inside the square $Q_2$, so
  \begin{equation*}
    C^{2,2}_{2,2} = \langle [00,01,10], [00,01,11], [00,10,11], [01,10,11]\rangle.
  \end{equation*}
\item In dimension 3, we have a single tetrahedron on the vertices of $Q_2$, so
  \begin{equation*}
    C^{2,2}_{3,2} = \langle [00,01,10,11]\rangle.
  \end{equation*}
\end{itemize}

\begin{figure}[htb]
  \centering
  \begin{tikzpicture}[scale=1]
    \tikzset{vertex/.style = {shape=circle,fill,thick,minimum size=3pt,inner sep=0pt,outer sep=0pt}}
    \begin{scope}[shift={(0,0)}]
      \node[vertex] (00) at (0,0) {};
      \node[vertex] (10) at (1,0) {};
      \node[vertex] (01) at (0,1) {};
      \node[vertex] (11) at (1,1) {};
      \draw[densely dotted] (00) -- (10) -- (11) -- (01) -- (00);
      \draw[thick] (00) -- (11);
    \end{scope}
    \begin{scope}[shift={(1.5,0)}]
      \node[vertex] (00) at (0,0) {};
      \node[vertex] (10) at (1,0) {};
      \node[vertex] (01) at (0,1) {};
      \node[vertex] (11) at (1,1) {};
      \draw[densely dotted] (00) -- (10) -- (11) -- (01) -- (00);
      \draw[thick] (01) -- (10);
    \end{scope}
    \begin{scope}[shift={(4,0)}]
      \node[vertex] (00) at (0,0) {};
      \node[vertex] (10) at (1,0) {};
      \node[vertex] (01) at (0,1) {};
      \node[vertex] (11) at (1,1) {};
      \draw[densely dotted] (00) -- (10) -- (11) -- (01) -- (00);
      \draw[thick] (00) -- (10) -- (01) -- (00);
    \end{scope}
    \begin{scope}[shift={(5.5,0)}]
      \node[vertex] (00) at (0,0) {};
      \node[vertex] (10) at (1,0) {};
      \node[vertex] (01) at (0,1) {};
      \node[vertex] (11) at (1,1) {};
      \draw[densely dotted] (00) -- (10) -- (11) -- (01) -- (00);
      \draw[thick] (00) -- (11) -- (01) -- (00);
    \end{scope}
    \begin{scope}[shift={(7,0)}]
      \node[vertex] (00) at (0,0) {};
      \node[vertex] (10) at (1,0) {};
      \node[vertex] (01) at (0,1) {};
      \node[vertex] (11) at (1,1) {};
      \draw[densely dotted] (00) -- (10) -- (11) -- (01) -- (00);
      \draw[thick] (00) -- (11) -- (10) -- (00);
    \end{scope}
    \begin{scope}[shift={(8.5,0)}]
      \node[vertex] (00) at (0,0) {};
      \node[vertex] (10) at (1,0) {};
      \node[vertex] (01) at (0,1) {};
      \node[vertex] (11) at (1,1) {};
      \draw[densely dotted] (00) -- (10) -- (11) -- (01) -- (00);
      \draw[thick] (11) -- (10) -- (01) -- (11);
    \end{scope}
    \begin{scope}[shift={(11,0)}]
      \fill[lightgray] (0,0) -- (1,0) -- (1,1) -- (0,1) -- (0,0);
      \node[vertex] (00) at (0,0) {};
      \node[vertex] (10) at (1,0) {};
      \node[vertex] (01) at (0,1) {};
      \node[vertex] (11) at (1,1) {};
      \draw[thick] (00) -- (10) -- (11) -- (01) -- (00);
      \draw[thick,densely dashed] (01) -- (10);
      \draw[line width=4pt,lightgray] (0.3,0.3) -- (0.7,0.7);
      \draw[thick] (00) -- (11);
    \end{scope}
  \end{tikzpicture}
  \caption{Simplices of cubic dimension 2 in $X^{2,2}$}
  \label{fig:1}
\end{figure}
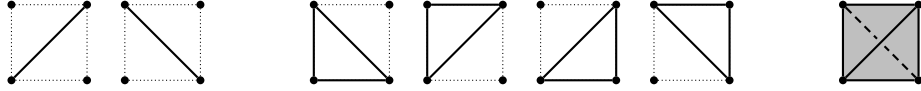

Thus, the complex $C^{2,2}_{\bullet,2}$ is
\begin{equation*}
  0\leftarrow C^{2,2}_{1,2}
  \xleftarrow{
    \left[
      \begin{smallmatrix}
        0&-1&-1&0\\
        1&0&0&1
      \end{smallmatrix}
    \right]
  }
  C^{2,2}_{2,2}
  \xleftarrow{
    \left[
      \begin{smallmatrix}
        -1\\1\\-1\\1
      \end{smallmatrix}
    \right]
  }
  C^{2,2}_{3,2} \leftarrow 0
\end{equation*}
with differentials obtained by restricting the boundary maps of
$X^{2,2}$ to the simplices of cubic dimension $2$. Since the first matrix has
rank $2$ and the second one has rank $1$, the only nonzero homology is
the $1$-dimensional space $H_2 (C^{2,2}_{\bullet,2})$.  The kernel of
the first matrix can be generated by the columns of
\begin{equation*}
  \left[\begin{smallmatrix}1&0\\0&1\\0&-1\\-1&0\end{smallmatrix}\right].
\end{equation*}
Therefore, we have an explicit presentation
\begin{equation*}
  \begin{split}
    H_2 (C^{2,2}_{\bullet,2}) =
    {}&\langle [00,01,10] - [01,10,11], [00,01,11] - [00,10,11]\rangle\\
    &/ \langle -[00,01,10]+[00,01,11]-[00,10,11]+[01,10,11] \rangle.
  \end{split}
\end{equation*}
The relation implies that the two generators are equal in the
quotient. Therefore, $H_2 (C^{2,2}_{\bullet,2})$ can be generated by
the class of either $[00,01,10] - [01,10,11]$ or
$[00,01,11] - [00,10,11]$.  To determine what
$H_2 (C^{2,2}_{\bullet,2})$ looks like as a representation of
$\mathfrak{H}_2$, we recall the concept of character of a
representation.

Let $V$ be a finite dimensional representation of
$\mathfrak{H}_n$. Every element $(s,\sigma) \in \mathfrak{H}_n$ acts
on $V$ as a linear operator. The function
$\chi_V \colon \mathfrak{H}_n \to \Bbbk$,
$(s,\sigma) \mapsto \operatorname{trace} (s,\sigma)$ is called the
\emph{character} of $V$. Since our field $\Bbbk$ has characteristic
zero, two representations $V$ and $W$ of $\mathfrak{H}_n$ are
isomorphic if and only if $\chi_V = \chi_W$ \cite[Corollary
4.2.4]{zbMATH05943514}.  Since the trace of a linear operator is
invariant under conjugation, characters are constant on each conjugacy
class of $\mathfrak{H}_n$.  The conjugacy classes of $\mathfrak{H}_n$
are in bijection with the bipartitions of $n$, as we now
explain. Paraphrasing from \cite[Section 3]{zbMATH00425590}, we can
determine the conjugacy class of an element
$(s,\sigma) \in \mathfrak{H}_n$ as follows. Decompose the underlying
permutation $\sigma \in \mathfrak{S}_n$ into a product of disjoint
cycles $\gamma_1,\dots,\gamma_m$. A cycle $\gamma = (i_1\dots i_k)$ in
this decomposition is \emph{balanced} if $s_{i_1} + \dots + s_{i_k}$
is even; otherwise, it is \emph{unbalanced}. Let $\alpha$ be the
partition representing the cycle type of the product of the balanced
cycles, and let $\beta$ be the partition representing the cycle type
of the product of the unbalanced cycles; then, $(s,\sigma)$ belongs to
the conjugacy class indexed by the bipartition $(\alpha;\beta)$ of
$n$.  The character table of $\mathfrak{H}_n$ is constructed as
follows.
\begin{itemize}
\item The rows of the table are indexed by the irreducible
  representations $\{\lambda;\mu\}$ of $\mathfrak{H}_n$.
\item The columns of the table are indexed by the conjugacy classes
  $(\alpha;\beta)$ of $\mathfrak{H}_n$.
\item The entry in row $\{\lambda;\mu\}$ and column $(\alpha;\beta)$
  is the character value $\chi_{\{\lambda;\mu\}} (s,\sigma)$ for any
  $(s,\sigma) \in \mathfrak{H}_n$ belonging to the conjugacy class
  indexed by $(\alpha;\beta)$.
\end{itemize}
The character table of $\mathfrak{H}_n$ can be produced using the
Macaulay2 package \texttt{BettiCharacters}
\cite{zbMATH07771377,BettiCharactersSource}, or other dedicated
software such as GAP \cite{GAP4}.

We now go back to describing $H_2 (C^{2,2}_{\bullet,2})$ as a
representation of $\mathfrak{H}_2$. Given that
$H_2 (C^{2,2}_{\bullet,2})$ is $1$-dimensional, it must be an
irreducible representation.  As illustrated in
\Cref{sec:irred-repr-hyper}, $\mathfrak{H}_2$ has five irreducible
representations: $\{2;0\}$, $\{1^2;0\}$, $\{0;2\}$, and $\{0;1^2\}$
are $1$-dimensional, and $\{1;1\}$ is $2$-dimensional. Thus,
$H_2 (C^{2,2}_{\bullet,2})$ is isomorphic to one of $\{2;0\}$,
$\{1^2;0\}$, $\{0;2\}$, or $\{0;1^2\}$, and we can determine which one
exactly by comparing their characters. The reader will find the
character table of $\mathfrak{H}_2$ in \Cref{tab:1}; the integers
above each column count the number of elements in the conjugacy class
for that column.
\begin{table}[htb]
  \centering
  $\begin{array}{c|rrrrr}
    &2&1&2&2&1\\
    &(2;0)&(1^2;0)&(1;1)&(0;2)&(0;1^2)\\ \hline
    \{2;0\}&1&1&1&1&1\\ 
    \{1^2;0\}&-1&1&1&-1&1\\ 
    \{1;1\}&0&2&0&0&-2\\ 
    \{0;2\}&1&1&-1&-1&1\\ 
    \{0;1^2\}&-1&1&-1&1&1\\
  \end{array}$
  \bigskip
  \caption{Character table of $\mathfrak{H}_2$}
  \label{tab:1}
\end{table}
Given the entries of the table, it suffices to look at the columns for
the conjugacy classes $(2;0)$ and $(1;1)$ to distinguish
irreducibles. The elements in the conjugacy class $(2;0)$ have a
single balanced cycle of length $2$, so they are $((0,0),(12))$ and
$((1,1),(12))$. The element $((0,0),(12))$ sends the generator
$[00,01,10] - [01,10,11]$ of $H_2 (C^{2,2}_{\bullet,2})$ to
\begin{equation*}
  [00,10,01] - [10,01,11] = -[00,01,10] + [01,10,11],
\end{equation*}
i.e., its negative. Therefore, the trace of $((0,0),(12))$ on
$H_2 (C^{2,2}_{\bullet,2})$ is $-1$. Similarly, the elements in the
conjugacy class $(1;1)$ have one balanced cycle and one unbalanced
cycle, each of length one, so they are
$((0,1),\operatorname{id}_{\mathfrak{S}_2})$ and
$((1,0),\operatorname{id}_{\mathfrak{S}_2})$. The element
$((0,1),\operatorname{id}_{\mathfrak{S}_2})$ sends the generator
$[00,01,10] - [01,10,11]$ of $H_2 (C^{2,2}_{\bullet,2})$ to
\begin{equation*}
  [01,00,11] - [00,11,10] = -[00,01,11] + [00,10,11]
  =-[00,01,10] + [01,10,11],
\end{equation*}
i.e., its negative. Therefore, the trace of
$((0,1),\operatorname{id}_{\mathfrak{S}_2})$ on
$H_2 (C^{2,2}_{\bullet,2})$ is also $-1$. We deduce that
$H_2 (C^{2,2}_{\bullet,2}) \cong \{0;1^2\}$.  By \Cref{pro:1}, we
conclude that $E^1_{2,q} = 0$ for all $q\neq 0$, and
\begin{equation*}
  E^1_{2,0} \cong
  \operatorname{Ind}^{\mathfrak{H}_n}_{\mathfrak{S}_{n-2} \times \mathfrak{H}_2}
  \left( \{n-2\} \otimes \{0;1^2\} \right).
\end{equation*}
Using the notation of \Cref{sec:second-page}, we also have
$E^1_{2,0} \cong \mathcal{F}_2$.

\subsection{Cubic dimension three}
\label{sec:size-three}

The simplices of cubic dimension $3$ in $X^{3,2}$ are as follows. Some
of them are depicted in \Cref{fig:2}.
\begin{itemize}
\item In dimension 2, we have eight triangles whose interior is
  completely contained in the interior of the cube $Q_3$, so
  \begin{equation*}
    \begin{split}
      C^{3,2}_{2,3} = \langle
      &[000, 011, 101], [000, 011, 110], [000, 101, 110], [001, 010, 100],\\
      &[001, 010, 111], [001, 100, 111], [010, 100, 111], [011, 101, 110]\rangle.
    \end{split}
  \end{equation*}
\item In dimension 3, we have eight tetrahedra with three facets lying
  on the boundary of the cube $Q_3$, and two tetrahedra with facets
  whose interior is completely contained in the interior of $Q_3$, so
  \begin{equation*}
    \begin{split}
      C^{3,2}_{3,3} = \langle
      &[000,001,010,100],[000,001,011,101],[000,010,011,110],\\
      &[000,100,101,110],[001,010,011,111],[001,100,101,111],\\
      &[010,100,110,111],[011,101,110,111] \rangle\\
      {+} \langle &[000,011,101,110],[001,010,100,111] \rangle.
    \end{split}
  \end{equation*}
\end{itemize}

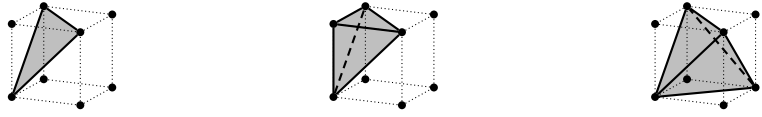
\begin{figure}[htb]
  \centering
  \hfill{}
  \tdplotsetmaincoords{75}{25}
  \begin{minipage}{0.25\textwidth}
    \begin{tikzpicture}[tdplot_main_coords,scale=1]
      \tikzset{vertex/.style = {shape=circle,fill,thick,minimum size=3pt,inner sep=0pt,outer sep=0pt}}
      \fill[lightgray] (0,0,0) -- (0,1,1) -- (1,0,1) -- (0,0,0);
      \node[vertex] (000) at (0,0,0) {};
      \node[vertex] (100) at (1,0,0) {};
      \node[vertex] (010) at (0,1,0) {};
      \node[vertex] (110) at (1,1,0) {};
      \node[vertex] (001) at (0,0,1) {};
      \node[vertex] (101) at (1,0,1) {};
      \node[vertex] (011) at (0,1,1) {};
      \node[vertex] (111) at (1,1,1) {};
      \draw[densely dotted] (000) -- (100) -- (110) -- (010) -- (000);
      \draw[densely dotted] (001) -- (101) -- (111) -- (011) -- (001);
      \draw[densely dotted] (000) -- (001);
      \draw[densely dotted] (100) -- (101);
      \draw[densely dotted] (010) -- (011);
      \draw[densely dotted] (110) -- (111);
      \draw[thick] (000) -- (011) -- (101) -- (000);
    \end{tikzpicture}      
  \end{minipage}
  \begin{minipage}{0.25\textwidth}
    \begin{tikzpicture}[tdplot_main_coords,scale=1]
      \tikzset{vertex/.style = {shape=circle,fill,thick,minimum size=3pt,inner sep=0pt,outer sep=0pt}}
      \fill[lightgray] (0,0,0) -- (0,0,1) -- (1,0,1) -- (0,0,0);
      \fill[lightgray] (0,0,1) -- (0,1,1) -- (1,0,1) -- (0,0,1);
      \node[vertex] (000) at (0,0,0) {};
      \node[vertex] (100) at (1,0,0) {};
      \node[vertex] (010) at (0,1,0) {};
      \node[vertex] (110) at (1,1,0) {};
      \node[vertex] (001) at (0,0,1) {};
      \node[vertex] (101) at (1,0,1) {};
      \node[vertex] (011) at (0,1,1) {};
      \node[vertex] (111) at (1,1,1) {};
      \draw[densely dotted] (000) -- (100) -- (110) -- (010) -- (000);
      \draw[densely dotted] (001) -- (101) -- (111) -- (011) -- (001);
      \draw[densely dotted] (000) -- (001);
      \draw[densely dotted] (100) -- (101);
      \draw[densely dotted] (010) -- (011);
      \draw[densely dotted] (110) -- (111);
      \draw[thick] (001) -- (011) -- (101) -- (001) -- (000) -- (101);
      \draw[thick,densely dashed] (000) -- (011);
    \end{tikzpicture}
  \end{minipage}
  \begin{minipage}{0.25\textwidth}
    \begin{tikzpicture}[tdplot_main_coords,scale=1]
      \tikzset{vertex/.style = {shape=circle,fill,thick,minimum size=3pt,inner sep=0pt,outer sep=0pt}}
      \fill[lightgray] (0,0,0) -- (0,1,1) -- (1,0,1) -- (0,0,0);
      \fill[lightgray] (0,1,1) -- (1,0,1) -- (1,1,0) -- (0,1,1);
      \fill[lightgray] (0,1,1) -- (1,1,0) -- (0,0,0) -- (0,1,1);
      \node[vertex] (000) at (0,0,0) {};
      \node[vertex] (100) at (1,0,0) {};
      \node[vertex] (010) at (0,1,0) {};
      \node[vertex] (110) at (1,1,0) {};
      \node[vertex] (001) at (0,0,1) {};
      \node[vertex] (101) at (1,0,1) {};
      \node[vertex] (011) at (0,1,1) {};
      \node[vertex] (111) at (1,1,1) {};
      \draw[densely dotted] (000) -- (100) -- (110) -- (010) -- (000);
      \draw[densely dotted] (001) -- (101) -- (111) -- (011) -- (001);
      \draw[densely dotted] (000) -- (001);
      \draw[densely dotted] (100) -- (101);
      \draw[densely dotted] (010) -- (011);
      \draw[densely dotted] (110) -- (111);
      \draw[thick] (000) -- (011) -- (101) -- (000) -- (110) -- (101);
      \draw[thick,densely dashed] (110) -- (011);
    \end{tikzpicture}
  \end{minipage}
  \hfill{}
  \caption{Some simplices of cubic dimension 3 in $X^{3,2}$}
  \label{fig:2}
\end{figure}

Thus, the complex $C^{3,2}_{\bullet,3}$ is
\begin{equation*}
  0\leftarrow C^{3,2}_{2,3} \xleftarrow{
    \left[\begin{smallmatrix}
      0&-1&0&0&0&0&0&0&-1&0\\
      0&0&-1&0&0&0&0&0&1&0\\
      0&0&0&-1&0&0&0&0&-1&0\\
      1&0&0&0&0&0&0&0&0&-1\\
      0&0&0&0&1&0&0&0&0&1\\
      0&0&0&0&0&1&0&0&0&-1\\
      0&0&0&0&0&0&1&0&0&1\\
      0&0&0&0&0&0&0&-1&1&0
    \end{smallmatrix}\right]
} C^{3,2}_{3,3} \leftarrow 0  
\end{equation*}
with the differential obtained by restricting the boundary map of
$X^{3,2}$ to the simplices of cubic dimension $3$. As seen by looking
at the first eight columns, this matrix has rank $8$, so the only
nonzero homology is the two-dimensional vector space
$H_3 (C^{3,2}_{\bullet,3})$. The kernel of the matrix can be generated
by the columns of
\begin{equation}\label{eq:16}
  \left[\begin{smallmatrix}
    0&1\\-1&0\\1&0\\-1&0\\0&-1\\0&1\\0&-1\\1&0\\1&0\\0&1
  \end{smallmatrix}\right],
\end{equation}
which are linearly independent.  Therefore, we have a presentation
\begin{equation*}
  \begin{split}
    H_3 (C^{3,2}_{\bullet,3}) = \langle
    &-[000,001,011,101] + [000,010,011,110] - [000,100,101,110]\\
    &+ [011,101,110,111] + [000,011,101,110],\\
    &[000,001,010,100] - [001,010,011,111] + [001,100,101,111]\\
    &- [010,100,110,111] + [001,010,100,111] \rangle.
  \end{split}
\end{equation*}
Throughout the rest of this section, we denote by $x$ and $y$ the two
generators in this presentation. Since they are linearly independent,
$\{x,y\}$ is a basis of $H_3 (C^{3,2}_{\bullet,3})$ as a
$\Bbbk$-vector space.

To understand $H_3 (C^{3,2}_{\bullet,3})$ as a representation of
$\mathfrak{H}_3$, we turn again to characters as in
\Cref{sec:size-two}. Given that $H_3 (C^{3,2}_{\bullet,3})$ is
2-dimensional, it must be an irreducible representation of dimension
2, or a direct sum of two irreducible representations of dimension
1. As a result of \Cref{eq:29}, $\mathfrak{H}_3$ has four
1-dimensional irreducible representations: $\{3;0\}$, $\{1^3;0\}$,
$\{0;3\}$, $\{0;1^3\}$, and two 2-dimensional irreducible
representations: $\{2,1;0\}$, $\{0;2,1\}$.  \Cref{tab:2} (computed
using version 2.6 of the Macaulay2 package \texttt{BettiCharacters})
contains the characters of these six irreducible representations.
\begin{table}[htb]
  \centering
  $\begin{array}{c|rrrrrrrrrr}
    &8&6&1&6&3&6&3&8&6&1\\
    &\left(3;0\right)&\left(2,1;0\right)&\left(1^{3};0\right)
          &\left(2;1\right)&\left(1^{2};1\right)&\left(1;2\right)
                &\left(1;1^{2}\right)&\left(0;3\right)
                    &\left(0;2,1\right)&\left(0;1^{3}\right)\\ \hline
    \{3;0\}&1&1&1&1&1&1&1&1&1&1\\ 
    \{2,1;0\}&-1&0&2&0&2&0&2&-1&0&2\\ 
    \{1^{3};0\}&1&-1&1&-1&1&-1&1&1&-1&1\\ 
    \{0;3\}&1&1&1&-1&-1&-1&1&-1&1&-1\\ 
    \{0;2,1\}&-1&0&2&0&-2&0&2&1&0&-2\\ 
    \{0;1^{3}\}&1&-1&1&1&-1&1&1&-1&-1&-1
  \end{array}$
  \bigskip
  \caption{Partial character table of $\mathfrak{H}_3$}
  \label{tab:2}
\end{table}
Observe that the element $(000,(12)) \in \mathfrak{H}_3$, which
belongs to the conjugacy class $(2,1;0)$, acts by
\begin{equation}\label{eq:17}
  (000,(12))x=-x, \qquad (000,(12))y=-y.
\end{equation}
In other words, the matrix representing the action of $(000,(12))$ on
$H_3 (C^{3,2}_{\bullet,3})$ with respect to the basis $\{x,y\}$ is
\begin{equation*}
  \begin{bmatrix}
    -1&0\\0&-1
  \end{bmatrix}
\end{equation*}
Therefore, the trace of $(000,(12))$ on $H_3 (C^{3,2}_{\bullet,3})$ is
$-2$. Since the characters of $\{2,1;0\}$ and $\{0;2,1\}$ at the
conjugacy class $(2,1;0)$ are both 0, we deduce that
$H_3 (C^{3,2}_{\bullet,3})$ is not isomorphic to either one, so it
must be a direct sum of two 1-dimensional irreducible representations.
The character of a direct sum of two representations is the sum of the
characters of those two representations \cite[Corollary
1.9.4]{zbMATH01601795}.  According to \Cref{tab:2}, the only
irreducible representations that have a negative character at the
conjugacy class $(2,1;0)$ are $\{1^3;0\}$ and $\{0;1^3\}$, each with a
value of $-1$, so $H_3 (C^{3,2}_{\bullet,3})$ must be isomorphic to a
direct sum of two of these (possibly the same one twice). On the other
hand, the element $(001,\operatorname{id}_{\mathfrak{S}_3})$, which
belongs to the conjugacy class $(1^2;1)$, acts by
\begin{equation}\label{eq:18}
  (001,\operatorname{id}_{\mathfrak{S}_3})x=y, \qquad
  (001,\operatorname{id}_{\mathfrak{S}_3})y=x.
\end{equation}
In other words, the matrix representing the action of
$(001,\operatorname{id}_{\mathfrak{S}_3})$ on
$H_3 (C^{3,2}_{\bullet,3})$ with respect to the basis $\{x,y\}$ is
\begin{equation*}
  \begin{bmatrix}
    0&1\\1&0
  \end{bmatrix}
\end{equation*}
Therefore, the trace of $(001,\operatorname{id}_{\mathfrak{S}_3})$ on
$H_3 (C^{3,2}_{\bullet,3})$ is $0$. According to \Cref{tab:2}, the
characters of the irreducible representations $\{1^3;0\}$ and
$\{0;1^3\}$ at the conjugacy class $(1^2;1)$ are $1$ and $-1$, so we
conclude that
$H_3 (C^{3,2}_{\bullet,3}) \cong \{1^3;0\} \oplus \{0;1^3\}$.  It
follows from \Cref{pro:1} that $E^1_{3,q} = 0$ for all $q\neq 0$, and
\begin{equation*}
  E^1_{3,0} \cong
  \left(
    \operatorname{Ind}^{\mathfrak{H}_n}_{\mathfrak{S}_{n-3} \times \mathfrak{H}_3}
    \left( \{n-3\} \otimes \{1^3;0\} \right)
  \right)
  \oplus
  \left(
    \operatorname{Ind}^{\mathfrak{H}_n}_{\mathfrak{S}_{n-3} \times \mathfrak{H}_3}
    \left( \{n-3\} \otimes \{0;1^3\} \right)
  \right).
\end{equation*}
Using the notation of \Cref{sec:second-page}, the first direct summand
is isomorphic to $\mathcal{F}_3$.

\begin{remark}\label{rem:1}
  The set $\{x+y,x-y\}$ is also a basis of $H_3 (C^{3,2}_{\bullet,3})$
  as a $\Bbbk$-vector space, since $\Bbbk$ has characteristic
  zero. \Cref{eq:17,eq:18} imply that
  \begin{align*}
    &(000,(12)) (x+y)=-(x+y),& &(000,(12)) (x-y)=-(x-y),\\
    &(001,\operatorname{id}_{\mathfrak{S}_3}) (x+y)=x+y,&
    &(001,\operatorname{id}_{\mathfrak{S}_3}) (x-y)=-(x-y).
  \end{align*}
  Similarly, one can verify that
  \begin{equation*}
    (000,(13)) (x+y)=-(x+y),\qquad (000,(13)) (x-y)=-(x-y).    
  \end{equation*}
  Since the elements $(000,(12))$, $(000,(13))$, and
  $(001,\operatorname{id}_{\mathfrak{S}_3})$ generate
  $\mathfrak{H}_3$, this shows that the subspaces $\langle x+y\rangle$
  and $\langle x-y\rangle$ of $H_3 (C^{3,2}_{\bullet,3})$ are stable
  under the action of $\mathfrak{H}_3$.  Moreover, using \Cref{tab:2},
  we see that $\langle x+y\rangle \cong \{1^3;0\}$ and
  $\langle x-y\rangle \cong \{0;1^3\}$.  It follows that $x+y$
  generates the summand
  $\operatorname{Ind}^{\mathfrak{H}_n}_{\mathfrak{S}_{n-3} \times
    \mathfrak{H}_3} \left( \{n-3\} \otimes \{1^3;0\} \right)$ of
  $E^1_{3,0}$ as a $\Bbbk [\mathfrak{H}_n]$-module, while $x-y$
  generates the summand
  $\operatorname{Ind}^{\mathfrak{H}_n}_{\mathfrak{S}_{n-3} \times
    \mathfrak{H}_3} \left( \{n-3\} \otimes \{0;1^3\} \right)$ of
  $E^1_{3,0}$ as a $\Bbbk [\mathfrak{H}_n]$-module.
\end{remark}

\subsection{Cubic dimension four and higher}
\label{sec:cubic-dimension-four}

We will give here a complete description of the simplices of cubic
dimension 4 and higher in $X^{n,2}$. As we will see, these simplices
do not contribute to the homology of $X^{n,2}$.

\begin{lemma}\label{lem:2}
  Let $e_i$ be the binary string of length $p$ containing a single 1
  in the $i$-th position.
  \begin{enumerate}[label=(\arabic*)]
  \item\label{item:1} The smallest dimension of a simplex of cubic
    dimension $p\geqslant 4$ in $X^{p,2}$ is $p-1$.
  \item\label{item:2} For every $p\geqslant 4$, every facet of cubic
    dimension $p$ in $X^{p,2}$ is in the $\mathfrak{H}_p$-orbit of
    $[0^p,e_1,e_2,\dots,e_p]$,
  \item\label{item:3} For every $p\geqslant 4$, the complex
    $C^{p,2}_{\bullet,p}$ has the form
    \begin{equation*}
      \cdots \leftarrow 0 \leftarrow C^{p,2}_{p-1,p}
      \xleftarrow{d} C^{p,2}_{p,p} \leftarrow 0 \leftarrow \cdots
    \end{equation*}
    with $C^{p,2}_{p-1,p} \cong C^{p,2}_{p,p} \cong \Bbbk^{2^p}$.
  \item\label{item:4} For every $p\geqslant 4$, the complex
    $C^{p,2}_{\bullet,p}$ has no homology.
  \end{enumerate}
\end{lemma}

\begin{proof}
  The only simplices in $X^{p,2}$ having diameter one are the
  1-simplices corresponding to the edges of the hypercube $Q_p$, which
  have cubic dimension 1. Thus, if $\Delta$ is a simplex of cubic
  dimension $p\geqslant 4$ in $X^{p,2}$, then it has diameter 2. Up to
  the action of $\mathfrak{H}_p$, we can assume that $\Delta$ contains
  the vertices $e_1$ and $e_2$. To increase its cubic dimension,
  $\Delta$ must contain a vertex $x=(x_1,\dots,x_p)$ having at least
  one entry $x_i=1$ with $i>2$. In order to have both
  $d(e_1,x)\leqslant 2$ and $d(e_2,x)\leqslant 2$, $x$ must have a
  single entry $x_i=1$ with $i>2$, and either $x_1=x_2=0$ (in which
  case, $x=e_i$) or $x_1=x_2=1$. If $x_1=x_2=1$, then acting with
  $(110\dots 0, \operatorname{id}_{\mathfrak{S}_p}) \in
  \mathfrak{H}_p$ swaps $e_1$ with $e_2$, and replaces $x$ with
  $e_i$. Also, acting with $(0^p,(3i)) \in \mathfrak{H}_p$, we can
  always assume that the new vertex added is $e_3$. It follows that
  each vertex in $\Delta$ beyond $e_1$ and $e_2$ can only increase
  cubic dimension by at most one. Therefore, to achieve cubic
  dimension $p$, $\Delta$ must have at least $p$ vertices, i.e.,
  dimension at least $p-1$, which proves \Cref{item:1}.
  
  Iterating the argument above shows that when $\Delta$ has dimension
  $p-1$, it is in the $\mathfrak{H}_p$-orbit of
  $[e_1,e_2,\dots,e_p]$. At this point, the only vertex that can be
  added to produce a simplex of higher dimension in $X^{p,2}$ is
  $(0^p)$, which shows \Cref{item:2}.

  The stabilizer of $[e_1,\dots,e_p]$ in $\mathfrak{H}_p$ is
  $\mathfrak{S}_p$. By the orbit-stabilizer theorem, the orbit of
  $[e_1,\dots,e_p]$ contains exactly
  $|\mathfrak{H}_p| / |\mathfrak{S}_p| = 2^p$ elements. Therefore,
  $C^{p,2}_{p-1,p}$ is a vector space of dimension $2^p$ with a basis
  given by the simplices in the $\mathfrak{H}_p$-orbit of
  $[e_1,\dots,e_p]$. A similar argument holds for
  $C^{p,2}_{p,p}$. Thus, the complex $C^{p,2}_{\bullet,p}$ has the
  form in \Cref{item:3}.

  The boundary map of $X^{p,2}$ operates by
  \begin{equation*}
    \partial_{p+1} \left( [0^p,e_1,\dots,e_p] \right) = [e_1,\dots,e_p]
    + \sum_{i=1}^p (-1)^i [0^p,e_1,\dots,e_{i-1},e_i,\dots,e_p],
  \end{equation*}
  and all faces $[0^p,e_1,\dots,e_{i-1},e_i,\dots,e_p]$ have cubic
  dimension $p-1$.  Hence, the differential in $C^{p,2}_{\bullet,p}$
  acts by
  \begin{equation*}
    d \left( [0^p,e_1,\dots,e_p] \right) = [e_1,\dots,e_p],
  \end{equation*}
  so it is an isomorphism. This shows that $C^{p,2}_{\bullet,p}$ has
  no homology.
\end{proof}

It follows from \Cref{lem:2} and \Cref{pro:1} that $E^1_{p,q} = 0$ for
all $p\geqslant 4$ and all $q$.

\subsection{The spectral sequence for scale two}
\label{sec:spectr-sequ-scale}

We start by summarizing the computations of the terms in the first
page $E^1_{\bullet,\bullet}$ of the spectral sequence from the
previous sections.  As depicted in \Cref{fig:4}, these terms are
$E^1_{i,0} \cong \mathcal{G}_i \oplus \mathcal{H}_i$ for all $i$, and
$E^1_{i,j} = 0$ when $j\neq 0$, where
\begin{equation*}
  \mathcal{G}_i = \operatorname{Ind}^{\mathfrak{H}_n}
  _{\mathfrak{S}_{n-i} \times \mathfrak{H}_i}
  (\{n-i\} \otimes \{0;1^i\}),
\end{equation*}
for $0\leqslant i \leqslant n$, and
\begin{equation*}
  \mathcal{H}_i =
  \begin{cases}
    \operatorname{Ind}^{\mathfrak{H}_n}_{\mathfrak{S}_{n-3} \times \mathfrak{H}_3}
    \left( \{n-3\} \otimes \{1^3;0\} \right),
    & i=3,\\
    0, & i\neq 3.
  \end{cases}
\end{equation*}

\begin{figure}[htb]
  \centering
  \begin{tikzpicture}[description/.style={fill=white,inner sep=2pt}]
    \matrix (m) [matrix of math nodes, row sep=1.2em,
    column sep=2em, text height=1.5ex, text depth=0.25ex]
    { 0 & 0 & 0 & 0 & 0\\
      \mathcal{G}_0 & \mathcal{G}_1 & \mathcal{G}_2 & \mathcal{G}_3 \oplus \mathcal{H}_3 & 0\\
      0 & 0 & 0 & 0 & 0\\ };
    \draw[<-] (m-1-1) edge (m-1-2) (m-1-2) edge (m-1-3) (m-1-3) edge (m-1-4) (m-1-4) edge (m-1-5);
    \draw[<-] (m-2-1) edge (m-2-2) (m-2-2) edge (m-2-3) (m-2-3) edge (m-2-4) (m-2-4) edge (m-2-5);
    \draw[<-] (m-3-1) edge (m-3-2) (m-3-2) edge (m-3-3) (m-3-3) edge (m-3-4) (m-3-4) edge (m-3-5);
  \end{tikzpicture}
  
  \caption{The page $E^1_{\bullet,\bullet}$ for scale two}
  \label{fig:4}
\end{figure}

We set out to understand the maps in the horizontal complex
$E^1_{\bullet,0}$.  We give explicit formulas for each differential
$d^1_{i,0} \colon E^1_{i,0} \to E^1_{i-1,0}$.  To start, notice that
each nonzero $\mathcal{G}_i$ and $\mathcal{H}_i$ is, by definition, a
representation of $\mathfrak{H}_n$ induced from a 1-dimensional
representation $U$ of $\mathfrak{S}_{n-i} \times \mathfrak{H}_i$. It
follows that any nonzero $u\in U$ generates $\mathcal{G}_i$ or
$\mathcal{H}_i$ as a $\Bbbk [\mathfrak{H}_n]$-module. Thus, the
differential $d^1_{i,0}$ is completely determined by its effect on a
generator $u$ of $\mathcal{G}_i$ and $\mathcal{H}_i$. We already
identified explicit generators in the previous sections, so all that
is left is to compute their image under the differential.

The term $\mathcal{G}_1$ is generated by the 1-simplex
$[0^n,0^{n-1}1]$, which we identify with $[0,1]$ after omitting the
$n-1$ leading zeros.  The differential in the spectral sequence
descends from the restriction of the boundary map in $X^{n,2}$, so we
have
\begin{equation*}
  d^1_{1,0} ([0,1]) = [1]-[0].
\end{equation*}
This is exactly the differential
$\partial_1\colon \mathcal{F}_1\to \mathcal{F}_0$ in the cellular
chain complex of the $n$-cube from \Cref{sec:second-page}.

As for $\mathcal{G}_2$, we showed in \Cref{sec:size-two} that it is
generated by $[00,01,10]-[01,10,11] = [00,01,11]-[00,10,11]$, after
omitting the $n-2$ leading zeros.  The differential in the spectral
sequence operates as follows:
\begin{equation*}
  \begin{split}
    d^1_{2,0} ([00,01,10]-[01,10,11]) &=
    [01,10]-[00,10]+[00,01] -[10,11]+[01,11]-[01,10]\\
    &=([00,01]-[10,11])-([00,10]-[01,11]).
  \end{split}
\end{equation*}
This map sends (a triangulation of) the square to its boundary; hence,
it is, up to a sign, the differential
$\partial_2\colon \mathcal{F}_2\to \mathcal{F}_1$ in the cellular
chain complex of the $n$-cube from \Cref{sec:second-page}.

As noted in \Cref{rem:1}, $\mathcal{G}_3$ is generated by the element
\begin{equation*}
  \begin{split}
    x-y ={} &(-[000,001,011,101] + [000,010,011,110] - [000,100,101,110]\\
    &+ [011,101,110,111] + [000,011,101,110] ) \\
    &-([000,001,010,100] - [001,010,011,111] + [001,100,101,111]\\
    &- [010,100,110,111] + [001,010,100,111]),
  \end{split}
\end{equation*}
written omitting the $n-3$ leading zeros. Each block in parentheses
represents a triangulation of the cube $[0,1]^3$. We apply the
differential in the spectral sequence, and simplify the output using
the relations in $\mathcal{G}_2$ to get
\begin{equation*}
  \begin{split}
    d^1_{3,0} (x-y) =
    &+2\big[ ([000,001,010]-[001,010,011]) - (-[100,101,110]+[101,110,111]) \big]\\
    &-2\big[ ([000,001,100]-[001,100,101]) - (-[010,011,110]+[011,110,111]) \big]\\
    &+2\big[ ([000,010,100]-[010,100,110]) - (-[001,011,101]+[011,101,111]) \big].
  \end{split}
\end{equation*}
The first set of parentheses in the top row contains a linear
combination of simplices representing a triangulation of the first
front face of the cube $[0,1]^3$. The second set of parentheses in the
top row contains a linear combination of simplices representing a
triangulation of the first back face of $[0,1]^3$. Similarly, for the
other rows. Thus, after replacing each 2-cell with a triangulation,
the map $\mathcal{G}_3\to \mathcal{G}_2$ is, up to a sign, twice the
differential $\partial_3\colon \mathcal{F}_3\to \mathcal{F}_2$ in the
cellular chain complex of the $n$-cube from \Cref{sec:second-page}.

Meanwhile, as noted in \Cref{rem:1}, $\mathcal{H}_3$ is generated by
the element
\begin{equation*}
  \begin{split}
    x+y ={} &(-[000,001,011,101] + [000,010,011,110] - [000,100,101,110]\\
    &+ [011,101,110,111] + [000,011,101,110])\\
    &+([000,001,010,100] - [001,010,011,111] + [001,100,101,111]\\
    &- [010,100,110,111] + [001,010,100,111]),
  \end{split}
\end{equation*}
written omitting the $n-3$ leading zeros. Applying the differential in
the spectral sequence and simplifying the output using the relations
in $\mathcal{G}_2$, we get $d^1_{3,0} (x+y) = 0$.

Altogether, these computations allow us to express $E^1_{\bullet,0}$
as a direct sum of subcomplexes
$\mathcal{G}_\bullet \oplus \mathcal{H}_\bullet$, where the maps in
$\mathcal{G}_\bullet$ are the restrictions of the differentials in the
spectral sequence, and the maps in $\mathcal{H}_\bullet$ are all
zero. Moreover, the maps in $\mathcal{G}_\bullet$ are, up to a nonzero
scalar multiple, the differentials in a truncation of the cellular
resolution of the $n$-cube, i.e., we have an isomorphism of complexes
$\mathcal{G}_\bullet \cong (\mathcal{F}_{\leqslant 3})_\bullet$ as
defined in \Cref{sec:second-page}. To summarize, we have shown that we
have an isomorphism
$E^1_{\bullet,0} \cong (\mathcal{F}_{\leqslant 3})_\bullet \oplus
\mathcal{H}_\bullet$ of complexes of $\mathfrak{H}_n$-representations.

\begin{remark}
  Using \cite[Lemma III.6, Theorem III.2]{zbMATH03590472}, we obtain
  the following decompositions into irreducible representations for
  $\mathcal{G}_2$ and $\mathcal{H}_3$.
  \begin{equation*}
    \begin{split}
      \mathcal{G}_2
      &\cong
        \left(
        \bigoplus_{j=0}^{n-2} \{n-2-j;(j+1,1)\}
        \right)
        \oplus
        \left(
        \bigoplus_{j=1}^{n-2} \{n-2-j;(j,1^2)\}
        \right)\\
      \mathcal{H}_3
      &\cong
        \left(
        \bigoplus_{j=0}^{n-3} \{(n-2-j,1^2);j\}
        \right)
        \oplus
        \left(
        \bigoplus_{j=0}^{n-4} \{(n-3-j,1^3);j\}
        \right)
    \end{split}
  \end{equation*}
  Because the decompositions of $\mathcal{H}_3$ and $\mathcal{G}_2$
  share no irreducible representation, Schur's Lemma guarantees that
  the only map of $\mathfrak{H}_n$-representations
  $\mathcal{H}_3\to \mathcal{G}_2$ is the zero map \cite[Proposition
  2.3.9]{zbMATH05943514}. This gives us the isomorphism of complexes
  $E^1_{\bullet,0} \cong \mathcal{G}_\bullet \oplus
  \mathcal{H}_\bullet$ without performing any computation. If one is
  willing to assume that $\Bbbk$ is algebraically closed, the
  isomorphism
  $\mathcal{G}_\bullet \cong (\mathcal{F}_{\leqslant 3})_\bullet$ can
  also be obtained without performing explicit computations as
  indicated in \Cref{sec:more-cell-compl}.
\end{remark}

Now, we can compute homology to advance to the
second page of the spectral sequence.  Since homology commutes with
direct sums, we get
\begin{equation*}
  E^2_{i,0} = H_i (E^1_{\bullet,0}) \cong
  H_i ((\mathcal{F}_{\leqslant 3})_\bullet) \oplus
  H_i (\mathcal{H}_\bullet) \cong
  H_i ((\mathcal{F}_{\leqslant 3})_\bullet) \oplus
  \mathcal{H}_i,
\end{equation*}
where $H_i (\mathcal{H}_\bullet) \cong \mathcal{H}_\bullet$ because
all the maps in the complex $\mathcal{H}_\bullet$ are zero.  Using
\Cref{pro:2}, we obtain the explicit descriptions
\begin{equation*}
  \begin{split}
    E^2_{0,0} &\cong \{n;0\},\\
    E^2_{3,0} &\cong
    \left( \bigoplus_{j=1}^{n-3} \{n-3-j; (j,1^3)\} \right) \oplus
    \operatorname{Ind}^{\mathfrak{H}_n}_{\mathfrak{S}_{n-3} \times \mathfrak{H}_3}
    \left( \{n-3\} \otimes \{1^3;0\} \right),
  \end{split}
\end{equation*}
as representations of $\mathfrak{H}_n$. For all other indices, we have
$E^2_{i,j}=0$.  Thus, the second page of the spectral sequence can be
depicted as in \Cref{fig:5}. All the maps in the second page of the
spectral sequence are zero, so the spectral sequence converges at this
page to give the homology of $X^{n,2}$. We summarize our computations
in the following theorem, where \cite[Lemma III.6, Theorem
III.2]{zbMATH03590472} are used to decompose into irreducible
representations.
\begin{figure}[htb]
  \centering
  \begin{tikzpicture}[description/.style={fill=white,inner sep=2pt}]
    \matrix (m) [matrix of math nodes, row sep=2em,
    column sep=2.5em, text height=1.5ex, text depth=0.25ex]
    { 0 & 0 & 0 & 0 & 0 & 0\\
      E^2_{0,0} & 0 & 0 & E^2_{3,0} & 0 &0\\
      0 & 0 & 0 & 0 & 0 & 0\\ };
    \draw[<-] (m-2-1) edge (m-3-3);
    \draw[<-] (m-1-2) edge (m-2-4) (m-2-4) edge (m-3-6);
  \end{tikzpicture}
  
  \caption{The page $E^2_{\bullet,\bullet}$ for scale two}
  \label{fig:5}
\end{figure}
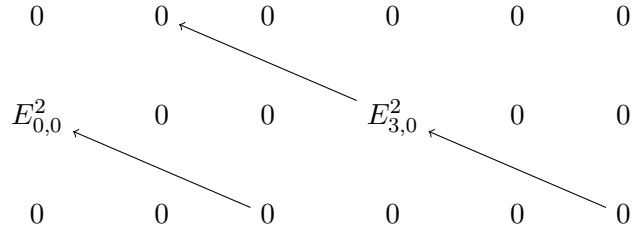

\begin{theorem}\label{thm:1}
  For every integer $n\geqslant 3$, the nonzero homology groups of
  $X^{n,2}$ are
  \begin{equation*}
    \begin{split}
      H_0 (X^{n,2}) &\cong \{n;0\},\\
      H_3 (X^{n,2}) &\cong
      \left( \bigoplus_{j=1}^{n-3} \{n-3-j; (j,1^3)\} \right) \oplus
      \operatorname{Ind}^{\mathfrak{H}_n}_{\mathfrak{S}_{n-3} \times \mathfrak{H}_3}
                      \left( \{n-3\} \otimes \{1^3;0\} \right)\\
      &\cong \left( \bigoplus_{j=1}^{n-3} \{n-3-j; (j,1^3)\} \right) \oplus
      \left(
      \bigoplus_{j=0}^{1}\;
      \bigoplus_{i=0}^{n-3-j}
      \{(n-2-i-j,1^{2+j});i\}
      \right)
    \end{split}
  \end{equation*}
  as representations of $\mathfrak{H}_n$.
\end{theorem}

\subsection{Dimension check}
\label{sec:dimension-check}

By \cite[Theorem 1]{zbMATH07553291}, we have
\begin{equation*}
  \dim_\Bbbk H_3 (X^{n,2}) =
  \sum_{0\leqslant j<i<n} (j+1)(2^{n-2}-2^{i-1})=
  \sum_{i=1}^{n-1} (2^{n-2}-2^{i-1}) \binom{i+1}{2}.
\end{equation*}
Separating the difference we get
\begin{equation*}
  \sum_{i=1}^{n-1} 2^{n-2} \binom{i+1}{2} =
  2^{n-3} \sum_{i=1}^{n-1} (i^2+i) = 2^{n-3} \frac{n^3-n}{3}.
\end{equation*}
and
\begin{equation*}
  \sum_{i=1}^{n-1} 2^{i-1} \binom{i+1}{2} =
  \frac{1}{2} \sum_{i=1}^{n-1} 2^{i-1} (i^2+i).
\end{equation*}
Evaluating the identity 
\begin{equation*}
  \frac{d^2}{dx^2} \left( x \frac{1-x^n}{1-x} \right) =
  \sum_{i=0}^{n-1} (i^2+i) x^{i-1}
\end{equation*}
at $x=2$, we obtain
\begin{equation*}
  \sum_{i=1}^{n-1} (i^2+i) 2^{i-1} =
  2^{n-1} n^2-3n 2^{n-1} +2^{n+1} -2.
\end{equation*}
Combining the previous equalities, we have
\begin{equation}\label{eq:19}
  \dim_\Bbbk H_3 (X^{n,2})
  = \frac{2^{n-3}}{3} (n^3-6n^2+17n-24)+1.
\end{equation}
Next, we compute the dimension of $H_3 (X^{n,2})$ using
\Cref{thm:1}. By \Cref{eq:29}, the dimension of
$\bigoplus_{j=1}^{n-3} \{n-3-j; (j,1^3)\}$ is
\begin{equation*}
  \sum_{j=1}^{n-3} \binom{n}{j+3} \binom{j+2}{3} =
  \sum_{k=4}^n \binom{n}{k} \binom{k-1}{3} =
  \sum_{k=4}^n \binom{n}{k} \frac{k^3 -6k^2+11k-6}{6}.
\end{equation*}
The last expression is equal to
\begin{multline}\label{eq:20}
  \frac{1}{6} \left(
    \sum_{k=0}^n k^3 \binom{n}{k}
    -6 \sum_{k=0}^n k^2 \binom{n}{k}
    +11 \sum_{k=0}^n k \binom{n}{k}
    -6 \sum_{k=0}^n \binom{n}{k}
  \right) + 1 =\\ 	
  \frac{2^{n-4}}{3} (n^3 -9n^2+32n-48) +1.
\end{multline}
The other direct summand in $H_3 (X^{n,2})$ is
$ \operatorname{Ind}^{\mathfrak{H}_n}_{\mathfrak{S}_{n-3} \times
  \mathfrak{H}_3} \left( \{n-3\} \otimes \{1^3;0\} \right) $. As
$ \{n-3\} \otimes \{1^3;0\}$ is one-dimensional, the dimension of the
induced representation is
\begin{equation}\label{eq:21}
	2^{n-3} \binom{n}{3} = 2^{n-3} \frac{n(n-1)(n-2)}{6}.
\end{equation}
Adding the expressions in \Cref{eq:20,eq:21} we recover the right
hand side of \Cref{eq:19}, as desired.

\section{Scale three}
\label{sec:scale-three}

We use the strategy from \Cref{sec:suppl-filtr,sec:scale-two} to
describe the homology of $X^{n,3}$ as a representation of the
hyperoctahedral group, separating the computations by cubic dimension.

If $\Delta$ in $X^{n,3}$ is a simplex of cubic dimension at most two,
then $\Delta$ can be embedded in $Q_2$, which forces the distance
between any two vertices of $\Delta$ to be at most two. This implies
that $\Delta$ is also a simplex in $X^{n,2}$. As a result, for
$0\leqslant p\leqslant 2$, we have an isomorphism
$C^{p,3}_{\bullet, p} \cong C^{p,2}_{\bullet, p}$ of complexes of
$\mathfrak{H}_p$-representations. Thus, in the zeroth page of the
spectral sequence for $H_\bullet (X^{n,3})$, the columns
$E^0_{p,\bullet}$ with $0\leqslant p\leqslant 2$ and their homology
have exactly the same description already given in
\Cref{sec:size-zero-one,sec:size-two}.

\subsection{Cubic dimension three and four}
\label{sec:cubic-dimens-three}

Starting with cubic dimension three, we encounter simplices that were
not present in $X^{n,2}$. For example, the diagonals of cubes in
$X^{n,3}$, i.e., the simplices in the $\mathfrak{H}_n$-orbit of
$[0^n,0^{n-3} 1^3]$, are new.  A complete characterization of the
chain complexes $C^{p,3}_{\bullet,p}$ for cubic dimension
$p\geqslant 3$ becomes combinatorially challenging. Instead, we wrote
a Macaulay2 package to construct these complexes with the aid of
software (see \Cref{sec:macaulay2-package}).

In a running instance of Macaulay2, we issue the commands
\begin{verbatim}
needsPackage "VietorisRipsHypercube"
X53=VRHypercube(5,3)
\end{verbatim}
to create a computational representation for the Vietoris-Rips complex
$X^{5,3}$. The chain complex $C^{3,3}_{\bullet,3}$ and its homology
are obtained with the following commands.
\begin{verbatim}
complex(X53,3)
homology(X53,3)
\end{verbatim}
The complex has the form
\begin{equation*}
  0\longleftarrow
  \Bbbk^4\longleftarrow
  \Bbbk^{32}\longleftarrow
  \Bbbk^{64}\longleftarrow
  \Bbbk^{56}\longleftarrow
  \Bbbk^{28}\longleftarrow
  \Bbbk^8\longleftarrow
  \Bbbk^1\longleftarrow
  0
\end{equation*}
with the leftmost zero term corresponding to 0-simplices, and its only
nonzero homology group is $H_3 (C^{3,3}_{\bullet,3}) \cong
\Bbbk$. Since $H_3 (C^{3,3}_{\bullet,3})$ is 1-dimensional, it is an
irreducible representation of $\mathfrak{H}_3$. To find out its
isomorphism type, we compute its character and decompose it against
the character table of $\mathfrak{H}_3$ with the following commands.
\begin{verbatim}
c33 = character(X53,3,3)
T3 = hyperoctahedralGroupTable(3,QQ)
c33 / T3
\end{verbatim}
As a result, we find $H_3 (C^{3,3}_{\bullet,3}) \cong \{0;1^3\}$.

To understand the behavior in cubic dimension four, we proceed similarly.
\begin{verbatim}
complex(X53,4)
homology(X53,4)
\end{verbatim}
The complex $C^{4,3}_{\bullet,4}$ has the form
\begin{equation*}
  0\longleftarrow
  \Bbbk^{96}\longleftarrow
  \Bbbk^{584}\longleftarrow
  \Bbbk^{1344}\longleftarrow
  \Bbbk^{1568}\longleftarrow
  \Bbbk^{960}\longleftarrow
  \Bbbk^{248}\longleftarrow
  0
\end{equation*}
with the leftmost zero term corresponding to 1-simplices, but there are now
two nonzero homology groups, $H_4 (C^{4,3}_{\bullet,4})$ and
$H_7 (C^{4,3}_{\bullet,4})$, both 1-dimensional.  We compute their
characters and decompose them against the character table of
$\mathfrak{H}_4$.
\begin{verbatim}
c44 = character(X53,4,4)
c47 = character(X53,4,7)
T4 = hyperoctahedralGroupTable(4,QQ)
c44 / T4
c47 / T4
\end{verbatim}
As a result, we find that $H_4 (C^{4,3}_{\bullet,4}) \cong \{0;1^4\}$
and $H_7 (C^{4,3}_{\bullet,4}) \cong \{4;0\}$.

\subsection{Cubic dimension five, six, and seven}
\label{sec:cubic-dimension-five}

Proceeding as in the previous section, we find that the chain
complexes of simplices with cubic dimension five, six, and seven have
no homology. For reference, we record these complexes below.

\begin{itemize}
\item The complex $C^{5,3}_{\bullet,5}$ has the form
  \begin{equation*}
    0\longleftarrow
    \Bbbk^{800}\longleftarrow
    \Bbbk^{4192}\longleftarrow
    \Bbbk^{8512}\longleftarrow
    \Bbbk^{8640}\longleftarrow
    \Bbbk^{4560}\longleftarrow
    \Bbbk^{1120}\longleftarrow
    \Bbbk^{80}\longleftarrow
    0  
  \end{equation*}
  with the leftmost zero term corresponding to 2-simplices.
\item The complex $C^{6,3}_{\bullet,6}$ has the form
  \begin{equation*}
    0\longleftarrow
    \Bbbk^{4800}\longleftarrow
    \Bbbk^{25344}\longleftarrow
    \Bbbk^{54144}\longleftarrow
    \Bbbk^{60480}\longleftarrow
    \Bbbk^{37760}\longleftarrow
    \Bbbk^{12992}\longleftarrow
    \Bbbk^{2304}\longleftarrow
    \Bbbk^{192}\longleftarrow
    0  
  \end{equation*}
  with the leftmost zero term corresponding to 3-simplices.
\item The complex $C^{7,3}_{\bullet,7}$ has the form
  \begin{multline*}
    0\longleftarrow
    \Bbbk^{25088}\longleftarrow
    \Bbbk^{138112}\longleftarrow
    \Bbbk^{317312}\longleftarrow
    \Bbbk^{396928}\longleftarrow
    \Bbbk^{295680}\longleftarrow\\
    \longleftarrow
    \Bbbk^{135296}\longleftarrow
    \Bbbk^{38080}\longleftarrow
    \Bbbk^{6272}\longleftarrow
    \Bbbk^{448}\longleftarrow
    0  
  \end{multline*}
  with the leftmost zero term corresponding to 4-simplices.
\end{itemize}

\subsection{Cubic dimension eight and nine}
\label{sec:cubic-dimens-eight}

Due to the limit of our computational resources, the approach of the
previous sections did not work for cubic dimension eight or nine. To
overcome this issue we need the following theoretical result.

\begin{lemma}\label{pro:3}
  Define the length $p$ binary sequences $z=(0000\dots 0)$,
  $u=(1110\dots 0)$, and for integers $i$ with
  $4\leqslant i\leqslant p$:
  \begin{equation*}
    v_i=(1100\dots 010\dots 0),\qquad
    w_i=(1000\dots 010\dots 0),
  \end{equation*}
  where the last 1 appears in the $i$-th entry.
  \begin{enumerate}[label=(\arabic*)]
  \item\label{item:5} The smallest dimension of a simplex of cubic
    dimension $p\geqslant 5$ in $X^{p,3}$ is $p-2$.
  \item\label{item:6} The simplices
    $[z,u,v_4,\dots,v_k,w_{k+1},\dots,w_p]$ with
    $\lceil \frac{p+3}{2}\rceil \leqslant k\leqslant p$ form a
    complete set of representatives for the orbits of $\mathfrak{H}_p$
    acting on the set of simplices of cubic dimension $p\geqslant 5$
    and dimension $p-2$ in $X^{p,3}$.
  \item\label{item:7} For every $p\geqslant 5$, we have
    $H_{p-2} (C^{p,3}_{\bullet,p}) = 0$.
  \item\label{item:8} For every $p\geqslant 5$ and
    $\lceil \frac{p+3}{2}\rceil \leqslant k\leqslant p$, the
    stabilizer of $[z,u,v_4,\dots,v_k,w_{k+1},\dots,w_p]$ in
    $\mathfrak{H}_p$ (up to orientation) contains
    $\mathfrak{S}_{\{3,\dots,k\}} \times
    \mathfrak{S}_{\{k+1,\dots,p\}}$, where $\mathfrak{S}_I$ denotes
    the symmetric group on the elements of $I\subseteq \{1,\dots,p\}$.
  \end{enumerate}
\end{lemma}

\begin{proof}
  Let $\Delta$ be a simplex of cubic dimension $p\geqslant 5$ in
  $X^{n,3}$. If $\Delta$ has diameter 2, then it is also a simplex in
  $X^{n,2}$, so \Cref{lem:2} implies that $\Delta$ has dimension at
  least $p-1$.

  Next, suppose $\Delta$ has a pair of vertices at distance three. Up
  to the action of $\mathfrak{H}_p$, we can assume these vertices are
  $z$ and $u$, so $\Delta$ has cubic dimension at least 3 and
  dimension at least 1. To increase its cubic dimension, $\Delta$ must
  contain a vertex $x=(x_1\dots x_p)$ having at least one entry
  $x_i= 1$ with $i>3$. In order to have both $d(z,x) \leqslant 3$ and
  $d(u,x) \leqslant 3$, $x$ must have a single entry $x_i=1$ with
  $i>3$, and either two or one entries $x_i=1$ with $i\leqslant 3$. It
  follows that each vertex in $\Delta$ beyond $z$ and $u$ can only
  increase cubic dimension by at most one. Therefore, to achieve cubic
  dimension $p$, $\Delta$ must have at least $p-1$ vertices, i.e.,
  dimension at least $p-2$. This proves \Cref{item:5}.

  Now, assume $\Delta$ has the smallest possible dimension, i.e.,
  $p-2$. Consider two vertices $x=(x_1\dots x_p)$ and
  $y=(y_1\dots y_p)$ of $\Delta$, other than $z$ and $u$. Reasoning as
  above, $x$ and $y$ both contain a single entry $x_i=y_j=1$ with
  $i>3$ and $j>3$. To increase cubic dimension while minimizing
  dimension, we must have $i\neq j$, which implies
  $d(x,y) \geqslant 2$. This means that the first three entries of $x$
  and $y$ differ in at most one entry. As noted in the previous
  paragraph, these first three entries must contain at least a 0 and a
  1. Therefore, up to a permutation, we can assume that the first
  three entries of $x$ and $y$ are either 110 or 100. In other words,
  $x$ is equal to $v_i$ or $w_i$ for some integer $i$ with
  $4\leqslant i\leqslant p$, and similarly for $y$. Reordering the
  vertices, we can ensure that $\Delta$ is in the
  $\mathfrak{H}_p$-orbit of
  \begin{equation*}
    [z,u, v_{i_1}, \dots, v_{i_s}, w_{j_1}, \dots, w_{j_t}].
  \end{equation*}
  Since the cubic dimension of $\Delta$ is $p$, we must have
  $\{i_1,\dots,i_s,j_1,\dots,j_t\} = \{4,\dots,p\}$. If $t>s$, then
  acting with $(1110\dots 0,(13)) \in \mathfrak{H}_p$ swaps $z$ with
  $u$, and $v_k$ with $w_k$. Thus, we can always assume that
  $s\geqslant t$, i.e., that there are more $v$ vertices than $w$
  vertices. This is equivalent to
  \begin{equation*}
    s\geqslant
    \frac{1}{2} \left|\{i_1,\dots,i_s,j_1,\dots,j_t\}\right| =
    \frac{1}{2} \left|\{4,\dots,p\}\right| =
    \frac{1}{2} (p-3).
  \end{equation*}
  Finally, permuting the last $p-3$ entries, we see that $\Delta$ is
  in the $\mathfrak{H}_p$-orbit of
  \begin{equation*}
    [z,u, v_4, \dots, v_k, w_{k+1}, \dots, w_p]
  \end{equation*}
  where $k=s+3$. Therefore, we have $k\geqslant \frac{1}{2} (p-3) + 3$
  and, since $k$ is an integer, this is equivalent to
  $k\geqslant \lceil \frac{p+3}{2}\rceil$.

  It remains to prove that the simplices
  $[z,u, v_4, \dots, v_k, w_{k+1}, \dots, w_p]$ belong to different
  $\mathfrak{H}_p$-orbits for different values of $k$. Since the
  $\mathfrak{H}_p$-action preserves the Hamming distance, the number
  of pairs of vertices at distance two is constant across each
  orbit. We proceed to count the number of pairs of vertices of
  $[z,u, v_4, \dots, v_k, w_{k+1}, \dots, w_p]$ that have distance
  two, and show that it is different for each value of $k$.  For
  indices $i\neq j$, we have $d(v_i,v_j)=2$. Extending our notation,
  we can think of $u$ as $v_3$, which means we have $\binom{k-2}{2}$
  pairs $(v_i,v_j)$. We also have $d(z,w_i)=2$, and there are $p-k$
  pairs $(z,w_i)$. There are no more pairs of vertices that have
  distance two because:
  \begin{itemize}
  \item $d(z,u)=3$;
  \item $d(z,v_i)=3$ for all $i$;
  \item $d(u,w_j)=3$ for all $j$;
  \item and $d(v_i,w_j) = 3$ for all $i\neq j$.
  \end{itemize}
  Thus, the number of pairs of vertices at distance two in
  $[z,u, v_4, \dots, v_k, w_{k+1}, \dots, w_p]$ is
  \begin{equation*}
    \binom{k-2}{2} + p-k = \frac{1}{2} (k^2-7k) +p+3.
  \end{equation*}
  As a function of $k$, this has derivative $\frac{1}{2}(2k-7)$. Given
  that $p\geqslant 5$, we have
  $k\geqslant \lceil \frac{p+3}{2}\rceil \geqslant 4$, which makes the
  function strictly increasing for $p\geqslant 5$.  This concludes the
  proof of \Cref{item:6}.

  It follows from \Cref{item:5} that, for every $p\geqslant 5$,
  $C^{p,3}_{i,p} = 0$ for all $i<p-2$. Thus, the complex
  $C^{p,3}_{\bullet,p}$ has the following form.
  \begin{equation*}
    \cdots \longleftarrow 0 \longleftarrow C^{p,3}_{p-2,p}
    \xleftarrow{\ d\ } C^{p,3}_{p-1,p} \longleftarrow \cdots
  \end{equation*}
  The $\Bbbk$-vector space $C^{p,3}_{p-2,p}$ is spanned by the
  simplices of cubic dimension $p$ and dimension $p-2$ in
  $X^{p,3}$. Let $\Delta$ be one such simplex.  By \Cref{item:6}, up
  to the action of $\mathfrak{H}_p$, we have
  \begin{equation*}
    \Delta = [z,u, v_4, \dots, v_k, w_{k+1}, \dots, w_p]
  \end{equation*}
  for some integer $k$ with
  $\lceil \frac{p+3}{2}\rceil \leqslant k\leqslant p$. Consider the
  length $p$ binary sequence $e_1=(10\dots 0)$. We have $d(e_1,z)=1$,
  $d(e_1,u)=2$, and for $4\leqslant i\leqslant p$, $d(e_1,v_i)=2$, and
  $d(e_1,w_i)=1$. Therefore,
  \begin{equation*}
    \Gamma = [e_1,z,u, v_4, \dots, v_k, w_{k+1}, \dots, w_p]
  \end{equation*}
  is a simplex of cubic dimension $p$ and dimension $p-1$ in
  $X^{p,3}$. To understand the effects of the boundary map on
  $\Gamma$, observe the following.
  \begin{itemize}
  \item Removing $e_1$ from $\Gamma$ gives $\Delta$.
  \item Removing $z$ leaves only vertices whose first entry is 1.
  \item Removing $u$ leaves only vertices whose third entry is 0.
  \item Removing $v_i$ or $w_i$ leaves only vertices whose $i$-th entry is 0.
  \end{itemize}
  In other words, removing from $\Gamma$ any vertex other than $e_1$
  produces a simplex of cubic dimension strictly less than $p$. It
  follows that $d(\Gamma) = \Delta$. We conclude that $d$ is
  surjective, which proves \Cref{item:7}.

  To prove \Cref{item:8}, it is enough to observe that:
  \begin{itemize}
  \item the elements of $\mathfrak{S}_{\{3,\dots,k\}}$ fix the
    vertices $z,u,w_{k+1},\dots,w_p$ and permute $v_4,\dots,v_k$;
  \item the elements of $\mathfrak{S}_{\{k+1,\dots,p\}}$ fix the
    vertices $z,u,v_4,\dots,v_k$ and permute $w_{k+1},\dots,w_p$.
  \end{itemize}
  Therefore,
  $\mathfrak{S}_{\{3,\dots,k\}} \times \mathfrak{S}_{\{k+1,\dots,p\}}$
  stabilizes $[z,u,v_4,\dots,v_k,w_{k+1},\dots,w_p]$ up to
  orientation.
\end{proof}

Combining \Cref{pro:3} with the computational approach described in
\Cref{sec:python-code}, we obtain the following partial description
for the complex $C^{8,3}_{\bullet,8}$:
\begin{equation*}
  0\longleftarrow
  \begin{matrix}\Bbbk^{121856}\\\oplus\\0\end{matrix}
  \xleftarrow{\,d_1\,}
  \begin{matrix}\Bbbk^{709632}\\\oplus\\\Bbbk^{256}\end{matrix}
  \xleftarrow{\,d_2\,}
  \begin{matrix}\Bbbk^{1770496}\\\oplus\\\Bbbk^{256}\end{matrix}
  \xleftarrow{\,d_3\,}
  \begin{matrix}\Bbbk^{2487296}\\\oplus\\0\end{matrix}
  \longleftarrow
  \dots \longleftarrow
\end{equation*}
where the leftmost zero term corresponds to 5-simplices. Each term is a
direct sum of subspaces. The top summand is spanned by simplices with
diameter 3, which are found using the process described in
\Cref{sec:python-code}. The bottom summand is spanned by simplices
with diameter 2; note that these simplices exist in $X^{n,2}$ and are
completely determined in \Cref{lem:2}.  Let $d'_i$ be the restriction
of $d_i$ to the top summands, and let $d''_i$ be the restriction of
$d_i$ to the bottom summands.  Then, we have the block decompositions
\begin{equation*}
  d_1 =
  \begin{bmatrix}
    d'_1 & 0
  \end{bmatrix},\qquad
  d_2 =
  \begin{bmatrix}
    d'_2 & 0\\
    * & d''_2
  \end{bmatrix},\qquad
  d_3 =
  \begin{bmatrix}
    d'_3\\0
  \end{bmatrix}.
\end{equation*}
Although we compute homology with coefficients in $\Bbbk$, it is
convenient to think of the matrices $d_i,d'_i,d''_i$ as having integer
entries, so we can apply the functor $- \otimes_{\mathbb{Z}} R$ to
work over a different ring $R$.  By \Cref{pro:3}, we know
$H_{6} (C^{8,3}_{\bullet,8}) = 0$, so $d_1$ is surjective. Since
$\operatorname{im} (d_2) \subseteq \ker (d_1)$, we have
\begin{equation*}
  \dim_\Bbbk \operatorname{im} (d_2 \otimes_{\mathbb{Z}} \Bbbk) \leqslant
  \dim_\Bbbk \ker (d_1 \otimes_{\mathbb{Z}} \Bbbk) = 588032
\end{equation*}
by the Rank-Nullity Theorem. The methods of \Cref{sec:python-code}
show that $d'_2 \otimes_{\mathbb{Z}} \mathbb{Z}/2$ has rank 587776. By
\Cref{lem:2}, $d''_2$ is an isomorphism, so
$d''_2 \otimes_{\mathbb{Z}} \mathbb{Z}/2$ has rank 256. Therefore, we
have
\begin{equation*}
  \dim_\Bbbk \operatorname{im} (d_2 \otimes_{\mathbb{Z}} \Bbbk) \geqslant
  \dim_\Bbbk \operatorname{im} (d_2 \otimes_{\mathbb{Z}} \mathbb{Z}/2) 
  = 588032,
\end{equation*}
since the rank can only decrease when working modulo a prime.  We
conclude that $\operatorname{im} (d_2) = \ker (d_1)$ and
$H_{7} (C^{8,3}_{\bullet,8}) = 0$. Similarly, we have
$\operatorname{im} (d_3) \subseteq \ker (d_2)$ and
\begin{equation*}
  \dim_\Bbbk \operatorname{im} (d_3 \otimes_{\mathbb{Z}} \Bbbk) \leqslant
  \dim_\Bbbk \ker (d_2 \otimes_{\mathbb{Z}} \Bbbk) = 1182720.
\end{equation*}
Computationally, we find that $d'_3 \otimes_{\mathbb{Z}} \mathbb{Z}/2$
has rank 1182720, so
\begin{equation*}
  \dim_\Bbbk \operatorname{im} (d_3 \otimes_{\mathbb{Z}} \Bbbk) \geqslant
  \dim_\Bbbk \operatorname{im} (d_3 \otimes_{\mathbb{Z}} \mathbb{Z}/2) 
  = 1182720.
\end{equation*}
We conclude that $\operatorname{im} (d_3) = \ker (d_3)$ and
$H_{8} (C^{8,3}_{\bullet,8}) = 0$.

Next, using the same notation as before along with the methods in
\Cref{sec:python-code}, we find the following partial description for
the complex $C^{9,3}_{\bullet,9}$:
\begin{equation*}
  0\longleftarrow
  \begin{matrix}\Bbbk^{566784}\\\oplus\\0\end{matrix}
  \xleftarrow{\,d_1\,}
  \begin{matrix}\Bbbk^{3517440}\\\oplus\\\Bbbk^{512}\end{matrix}
  \xleftarrow{\,d_2\,}
  \begin{matrix}\Bbbk^{9549312}\\\oplus\\\Bbbk^{512}\end{matrix}
  \dots \longleftarrow
\end{equation*}
where the leftmost zero term corresponds to 6-simplices, and the
differentials decompose as:
\begin{equation*}
  d_1 =
  \begin{bmatrix}
    d'_1 & 0
  \end{bmatrix},\qquad
  d_2 =
  \begin{bmatrix}
    d'_2 & 0\\
    * & d''_2
  \end{bmatrix}.
\end{equation*}
By \Cref{pro:3}, we know $H_{7} (C^{9,3}_{\bullet,9}) = 0$, so $d_1$
is surjective. Since $\operatorname{im} (d_2) \subseteq \ker (d_1)$,
we have
\begin{equation*}
  \dim_\Bbbk \operatorname{im} (d_2 \otimes_{\mathbb{Z}} \Bbbk) \leqslant
  \dim_\Bbbk \ker (d_1 \otimes_{\mathbb{Z}} \Bbbk) = 2951168
\end{equation*}
The methods of \Cref{sec:python-code} show that
$d'_2 \otimes_{\mathbb{Z}} \mathbb{Z}/2$ has rank 2950656, and
$d''_2 \otimes_{\mathbb{Z}} \mathbb{Z}/2$ has rank 512 by
\Cref{lem:2}. This implies that
\begin{equation*}
  \dim_\Bbbk \operatorname{im} (d_2 \otimes_{\mathbb{Z}} \Bbbk) \geqslant
  \dim_\Bbbk \operatorname{im} (d_2 \otimes_{\mathbb{Z}} \mathbb{Z}/2) 
  = 2951168.
\end{equation*}
Therefore, $\operatorname{im} (d_2) = \ker (d_1)$ and
$H_{8} (C^{9,3}_{\bullet,9}) = 0$.

\subsection{The spectral sequence for scale three}
\label{sec:first-page-diff}

We start by collecting the information from the previous sections, and
assemble the spectral sequence described in \Cref{sec:suppl-filtr} for
$X^{n,3}$. As a reminder, $E^0_{p,q} = C^{n,r}_{p+q,p}$. \Cref{fig:6}
depicts the vertical complexes in the zeroth page for cubic dimensions
0 through 9; we use question marks for terms we did not describe or
compute explicitly.
\begin{figure}[htb]
  \centering
  \begin{tikzpicture}[description/.style={fill=white,inner sep=2pt}]
    \matrix (m) [matrix of math nodes, row sep=1.2em,
    column sep=1.5em, text height=1.5ex, text depth=0.25ex]
    { 0 & 0 & 0 & 0 & 0 & 0 & 0 & 0 & ? & ?\\
      0 & 0 & 0 & 0 & 0 & 0 & 0 & C^{n,3}_{13,7} & ? & ?\\
      0 & 0 & 0 & 0 & 0 & 0 &  C^{n,3}_{11,6} & C^{n,3}_{12,7} & ? & ?\\
      0 & 0 & 0 & C^{n,3}_{7,3} & 0 & C^{n,3}_{9,5} & C^{n,3}_{10,6} & C^{n,3}_{11,7} & ? & ?\\
      0 & 0 & 0 & C^{n,3}_{6,3} & C^{n,3}_{7,4} & C^{n,3}_{8,5} & C^{n,3}_{9,6} & C^{n,3}_{10,7} & ? & ?\\
      0 & 0 & 0 & C^{n,3}_{5,3} & C^{n,3}_{6,4} & C^{n,3}_{7,5} & C^{n,3}_{8,6} & C^{n,3}_{9,7} & ? & ?\\
      0 & 0 & C^{n,3}_{3,2} & C^{n,3}_{4,3} & C^{n,3}_{5,4} & C^{n,3}_{6,5} & C^{n,3}_{7,6} & C^{n,3}_{8,7} & C^{n,3}_{9,8} & ?\\
      C^{n,3}_{0,0} & C^{n,3}_{1,1} & C^{n,3}_{2,2} & C^{n,3}_{3,3} & C^{n,3}_{4,4} & C^{n,3}_{5,5} & C^{n,3}_{6,6} & C^{n,3}_{7,7} & C^{n,3}_{8,8} & C^{n,3}_{9,9}\\
      0 & 0 & C^{n,3}_{1,2} & C^{n,3}_{2,3} & C^{n,3}_{3,4} & C^{n,3}_{4,5} & C^{n,3}_{5,6} & C^{n,3}_{6,7} & C^{n,3}_{7,8} & C^{n,3}_{8,9}\\
      0 & 0 & 0 & C^{n,3}_{1,3} & C^{n,3}_{2,4} & C^{n,3}_{3,5} & C^{n,3}_{4,6} & C^{n,3}_{5,7} & C^{n,3}_{6,8} & C^{n,3}_{7,9}\\
      0 & 0 & 0 & 0 & 0 & 0 & 0 & 0 & 0 & 0\\ };
    \foreach \j in {1,2,3,4,5,6,7,8,9,10}
    \draw[->] (m-1-\j) edge (m-2-\j) (m-2-\j) edge (m-3-\j) (m-3-\j) edge (m-4-\j) (m-4-\j) edge (m-5-\j) (m-5-\j) edge (m-6-\j) (m-6-\j) edge (m-7-\j) (m-7-\j) edge (m-8-\j) (m-8-\j) edge (m-9-\j) (m-9-\j) edge (m-10-\j) (m-10-\j) edge (m-11-\j);
  \end{tikzpicture}
  
  \caption{The page $E^0_{\bullet,\bullet}$ for scale three}
  \label{fig:6}
\end{figure}
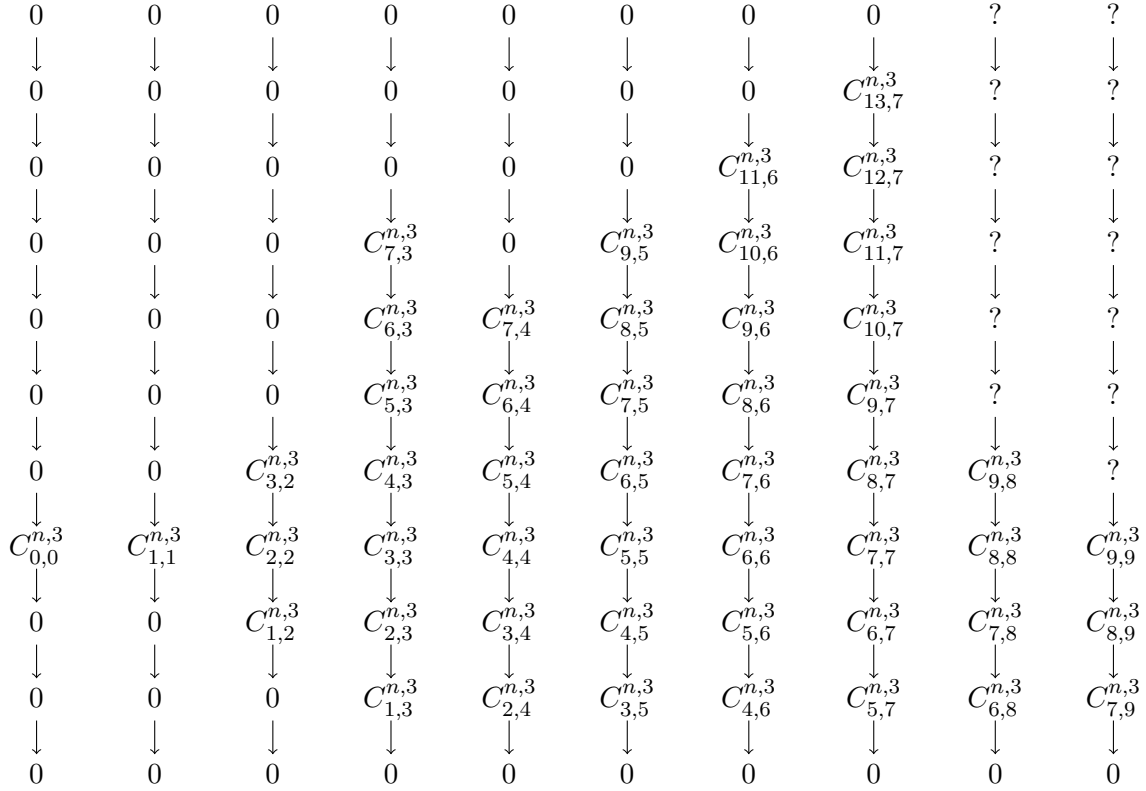

The first page, depicted in \Cref{fig:7}, is obtained by taking
homology of the vertical complexes in the zeroth page.
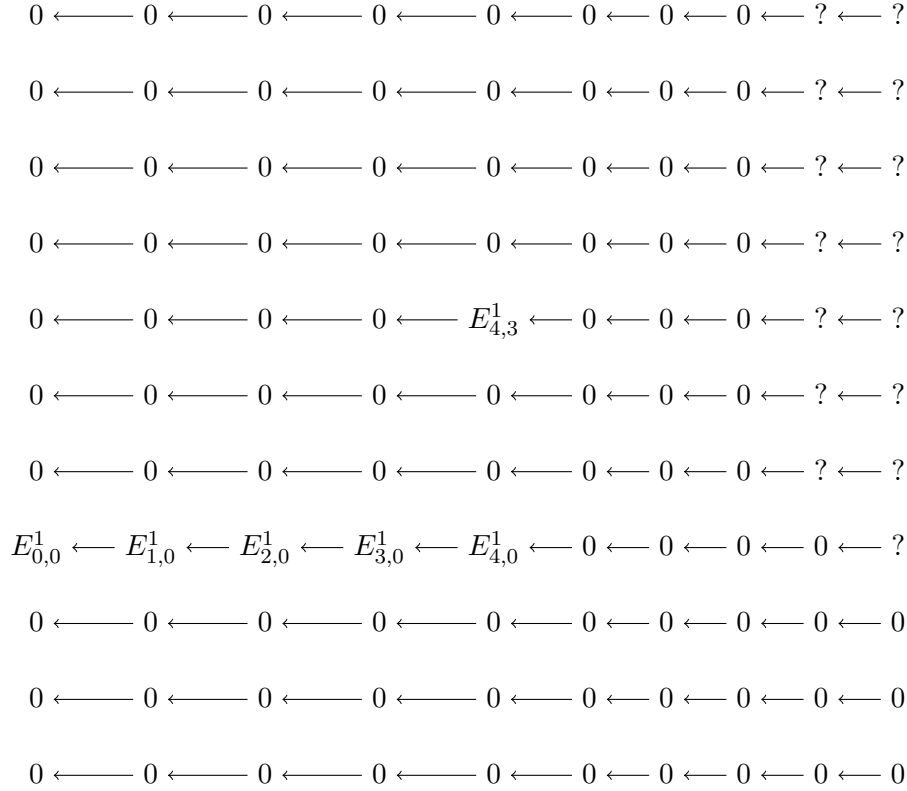
\begin{figure}[htb]
  \centering
  \begin{tikzpicture}[description/.style={fill=white,inner sep=2pt}]
    \matrix (m) [matrix of math nodes, row sep=1.2em,
    column sep=1.5em, text height=1.5ex, text depth=0.25ex]
    { 0 & 0 & 0 & 0 & 0 & 0 & 0 & 0 & ? & ?\\
      0 & 0 & 0 & 0 & 0 & 0 & 0 & 0 & ? & ?\\
      0 & 0 & 0 & 0 & 0 & 0 & 0 & 0 & ? & ?\\
      0 & 0 & 0 & 0 & 0 & 0 & 0 & 0 & ? & ?\\
      0 & 0 & 0 & 0 & E^1_{4,3} & 0 & 0 & 0 & ? & ?\\
      0 & 0 & 0 & 0 & 0 & 0 & 0 & 0 & ? & ?\\
      0 & 0 & 0 & 0 & 0 & 0 & 0 & 0 & ? & ?\\
      E^1_{0,0} & E^1_{1,0} & E^1_{2,0} & E^1_{3,0} & E^1_{4,0} & 0 & 0 & 0 & 0 & ?\\
      0 & 0 & 0 & 0 & 0 & 0 & 0 & 0 & 0 & 0\\
      0 & 0 & 0 & 0 & 0 & 0 & 0 & 0 & 0 & 0\\
      0 & 0 & 0 & 0 & 0 & 0 & 0 & 0 & 0 & 0\\ };
    \foreach \j in {1,2,3,4,5,6,7,8,9,10,11}
    \draw[<-] (m-\j-1) edge (m-\j-2) (m-\j-2) edge (m-\j-3) (m-\j-3) edge (m-\j-4) (m-\j-4) edge (m-\j-5) (m-\j-5) edge (m-\j-6) (m-\j-6) edge (m-\j-7) (m-\j-7) edge (m-\j-8) (m-\j-8) edge (m-\j-9) (m-\j-9) edge (m-\j-10);
  \end{tikzpicture}
  
  \caption{The page $E^1_{\bullet,\bullet}$ for scale three}
  \label{fig:7}
\end{figure}

Next, we identify the horizontal complex $E^1_{\bullet,0}$ with
$(\mathcal{F}_{\geqslant 4} )_\bullet$, the cellular chain complex of
the 4-skeleton of the $n$-cube described in
\Cref{sec:second-page}. The first two differentials already match,
since they behave exactly as the corresponding maps at scale two (see
the computations in \Cref{sec:spectr-sequ-scale}).  By \Cref{pro:1},
we have
\begin{equation*}
  E^1_{3,0} \cong 
  \operatorname{Ind}^{\mathfrak{H}_n}_{\mathfrak{S}_{n-3} \times \mathfrak{H}_3}
  \left( \{n-3\} \otimes H_3 (C^{3,3}_{\bullet,3}) \right) \cong
  \operatorname{Ind}^{\mathfrak{H}_n}_{\mathfrak{S}_{n-3} \times \mathfrak{H}_3}
  \left( \{n-3\} \otimes \{0;1^3\} \right)
\end{equation*}
where the isomorphism $H_3 (C^{3,3}_{\bullet,3}) \cong \{0;1^3\}$ was
established in \Cref{sec:cubic-dimens-three}. Notice that
$\{n-3\} \otimes \{0;1^3\}$ is 1-dimensional, so $E^1_{3,0}$ is
generated as a $\Bbbk[\mathfrak{H}_n]$-module by a single
element. Thus, the differential
$d^1_{3,0} \colon E^1_{3,0} \to E^1_{2,0}$ is completely determined by
its value on a generator of $E^1_{3,0}$. We illustrate how to compute
this value using our software package. In a new instance of Macaulay2,
we run the following commands.
\begin{verbatim}
needsPackage "VietorisRipsHypercube"
X53 = VRHypercube(5,3)
differential(X53,3)
\end{verbatim}
The output tells us that $d^1_{3,0}$ sends the generator of
$E^1_{3,0}$ corresponding to the cube $\{0\}^{n-3} \times [0,1]^3$ to
the linear combination of simplices written in \Cref{eq:22} omitting
the $n-3$ initial zeros.
\begin{equation}\label{eq:22}
  \begin{split}
    &-\left[ ([000,001,011]-[000,010,011]) - ([100,101,111]-[100,110,111])\right]\\
    &+\left[ ([000,001,101]-[000,100,101]) -([010,011,111]-[010,110,111]) \right]\\
    &-\left[ ([000,010,110]-[000,100,110]) -([001,011,111]-[001,101,111]) \right]
  \end{split}
\end{equation}
As was the case for $d^1_{3,0}$ in \Cref{sec:spectr-sequ-scale}, each
row contains a difference of two opposite faces of
$\{0\}^{n-3} \times [0,1]^3$, each face being a square triangulated
along a diagonal. Our software chose a different triangulation than
the one we picked by hand in \Cref{sec:spectr-sequ-scale}, but the two
are equivalent in $E^1_{2,0}$, as observed in
\Cref{sec:size-two}. This allows us to identify $d^1_{3,0}$ with the
differential $\partial_3\colon \mathcal{F}_3\to \mathcal{F}_2$ in the
cellular chain complex of the $n$-cube. Similarly, we have
\begin{equation*}
  E^1_{4,0} \cong 
  \operatorname{Ind}^{\mathfrak{H}_n}_{\mathfrak{S}_{n-4} \times \mathfrak{H}_4}
  \left( \{n-4\} \otimes H_4 (C^{4,3}_{\bullet,4}) \right) \cong
  \operatorname{Ind}^{\mathfrak{H}_n}_{\mathfrak{S}_{n-4} \times \mathfrak{H}_4}
  \left( \{n-4\} \otimes \{0;1^4\} \right)
\end{equation*}
where the isomorphism $H_4 (C^{4,3}_{\bullet,4}) \cong \{0;1^4\}$ was
established in \Cref{sec:cubic-dimens-three}. Again,
$\{n-4\} \otimes \{0;1^4\}$ is 1-dimensional, so $E^1_{4,0}$ is
generated as a $\Bbbk[\mathfrak{H}_n]$-module by a single
element. Issuing the command
\begin{verbatim}
differential(X53,4)
\end{verbatim}
we obtain the image under the differential $d^1_{4,0}$ of the
generator of $E^1_{4,0}$ corresponding to the hypercube
$\{0\}^{n-4} \times [0,1]^4$. After rearranging terms and omitting the
$n-4$ initial zeros, we get the expression in \Cref{eq:23}.
\begin{equation}\label{eq:23}
  \begin{split}
    -\big[\,
    (&[0000,0001,0011,0111]-[0000,0010,0011,0111]+[0000,0010,0110,0111]\\
    &\quad -[0000,0100,0110,0111]+[0000,0100,0101,0111]-[0000,0001,0101,0111])\\
    -(&[1000,1001,1011,1111]-[1000,1010,1011,1111]+[1000,1010,1110,1111]\\
    &\quad -[1000,1100,1110,1111]+[1000,1100,1101,1111]-[1000,1001,1101,1111])
    \,\big]\\
    +\big[\,
    (&[0000,0001,0011,1011]-[0000,0010,0011,1011]+[0000,0010,1010,1011]\\
    &\quad -[0000,1000,1010,1011]+[0000,1000,1001,1011]-[0000,0001,1001,1011])\\
    -(&[0100,0101,0111,1111]-[0100,0110,0111,1111]+[0100,0110,1110,1111]\\
    &\quad -[0100,1100,1110,1111]+[0100,1100,1101,1111]-[0100,0101,1101,1111])
    \,\big]\\
    -\big[\,
    (&[0000,0001,0101,1101]-[0000,0100,0101,1101]+[0000,0100,1100,1101]\\
    &\quad -[0000,1000,1100,1101]+[0000,1000,1001,1101]-[0000,0001,1001,1101])\\
    -(&[0010,0011,0111,1111]-[0010,0110,0111,1111]+[0010,0110,1110,1111]\\
    &\quad -[0010,1010,1110,1111]+[0010,1010,1011,1111]-[0010,0011,1011,1111])
    \,\big]\\
    +\big[\,
    (&[0000,0010,0110,1110]-[0000,0100,0110,1110]+[0000,0100,1100,1110]\\
    &\quad -[0000,1000,1100,1110]+[0000,1000,1010,1110]-[0000,0010,1010,1110])\\
    -(&[0001,0011,0111,1111]-[0001,0101,0111,1111]+[0001,0101,1101,1111]\\
    &\quad -[0001,1001,1101,1111]+[0001,1001,1011,1111]-[0001,0011,1011,1111])
    \,\big]
  \end{split}
\end{equation}
Each block in square brackets contains the difference of two opposite
faces of $\{0\}^{n-4} \times [0,1]^4$. Each face is a cube
triangulated into 3-simplices. The base triangulation of the cube used
here is always given by the linear combination of simplices in
\Cref{eq:24}.
\begin{equation}
  \label{eq:24}
  \begin{split}
    &[000,001,011,111]-[000,010,011,111]+[000,010,110,111]\\
    &\quad -[000,100,110,111]+[000,100,101,111]-[000,001,101,111]
  \end{split}
\end{equation}
Appending a zero to all binary sequences appearing in \Cref{eq:24} we
get a triangulation of the $n$-th front face of
$\{0\}^{n-4} \times [0,1]^4$, which appears as the first block in
parentheses in the last set of square brackets of
\Cref{eq:23}. Similarly, appending a one, we get a triangulation of
the $n$-th back face of $\{0\}^{n-4} \times [0,1]^4$, which appears as
the second block in parentheses in the last set of square brackets of
\Cref{eq:23}. Inserting zeros and ones in different positions gives a
triangulation of the other front and back faces. After replacing each
cube with its triangulation, the differential
$\partial_4\colon \mathcal{F}_4 \to \mathcal{F}_3$ in the cellular
chain complex of the $n$-cube produces exactly the expression of
\Cref{eq:23}. We conclude that
$E^1_{\bullet,0} \cong (\mathcal{F}_{\leqslant 4})_\bullet$.

We can move on to the second
page of the spectral sequence, which is depicted in
\Cref{fig:8}. Using \Cref{pro:2}, we have isomorphisms
\begin{equation}\label{eq:25}
  \begin{split}
    &E^2_{0,0} = H_0 (E^1_{\bullet,0}) \cong \{n;0\},\\
    &E^2_{4,0} = H_4 (E^1_{\bullet,0}) \cong
      \bigoplus_{j=1}^{n-4} \{n-4-j; (j,1^4)\}
  \end{split}
\end{equation}
of $\mathfrak{H}_n$-representations.
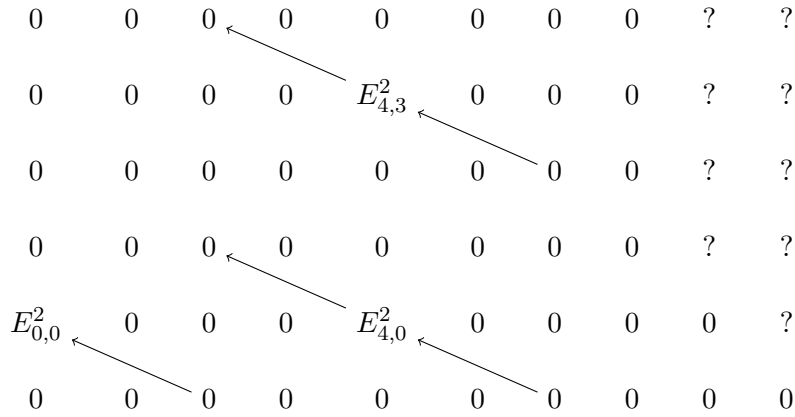
\begin{figure}[htb]
  \centering
  \begin{tikzpicture}[description/.style={fill=white,inner sep=2pt}]
    \matrix (m) [matrix of math nodes, row sep=1.2em,
    column sep=1.5em, text height=1.5ex, text depth=0.25ex]
    { 0 & 0 & 0 & 0 & 0 & 0 & 0 & 0 & ? & ?\\
      0 & 0 & 0 & 0 & E^2_{4,3} & 0 & 0 & 0 & ? & ?\\
      0 & 0 & 0 & 0 & 0 & 0 & 0 & 0 & ? & ?\\
      0 & 0 & 0 & 0 & 0 & 0 & 0 & 0 & ? & ?\\
      E^2_{0,0} & 0 & 0 & 0 & E^2_{4,0} & 0 & 0 & 0 & 0 & ?\\
      0 & 0 & 0 & 0 & 0 & 0 & 0 & 0 & 0 & 0\\ };
    \draw[<-] (m-5-1) edge (m-6-3);
    \draw[<-] (m-4-3) edge (m-5-5) (m-5-5) edge (m-6-7);
    \draw[<-] (m-1-3) edge (m-2-5) (m-2-5) edge (m-3-7);
  \end{tikzpicture}
  
  \caption{The page $E^2_{\bullet,\bullet}$ for scale three}
  \label{fig:8}
\end{figure}
Since $E^1_{3,3} = E^1_{5,3} = 0$, we deduce that
$E^2_{4,3} \cong E^1_{4,3}$. Then, combining \Cref{pro:1} with the
computation from \Cref{sec:cubic-dimens-three}, we get an isomorphism
\begin{equation}\label{eq:26}
  E^2_{4,3} \cong
  \operatorname{Ind}^{\mathfrak{H}_n}_{\mathfrak{S}_{n-4} \times \mathfrak{H}_4}
  \left( \{n-4\} \otimes \{4;0\} \right)
\end{equation}
of $\mathfrak{H}_n$representations.

We are finally ready for the main result of this section.

\begin{theorem}\label{thm:2}
  For every integer $n\geqslant 4$, the nonzero homology groups of
  $X^{n,3}$ are
  \begin{equation*}
    \begin{split}
      H_0 (X^{n,3}) &\cong \{n;0\},\\
      H_4 (X^{n,3}) &\cong
      \bigoplus_{j=1}^{n-4} \{n-4-j; (j,1^4)\} \\
      H_7 (X^{n,3}) &\cong
      \operatorname{Ind}^{\mathfrak{H}_n}_{\mathfrak{S}_{n-4} \times \mathfrak{H}_4}
                      \left( \{n-4\} \otimes \{4;0\} \right)
      \cong \bigoplus_{i=0}^{n-4} \bigoplus_{j=0}^{\min\{4,n-4-i\}} \{(n-i-j,j);i\}
    \end{split}
  \end{equation*}
  as representations of $\mathfrak{H}_n$.
\end{theorem}

\begin{proof}
  Using the spectral sequence constructed in this section, we have
  isomorphisms
  \begin{equation}
    H_k (X^{n,3}) \cong \bigoplus_{p+q=k} E^\infty_{p,q}
  \end{equation}
  of $\mathfrak{H}_n$-representations. From \Cref{fig:8} and
  \Cref{pro:3}, it follows that, with the exception of $(p,q)=(0,0)$,
  $(p,q)=(4,0)$, and $(p,q)=(4,3)$, we have $E^2_{p,q} = 0$ whenever
  $0\leqslant p+q\leqslant 8$. This implies $E^i_{p,q} = 0$ for all
  $i\geqslant 3$ with the same choices of $p$ and $q$.  It follows
  that $H_k (X^{n,3}) = 0$ for $k\in \{1,2,3,5,6,8\}$.  We also get
  $E^\infty_{0,0} \cong E^2_{0,0}$ because $E^\infty_{0,0}$ is the
  quotient of $E^2_{0,0}$ by images of maps originating from terms
  $E^i_{p,q}$ with $i\geqslant 3$ and $p+q=1$, which are all zero.
  Similarly, we have $E^\infty_{4,0} \cong E^2_{4,0}$ and
  $E^\infty_{4,3} \cong E^2_{4,3}$.  Thus, we get
  $H_0 (X^{n,3}) \cong E^2_{0,0}$, $H_4 (X^{n,3}) \cong E^2_{4,0}$,
  and $H_7 (X^{n,3}) \cong E^2_{4,3}$, with these terms described in
  \Cref{eq:25,eq:26}. We appeal to \cite[Theorem A]{zbMATH07725058} to
  conclude that there are no other nonzero homology groups.

  The decomposition of $H_7 (X^{n,3})$ into irreducible
  representations follows from \cite[Lemma III.6, Theorem
  III.2]{zbMATH03590472}.
\end{proof}

\subsection{Dimension check}
\label{sec:dimension-check-2}

By \cite[Theorem 24]{zbMATH08181349}, we have
\begin{equation*}
  \dim_\Bbbk H_4 (X^{n,3}) =
  \sum_{i=4}^{n-1} 2^{i-4} \binom{i}{4} =
  \frac{2^{n-7}}{3} (n^4-14n^3-83n^2-262n+384) - 1,
\end{equation*}
where the last equality is obtained using the same methods that gave
us \Cref{eq:19}.  Next, we compute the dimension of $H_4 (X^{n,3})$
using \Cref{thm:2}. By \Cref{eq:29}, the dimension of
$\bigoplus_{j=1}^{n-4} \{n-4-j; (j,1^4)\}$ is
\begin{equation*}
  \sum_{j=1}^{n-4} \binom{n}{j+4} \binom{j+3}{4} =
  \frac{2^{n-7}}{3} (n^4-14n^3-83n^2-262n+384) - 1,
\end{equation*}
with the equality coming from the same methods used to get
\Cref{eq:20}. Thus, we see that our result is compatible with
\cite[Theorem 24]{zbMATH08181349}.

As for $H_7 (X^{n,3})$, we use its description as an induced
representation from \Cref{thm:2} to obtain
\begin{equation*}
  \dim_\Bbbk H_7 (X^{n,3}) = 2^{n-4} \binom{n}{4},
\end{equation*}
which also matches \cite[Theorem 24]{zbMATH08181349}.

\section{Submaximal scale}
\label{sec:submaximal-scale}

In this section, we describe the action of the hyperoctahedral group
$\mathfrak{H}_n$ on the homology of $X^{n,n-1}$. Since the largest
distance between two vertices of $Q_n$ is $n$, we refer to the scale
$r=n-1$ as \emph{submaximal}. It is known that $X^{n,n-1}$ is homotopy
equivalent to the sphere $S^{2^{n-1}-1}$, with the proof sketched out
in \cite[Section 3]{zbMATH07553291} and \cite[Section
4]{zbMATH07819241}.  We have $H_0 (X^{n,n-1}) \cong \{n;0\}$; this
follows from \Cref{sec:scale-one}, because $X^{n,n-1}$ has all the
vertices and edges of $X^{n,1}$. Thus, we only need to determine the
group action on $H_{2^{n-1}-1} (X^{n,n-1})$. We do this with tools of
commutative algebra.

For a simplicial complex $\Delta$, let $R$ be the polynomial ring with
coefficients in the field $\Bbbk$ and variables corresponding to the
vertices of $\Delta$. The \emph{Stanley-Reisner} ideal $I_\Delta$ is
the ideal of $R$ generated the squarefree monomials that are products
of the variables in the minimal nonfaces of $\Delta$. Hochster's
Formula \cite[Theorem 5.5.1]{zbMATH01194481} establishes an
isomorphism between the (co)homology of the induced subcomplexes of
$\Delta$ and the graded components of the modules
$\operatorname{Tor}^R_i (R/I_\Delta,\Bbbk)$. If $F_\bullet$ is a free
resolution of $R/I_\Delta$ as an $R$-module, then
$\operatorname{Tor}^R_i (R/I_\Delta,\Bbbk) \cong H_i (F_\bullet
\otimes_R \Bbbk)$ \cite[Definition 2.6.4]{zbMATH00595200}.

Consider the polynomial ring
$R = \Bbbk [x_{(0^n)} , \dots, x_{(1^n)}]$ with variables indexed by
the vertices of the hypercube $Q_n$. The Stanley-Reisner ideal of
$X^{n,n-1}$ is generated by the quadratic squarefree monomials
$x_u x_{\hat{u}}$ where $\hat{u}$ is the unique vertex of $Q_n$ such
that $d(u,\hat{u})=n$. Moreover, these $2^{n-1}$ monomials form a
minimal generating set of $I_{X^{n,n-1}}$, and a regular sequence as
they are pairwise coprime \cite[Section 2]{zbMATH05718313}.
Therefore, the quotient $R/I_{X^{n,n-1}}$ has a canonical free
resolution $(K_\bullet,d_\bullet)$ oven $R$ known as the \emph{Koszul
  complex}.  We briefly review its construction (see \cite[Section
1.6]{zbMATH01194481} for more details).  Fix an ordering
$f_1,\dots,f_{2^{n-1}}$ of the minimal generators of $I_{X^{n,n-1}}$.
Let $M=\langle f_1,\dots,f_{2^{n-1}}\rangle$ be the $\Bbbk$-vector
space spanned by $f_1,\dots,f_{2^{n-1}}$.  Then,
$K_j = \bigwedge^j M \otimes_\Bbbk R(-2j)$ is a free $R$-module of
rank $\binom{2^{n-1}}{j}$ with a basis given by the exterior products
$f_{i_1} \wedge \cdots \wedge f_{i_j}$ for all choices of indices
$1\leqslant i_1 < \dots < i_j \leqslant 2^{n-1}$. The $R$-module
structure and grading are determined by the tensor factor $R(-2j)$,
which is a copy of the ring $R$ with its grading shifted so that the
multiplicative unit $1_R$ has degree $2j$.  We define a map of
$R$-modules $d_j \colon K_j \to K_{j-1}$ by setting
\begin{equation}\label{eq:27}
  d_j (f_{i_1} \wedge \cdots \wedge f_{i_j}) =
  \sum_{k=1}^j f_{i_1} \wedge \cdots \wedge \widehat{f_{i_k}} \wedge \cdots \wedge f_{i_j}
  \otimes (-1)^k f_{i_k},
\end{equation}
where $f_{i_k}$ is removed from the exterior product and multiplied as
an $R$-coefficient, then extending $R$-linearly. We have
$d_j d_{j-1} = 0$, so $(K_\bullet, d_\bullet)$ is a complex of
$R$-modules. Note that the hyperoctahedral group $\mathfrak{H}_n$ acts
on $R$ and stabilizes the vector space $M$, therefore it acts
naturally on the modules $K_j$. Moreover, the $\mathfrak{H}_n$-action
commutes with the differentials $d_j$. Therefore,
$(K_\bullet, d_\bullet)$ is also a complex of
$\mathfrak{H}_n$-representations.

\begin{proposition}\label{pro:4}
  For every integer $n\geqslant 2$, there is an isomorphism
  $H_{2^{n-1}-1} (X^{n,n-1}) \cong \bigwedge^{2^{n-1}} M$ of
  $\mathfrak{H}_n$-representations, where $M$ is the $\Bbbk$-vector
  spaced spanned by the minimal generators of the Stanley-Reisner
  ideal of $X^{n,n-1}$.
\end{proposition}

\begin{proof}
  Since the minimal generators of $I_{X^{n,n-1}}$ form a regular
  sequence, the Koszul complex $(K_\bullet, d_\bullet)$ is a free
  resolution of $R/I_{X^{n,n-1}}$ \cite[Corollary
  1.6.14]{zbMATH01194481}. Given that
  $\Bbbk \cong R/(x_{(0^n)},\dots,x_{(1^n)})$, the complex
  $K_\bullet \otimes_\Bbbk \Bbbk$ has terms
  \begin{equation*}
    K_j \otimes_\Bbbk \Bbbk \cong \bigwedge^j M \otimes_\Bbbk \Bbbk (-2j),
  \end{equation*}
  which is a copy of $\bigwedge^j M$ concentrated in degree
  $2j$. Moreover, the differentials in $K_\bullet \otimes_\Bbbk \Bbbk$
  are all zero because tensoring with $\Bbbk$ means modding out by the
  ideal $(x_{(0^n)},\dots,x_{(1^n)})$, which kills all coefficients in
  \Cref{eq:27}. Therefore, we have
  \begin{equation*}
    \operatorname{Tor}^R_{2^{n-1}} (R/I_{X^{n,n-1}}, \Bbbk) \cong
    H_{2^{n-1}} (K_\bullet \otimes_\Bbbk \Bbbk) \cong
    \bigwedge^{2^{n-1}} M \otimes_\Bbbk \Bbbk (-2^n),
  \end{equation*}
  which has a single nonzero graded component living in degree $2^n$.
  Finally, by Hochster's Formula \cite[Theorem 5.5.1]{zbMATH01194481},
  we get
  \begin{equation*}
    H_{2^{n-1}-1} (X^{n,n-1}) \cong
    \left(
      \operatorname{Tor}^R_{2^{n-1}} (R/I_{X^{n,n-1}}, \Bbbk)
    \right)_{2^n} \cong
    \bigwedge^{2^{n-1}} M,
  \end{equation*}
  the middle term being the degree $2^n$ component of
  $\operatorname{Tor}^R_{2^{n-1}} (R/I_{X^{n,n-1}}, \Bbbk)$.  The
  isomorphisms used throughout the proof are compatible with the group
  action, so the result holds at the level of
  $\mathfrak{H}_n$-representations.
\end{proof}

\Cref{pro:4} gives a practical way to understand the action of
$\mathfrak{H}_n$ on the 1-dimensional vector space
$H_{2^{n-1}-1} (X^{n,n-1})$. By \Cref{eq:29}, $\mathfrak{H}_n$ has
only four 1-dimensional irreducible representations, namely:
\begin{itemize}
\item $\{n;0\}$, on which both subgroups $\mathbb{Z}_2^n$ and
  $\mathfrak{S}_n$ act trivially;
\item $\{1^n;0\}$, on which $\mathbb{Z}_2^n$ acts trivially and each
  transposition in $\mathfrak{S}_n$ acts by changing sign;
\item $\{0;n\}$, on which each $e_i \in \mathbb{Z}_2^n$ acts by
  changing sign and $\mathfrak{S}_n$ acts trivially;
\item $\{0;1^n\}$, on which each $e_i \in \mathbb{Z}_2^n$ and each
  transposition in $\mathfrak{S}_n$ acts by changing sign.
\end{itemize}
We illustrate some of the different possibilities.

\begin{example}
  Let $n=2$. The space $\bigwedge^2 M$ is spanned by the alternating
  tensor
  \begin{equation*}
    m = x_{(01)}x_{(10)} \wedge x_{(00)} x_{(11)}.
  \end{equation*}
  Observe that
  \begin{equation*}
    \begin{split}
      (01,\operatorname{id}_{\mathfrak{S}_2}) m
      &= x_{(00)}x_{(11)} \wedge x_{(01)} x_{(10)}
        = - m,\\
      (00,(12)) m
      &= x_{(10)}x_{(01)} \wedge x_{(00)} x_{(11)}
        = m.
    \end{split}
  \end{equation*}
  By symmetry, this shows that each $e_i\in \mathbb{Z}_2^2$ acts by
  changing sign, while $\mathfrak{S}_2$ acts trivially.  Therefore,
  $H_1 (X^{2,1}) \cong \{0;2\}$.
\end{example}

\begin{example}
  Let $n=3$. The space $\bigwedge^4 M$ is spanned by the alternating
  tensor
  \begin{equation*}
    m = x_{(011)}x_{(100)}\wedge x_{(010)}x_{(101)}\wedge
    x_{(001)}x_{(110)}\wedge x_{(000)}x_{(111)}.
  \end{equation*}
  Observe that
  \begin{equation*}
    \begin{split}
      (001,\operatorname{id}_{\mathfrak{S}_3}) m
      &= x_{(010)}x_{(101)}\wedge x_{(011)}x_{(100)}\wedge
    x_{(000)}x_{(111)}\wedge x_{(001)}x_{(110)}
        = m,\\
      (000,(23)) m
      &= x_{(011)}x_{(100)}\wedge x_{(001)}x_{(110)}\wedge
    x_{(010)}x_{(101)}\wedge x_{(000)}x_{(111)}
        = - m.
    \end{split}
  \end{equation*}
  By symmetry, this shows that $\mathbb{Z}_2^3$ acts trivially, while
  each transposition in $\mathfrak{S}_3$ acts by changing sign.
  Therefore, $H_3 (X^{3,2}) \cong \{1^3;0\}$.
\end{example}

\begin{example}
  Let $n=4$. The space $\bigwedge^8 M$ is spanned by the alternating
  tensor
  \begin{equation*}
    \begin{split}
      m = {}&x_{(0111)}x_{(1000)}\wedge x_{(0110)}x_{(1001)}\wedge x_{(0101)}x_{(1010)}
      \wedge x_{(0100)}x_{(1011)}\wedge\\
      &x_{(0011)}x_{(1100)}\wedge x_{(0010)}x_{(1101)}
      \wedge x_{(0001)}x_{(1110)}\wedge x_{(0000)}x_{(1111)}.
    \end{split}
  \end{equation*}
  Observe that
  \begin{equation*}
    \begin{split}
      (0001,\operatorname{id}_{\mathfrak{S}_4}) m =
      {}&x_{(0110)}x_{(1001)}\wedge x_{(0111)}x_{(1000)}\wedge x_{(0100)}x_{(1011)}
          \wedge x_{(0101)}x_{(1010)}\wedge\\
        &x_{(0010)}x_{(1101)}\wedge x_{(0011)}x_{(1100)}
          \wedge x_{(0000)}x_{(1111)}\wedge x_{(0001)}x_{(1110)}
      = m,\\
      (0000,(34)) m =
      {}&x_{(0111)}x_{(1000)}\wedge x_{(0101)}x_{(1010)}\wedge x_{(0110)}x_{(1001)}
          \wedge x_{(0100)}x_{(1011)}\wedge\\
        &x_{(0011)}x_{(1100)}\wedge x_{(0001)}x_{(1110)}
          \wedge x_{(0010)}x_{(1101)}\wedge x_{(0000)}x_{(1111)}
          = m.
    \end{split}
  \end{equation*}
  By symmetry, this shows that both $\mathbb{Z}_2^4$ and
  $\mathfrak{S}_4$ act trivially.  Therefore,
  $H_7 (X^{4,3}) \cong \{4;0\}$.
\end{example}

\begin{theorem}\label{thm:mainr-1}
  For every integer $n\geqslant 4$, the nonzero homology groups of
  $X^{n,n-1}$ are
  \begin{equation*}
    H_0 (X^{n,n-1}) \cong H_{2^{n-1}-1} (X^{n,n-1}) \cong \{n;0\}
  \end{equation*}
  as representations of $\mathfrak{H}_n$.
\end{theorem}

\begin{proof}
  As explained at the beginning of this section, we focus on the
  homology in dimension $2^{n-1}-1$.  By \Cref{pro:4}, we have
  $H_{2^{n-1}-1} (X^{n,n-1}) \cong \bigwedge^{2^{n-1}} M$, where $M$
  is the $\Bbbk$-vector space spanned by the monomials
  $x_u x_{\hat{u}} \in \Bbbk [x_{(0^n)} , \dots, x_{(1^n)}]$ such that
  $d(u,\hat{u})=n$. Recall that $\bigwedge^{2^{n-1}} M$ is
  1-dimensional, generated by the exterior product of the monomials
  $x_u x_{\hat{u}}$, which we denote by $m$.

  First, consider the action of
  $(10^{n-1},\operatorname{id}_{\mathfrak{S}_n}) \in
  \mathfrak{H}_n$. This element has order 2 and does not fix any
  monomial $x_u x_{\hat{u}}$, so each monomial $x_u x_{\hat{u}}$ is
  paired with a different monomial $x_v x_{\hat{v}}$, and
  $(10^{n-1},\operatorname{id}_{\mathfrak{S}_n})$ acts by swapping
  them. The number of pairs $x_u x_{\hat{u}},x_v x_{\hat{v}}$ being
  swapped by $(10^{n-1},\operatorname{id}_{\mathfrak{S}_n})$ is
  $2^{n-2}$, which is even for $n\geqslant 3$. This means that
  \begin{equation*}
    (10^{n-1},\operatorname{id}_{\mathfrak{S}_n}) m = (-1)^{2^{n-2}} m = m.
  \end{equation*}
  In other words, $(10^{n-1},\operatorname{id}_{\mathfrak{S}_n})$ acts
  trivially on $\bigwedge^{2^{n-1}} M$. By symmetry, the subgroup
  $\mathbb{Z}_2^n$ of $\mathfrak{H}_n$ acts trivially on
  $\bigwedge^{2^{n-1}} M$.

  Next, consider the action of $(0^n,(12)) \in \mathfrak{H}_n$. This
  element also has order 2, and it fixes all monomials
  $x_u x_{\hat{u}}$ such that the length $n$ binary string $u$ starts
  with 00 or 11.  Since $n>2$, every monomial $x_u x_{\hat{u}}$ such
  that $u$ starts with 01 or 10 is paired with a different monomial
  $x_v x_{\hat{v}}$, and $(0^n,(12))$ acts by swapping them. Note that
  if $u$ starts with 01, then $\hat{u}$ starts with 10, and vice
  versa.  Therefore, the number of pairs
  $x_u x_{\hat{u}},x_v x_{\hat{v}}$ being swapped by $(0^n,(12))$ is
  equal to half the number of length $n$ binary strings starting with
  01, i.e., $2^{n-3}$. Since $2^{n-3}$ is even for $n\geqslant 4$, we
  have
  \begin{equation*}
    (0^n,(12)) m = (-1)^{2^{n-3}} m = m.
  \end{equation*}
  Thus, $(0^n,(12))$ also acts trivially on $\bigwedge^{2^{n-1}} M$
  and, similarly, the same holds for all elements
  $(0^n,\tau)\in \mathfrak{H}_n$ where $\tau$ is a
  transposition. Since it is generated precisely by these elements,
  the subgroup $\mathfrak{S}_n$ of $\mathfrak{H}_n$ acts trivially on
  $\bigwedge^{2^{n-1}} M$.

  We conclude that
  $H_{2^{n-1}-1} (X^{n,n-1}) \cong \bigwedge^{2^{n-1}} M \cong
  \{n;0\}$.
\end{proof}

\section{Open problems}
\label{sec:open-problems}

There are several aspects of this article that we believe could use
further investigation. A few of these concern the filtration
$\{X^{n,r}_p\}_{p\geqslant 0}$ of \Cref{sec:suppl-filtr} and the
related spectral sequence. For scales two and three, the spectral
sequence degenerates at 2, meaning that all differentials in
$E^i_{p,q}$ are zero for $i\geqslant 2$.

\begin{problem}
  \begin{enumerate}
  \item Does the spectral sequence degenerate at 2 for all scales? 
  \item If not, does the spectral sequence degenerate at some index
    $i$ depending only on the scale $r$?
  \end{enumerate}
\end{problem}

Other questions concern some patterns observed at scales two and three
for the early pages of the spectral sequence. For example, the first
$r$ columns seemingly lead to the chain complex of the cube.
\begin{problem}
  Show that, for all integers $p$ with $0\leqslant p\leqslant r$, we
  have:
  \begin{itemize}
  \item $E^1_{p,0} \cong \mathcal{F}_p$, where $\mathcal{F}_\bullet$
    is the cellular chain complex of the $n$-cube;
  \item and $E^1_{p,q} = 0$ for all $q\neq 0$.
  \end{itemize}
\end{problem}
On the other hand, the later columns appear to have no homology at
all.
\begin{problem}
  \begin{enumerate}
  \item Is there an integer $s$ depending only on $r$ such that for
    all integers $p,q$ with $p> s$ we have $E^1_{p,q}=0$?
  \item If so, find an explicit formula for $s$ as a function of $r$.
  \end{enumerate}
\end{problem}
We note that for $r=2$ we have $s=3$, and for $r=3$ we have $s=4$.

Another series of questions relates to commutative algebra.  Let
$R = \Bbbk [x_{(0^n)} , \dots, x_{(1^n)}]$ and let
$I_{X^{n,r}} \subseteq R$ be the Stanley-Reisner ideal of
$X^{n,r}$. Let $\beta_{i,j}$ be the dimension of the component of
degree $j$ of $\operatorname{Tor}^R_i (R/I_{X^{n,r}})$ as a
$\Bbbk$-vector space. The integers $\beta_{i,j}$ are known as the
\emph{Betti numbers} of $R/I_{X^{n,r}}$ and carry significant
algebraic and geometric information about the ring $R/I_{X^{n,r}}$;
see \cite{zbMATH02136428} for an introduction to the subject. It
follows from Hochster's Formula that the Betti numbers of
$R/I_{X^{n,r}}$ encode the topological Betti numbers of all
subcomplexes of $X^{n,r}$. Moreover, the hyperoctahedral group
$\mathfrak{H}_n$ acts naturally on the modules
$\operatorname{Tor}^R_i (R/I_{X^{n,r}})$, and this matches the
$\mathfrak{H}_n$-action on the homology of $X^{n,r}$ and (orbits of)
its subcomplexes.  The characters of the graded components of
$\operatorname{Tor}^R_i (R/I_{X^{n,r}})$ are known as the \emph{Betti
  characters} of $R/I_{X^{n,r}}$ \cite{zbMATH07518263}. We initially
considered studying the Betti characters of $R/I_{X^{n,r}}$ to
determine how $\mathfrak{H}_n$ acts on $H_\bullet (X^{n,r})$, but this
remains open.

\begin{problem}\label{prob:open-problems-1}
  Determine the Betti numbers of $R/I_{X^{n,r}}$.
\end{problem}

\begin{problem}\label{prob:open-problems-2}
  Determine the Betti characters of $R/I_{X^{n,r}}$.
\end{problem}

In \cite{zbMATH07730801}, the authors developed an equivariant analog
of Hochster's Formula for simplicial complexes with vertex set
$\{1,\dots,n\}$ that are stable under the permutation action of
$\mathfrak{S}_n$.  The representation spanned by the vertices is
isomorphic to the vector space $\Bbbk^n$ with $\mathfrak{S}_n$ acting
as the group of $n\times n$ permutation matrices.  The natural analog
for $\mathfrak{H}_n$ would be to consider the action on $\Bbbk^n$ of
the group of $n\times n$ signed permutation matrices (see
\cite[Section 1.1]{zbMATH00047598} for an equivalent construction),
which is not a permutation representation.  However, our setting
features simplicial complexes with the same vertices as the hypercube
graph $Q_n$ that are stable under the permutation action of
$\mathfrak{H}_n$. In this case, the span of the vertices is a
$2^n$-dimensional permutation representation whose decomposition into
irreducibles leads to a more complex situation than
\cite{zbMATH07730801}.  An analog of Hochster's Formula for this new
setting would yield solutions to Problems \ref{prob:open-problems-1}
and \ref{prob:open-problems-2}, and might be of independent interest.
\begin{problem}
  Develop an equivariant analog of Hochster's Formula for simplicial
  complexes on the vertices of the hypercube graph $Q_n$ that are
  stable under the action of $\mathfrak{H}_n$.
\end{problem}
We also note that commutative algebra may provide additional ways to
study the complexes $X^{n,r}$ and their homology; see, for example,
\cite{arXiv:2605.00705}.

In a different direction, it may be interesting to study the homology
of $X^{n,r}$ from the point of view of representation stability and
FI-modules pioneered in \cite{zbMATH06264399,zbMATH06464862}. Our
results in \Cref{thm:main0,thm:main1,thm:1,thm:2,thm:mainr-1} suggest
that, under the natural inclusions
$X^{n,r} \hookrightarrow X^{n+1,r}$, $H_\bullet (X^{n,r})$ may enjoy
some form of representation stability.
\begin{problem}
  \begin{enumerate}
  \item Determine if $H_\bullet (X^{n,r})$ is a family of
    $\text{FI}_{BC}$-modules in the sense of \cite{zbMATH06350922},
    which is the natural framework for actions of hyperoctahedral
    groups.
  \item If so, determine if $H_\bullet (X^{n,r})$ is finitely
    generated and find the degrees of generation.
  \end{enumerate}
\end{problem}

Finally, \Cref{thm:main0,thm:main1,thm:1,thm:2,thm:mainr-1} show that
at scales $r\leqslant 3$ and $r=n-1$ the nonzero homology groups
$H_\bullet (X^{n,r})$ are \emph{multiplicity-free} representations of
$\mathfrak{H}_n$, meaning that no irreducible representation of
$\mathfrak{H}_n$ appears as a direct summand in their decomposition
more than once. Multiplicity-free representations are related to
special phenomena in geometry and combinatorics, especially for
algebraic groups \cite{zbMATH01211709,zbMATH07830344}, but also for
finite groups \cite{zbMATH03257452}.
\begin{problem}
  Is $H_i (X^{n,r})$ a multiplicity-free representation of
  $\mathfrak{H}_n$ for all values of $n,r$ and $i$?
\end{problem}

\appendix

\section{Irreducible representations of hyperoctahedral groups}
\label{sec:irred-repr-hyper}

\subsection{Weight spaces}
\label{sec:weight-spaces}

A \emph{linear character} of $\mathbb{Z}_2^n$ is a group homomorphism
$\psi \colon \mathbb{Z}_2^n \to \mathbb{C}^\times$. Every nontrivial
element $s\in \mathbb{Z}_2^n$ has order two, so $\psi (s)^2 =1$ or,
equivalently, $\psi (s) = \pm 1$; this allows us to replace
$\mathbb{C}^\times$ with the multiplicative group of an arbitrary
field $\Bbbk$ of characteristic zero.  The set of linear characters of
$\mathbb{Z}_2^n$ over $\Bbbk$ is a group under ``pointwise
multiplication'', and it is known as the \emph{character group} of
$\mathbb{Z}_2^n$ over $\Bbbk$. More formally, the character group is
$\operatorname{Hom}_\mathbb{Z} (\mathbb{Z}_2^n, \Bbbk^\times)$.

For $w\in \mathbb{Z}_2^n$, we define
\begin{equation*}
  \begin{split}
    \psi_w \colon \mathbb{Z}_2^n &\longrightarrow \Bbbk^\times\\
    s &\longmapsto (-1)^{w\cdot s}
  \end{split}
\end{equation*}
where $w\cdot s$ is the dot product of $w$ and $s$ as vectors in
$\mathbb{R}^n$. The function $\psi_w$ is a linear character of
$\mathbb{Z}_2^n$, and the assignment $w\mapsto \psi_w$ defines a group
homomorphism from $\mathbb{Z}_2^n$ to the character group
$\operatorname{Hom}_\mathbb{Z} (\mathbb{Z}_2^n, \Bbbk^\times)$.  In
fact, this function is a group isomorphism
$\mathbb{Z}_2^n \cong \operatorname{Hom}_\mathbb{Z} (\mathbb{Z}_2^n,
\Bbbk^\times)$ (see \cite[Section 4.3]{zbMATH05943514}).

Let $V$ be a representation of $\mathbb{Z}_2^n$ over $\Bbbk$. For
every $w\in \mathbb{Z}_2^n$, define the \emph{weight space} of $V$
with weight $w$ as
\begin{equation*}
  V_w = \{ v\in V \,|\, \forall s\in\mathbb{Z}_2^n, sv=\psi_w (s) v\};
\end{equation*}
the nonzero elements of $V_w$ are the \emph{weight vectors} of $V$
with weight $w$.  Since multiplication by $s$ has order two, it
diagonalizes (see \cite[Corollary 3.3.10]{zbMATH06125590}); then, we can
think of $V_w$ as an intersection of eigenspaces
\begin{equation*}
  V_w =  \bigcap_{s\in\mathbb{Z}_2^n} \{ v\in V \,|\, sv=\psi_w (s) v\}.
\end{equation*}
For example, if $w$ is the identity element in $\mathbb{Z}_2^n$, then
$V_w$ is the invariant subspace of $V$, i.e., the set of elements of
$V$ on which $\mathbb{Z}_2^n$ acts trivially. Moreover, since
$\mathbb{Z}_2^n$ is abelian, the operators corresponding to
multiplication by the group elements can be simultaneously
diagonalized (see \cite[Theorem 1.3.21]{zbMATH06125590}), i.e., $V$
admits a basis whose elements are eigenvectors of all $2^n$
operators. Thus, every representation of $\mathbb{Z}_2^n$ is
isomorphic to the direct sum of its weight spaces. Every subspace of a
weight space $V_w$ is stable under the $\mathbb{Z}_2^n$-action, so
$V_w$ further decomposes into a direct sum of 1-dimensional
subrepresentations which are irreducible by virtue of their
dimension. For every $w\in \mathbb{Z}_2^n$, we define the
1-dimensional $\Bbbk$-vector space $U_w = \langle u\rangle$ with the
action $su=\psi_w (s) u$ for every $s\in \mathbb{Z}_2^n$. We conclude
that $\mathbb{Z}_2^n$ has exactly $2^n$ irreducible representations,
which are isomorphic to the $U_w$ for $w\in \mathbb{Z}_2^n$.

\subsection{Construction and dimension}
\label{sec:constr-dimens}

Let $V$ be a representation of $\mathfrak{H}_n$. We identify
$\mathbb{Z}_2^n$ with the subgroup
$\{(s,1_{\mathfrak{S}_n}) \,|\, s\in \mathbb{Z}_2^n\}$; this makes $V$
a representation of $\mathbb{Z}_2^n$ which gives it a decomposition
into weight spaces. For every $s,w\in \mathbb{Z}_2^n$, $v\in V_w$ and
$\sigma \in \mathfrak{S}_n$, we have
\begin{equation*}
  (s,1_{\mathfrak{S}_n}) (0,\sigma) v =
  (0,\sigma) (\sigma^{-1} s,1_{\mathfrak{S}_n}) v =
  (0,\sigma) \psi_w (\sigma^{-1} s) v
\end{equation*}
and also
\begin{equation*}
  \psi_w (\sigma^{-1} s) =
  (-1)^{w\cdot \sigma^{-1} s} =
  (-1)^{\sigma w\cdot s} =
  \psi_{\sigma w} (s);
\end{equation*}
together, these equalities give
\begin{equation*}
  (s,1_{\mathfrak{S}_n}) (0,\sigma) v =
  \psi_{\sigma w} (s) (0,\sigma) v,
\end{equation*}
which means $(0,\sigma) v \in V_{\sigma w}$. If we identify
$\mathfrak{S}_n$ with the subgroup
$\{(0,\sigma ) \,|\, \sigma \in \mathfrak{S}_n \}$, then we can say
that $\sigma \in \mathfrak{S}_n$ acts by sending the weight space
$V_w$ into the weight space $V_{\sigma w}$. A similar argument with
$\sigma^{-1}$ implies that $\sigma V_w = V_{\sigma w}$. In other
words, $\mathfrak{S}_n$ acts on $V$ by permuting its weight
spaces. This implies that a weight space $V_w$ is nonzero if and only
if all the weight spaces $V_{w'}$ with $w'$ in the
$\mathfrak{S}_n$-orbit of $w$ are nonzero. Also, if
$\sigma \in \operatorname{Stab}_{\mathfrak{S}_n} (w)$, the stabilizer
of $w$ in $\mathfrak{S}_n$, then $V_{\sigma w} = V_w$; this makes
$V_w$ a representation of $\operatorname{Stab}_{\mathfrak{S}_n} (w)$.

Suppose now that $V$ is an irreducible representation of
$\mathfrak{H}_n$. The argument of the previous paragraph implies that
$V$ is only allowed weight spaces belonging to a single
$\mathfrak{S}_n$-orbit. Two elements of $\mathbb{Z}_2^n$ are in the
same $\mathfrak{S}_n$-orbit when they have the same numbers of zeros
and ones. In light of this, we choose the elements
\begin{equation*}
  w_k = (\underbrace{0,\dots,0}_{n-k},\underbrace{1,\dots,1}_k),
\end{equation*}
with $0\leqslant k\leqslant n$ to represent all the different orbits.
The stabilizer of $w_k$ in $\mathfrak{S}_n$ consists of the products
$\sigma_1 \sigma_2$ with
$\sigma_1 \in \mathfrak{S}_{n-k} \subset \mathfrak{S}_n$ permuting the
first $n-k$ coordinates of $\mathbb{Z}_2^n$ and
$\sigma_2 \in \mathfrak{S}_k \subset \mathfrak{S}_n$ permuting the
remaining $k$, so we have a group isomorphism
$\operatorname{Stab}_{\mathfrak{S}_n} (w_k) \cong \mathfrak{S}_{n-k}
\times \mathfrak{S}_k$. One should also expect the weight space
$V_{w_k}$ to be irreducible as a representation of
$\operatorname{Stab}_{\mathfrak{S}_n} (w_k)$; otherwise, if
$V_{w_k} = V_{w_k}' \oplus V_{w_k}''$, each summand would generate a
subrepresentation of $V$ under the $\mathfrak{H}_n$-action.  This
makes $V_{w_k}$ isomorphic to a tensor product of Specht modules, one
for $\mathfrak{S}_{n-k}$ and one for $\mathfrak{S}_k$. Combining the
actions of $\mathbb{Z}_2^n$ and
$\operatorname{Stab}_{\mathfrak{S}_n} (w_k)$, we have an isomorphism
$V_{w_k} \cong U_{w_k} \otimes \{\lambda\} \otimes \{\mu\}$ of
representations of
$\mathbb{Z}_2^n \rtimes (\mathfrak{S}_{n-k} \times \mathfrak{S}_k)$,
for some partitions $\lambda$ of $n-k$ and $\mu$ of $k$, and with
$U_{w_k}$ as defined in \Cref{sec:weight-spaces}. The other weight
spaces are $V_{\sigma w_k} = (0,\sigma) V_{w_k}$ for every
$(0,\sigma)$ representing a different left coset of
$\mathbb{Z}_2^n \rtimes (\mathfrak{S}_{n-k} \times \mathfrak{S}_k)$ in
$\mathfrak{H}_n$. Thus, $V$ can be identified with
\begin{equation}\label{eq:28}
  \operatorname{Ind}^{\mathfrak{H}_n}_{\mathbb{Z}_2^n
    \rtimes (\mathfrak{S}_{n-k} \times \mathfrak{S}_k)}
  \left( U_{w_k} \otimes \{\lambda\} \otimes \{\mu\} \right).
\end{equation}
We denote the representation in \Cref{eq:28} by $\{\lambda;\mu\}$. The
general theory guarantees that $\{\lambda;\mu\}$ is irreducible and
every irreducible representation of $\mathfrak{H}_n$ is isomorphic to
$\{\lambda;\mu\}$ for some bipartition $(\lambda,\mu)$ of $n$.

\begin{example}\label{exa:2}
  Consider the case $\lambda = (n-i)$ and $\mu = (i)$ with
  $0\leqslant i\leqslant n$, so that both $\{\lambda\}$ and $\{\mu\}$
  are trivial representations.  The weight space of weight $w_i$ in
  $\{n-i;i\}$ is the 1-dimensional vector space
  $U_{w_i} \otimes \{n-i\} \otimes \{i\}$.  For $u\in U_{w_i}$,
  $x\in \{n-i\}$ and $y\in \{i\}$, we have
  \begin{equation*}
    (s,\sigma_1 \sigma_2) u\otimes x\otimes y =
    (-1)^{w_i \cdot s} u \otimes x \otimes y,
  \end{equation*}
  so the action of
  $(s,\sigma_1 \sigma_2) \in \mathbb{Z}_2^n \rtimes
  \operatorname{Stab} (w_i) \cong \mathbb{Z}_2^n \rtimes
  (\mathfrak{S}_{n-i} \times \mathfrak{S}_i)$ is completely determined
  by the parity of the last $i$ entries of $s$. The other weight
  spaces correspond to permutations that do not stabilize $w_i$, so
  they have actions similar to the one above but determined by the
  parity of a different subset of entries of $s\in \mathbb{Z}_2^n$.
\end{example}

This account follows a general construction for the irreducible
representations of the semidirect product of a finite group acting on
an abelian group which can be found in \cite[Section 5.27]{zbMATH05943514}.
From this construction, we also get a formula for the dimension of
$\{\lambda; \mu\}$. Since the index of
$\mathbb{Z}_2^n \rtimes (\mathfrak{S}_k \times \mathfrak{S}_{n-k})$ in
$\mathfrak{H}_n$ is $\binom{n}{k}$, it follows from the definition of
induction that
\begin{equation}\label{eq:29}
  \dim_\Bbbk \{\lambda; \mu\} = \binom{n}{k} \dim_\Bbbk \{\lambda\}
  \dim_\Bbbk \{\mu\}.
\end{equation}
A dimension formula for the Specht modules $\{\lambda\}$ and $\{\mu\}$
can be found, for example, in \cite[Theorem 3.10.2]{zbMATH01601795}.

\section{More on the cellular complex of the cube}
\label{sec:more-cell-compl}

In \Cref{lem:1}, we described the representation-theoretic structure
of the cellular chain complex of the $n$-cube. We show here that this
structure determines the complex in an essentially unique way. The
next result allows us to identify the complex
$(E^1_{\bullet,0} d^1_{\bullet,0})$ in the spectral sequence for the
homology of $X^{n,r}$ without performing any computations, at the cost
of working over an algebraically closed field. This assumption is
needed only because of the stronger implications of Schur's Lemma over
algebraically closed fields.

\begin{proposition}\label{pro:5}
  Assume the field $\Bbbk$ is algebraically closed.  The cellular
  chain complex $(\mathcal{F}_\bullet, \partial_\bullet)$ of the unit
  $n$-cube $[0,1]^n$ is, up to isomorphism, the unique complex of
  $\mathfrak{H}_n$-representations over $\Bbbk$ satisfying the
  following conditions.
  \begin{enumerate}
  \item $\mathcal{F}_i \cong \operatorname{Ind}^{\mathfrak{H}_n}
    _{\mathfrak{S}_{n-i} \times \mathfrak{H}_i}
    (\{n-i\} \otimes \{0;1^i\})$
  \item $\partial_i ( \{n-i\} \otimes \{0;1^i\} ) \subseteq
    \{n-i\} \otimes \operatorname{Ind}^{\mathfrak{H}_i}
    _{\mathfrak{S}_1 \times \mathfrak{H}_{i-1}}
    (\{1\} \otimes \{0;1^{i-1}\})$
  \end{enumerate}
\end{proposition}

\begin{proof}
  We start by showing that $(\mathcal{F}_\bullet, \partial_\bullet)$
  satisfies both conditions. The first condition is the content of
  \Cref{lem:1}, so we turn to the second one.  The $i$-cell
  $\{0\}^{n-i} \times [0,1]^i$ spans a representation of
  $\mathfrak{S}_{n-i} \times \mathfrak{H}_i$ isomorphic to
  $\{n-i\} \otimes \{0;1^i\}$. The differential $\partial_i$ takes the
  $i$-cell $\{0\}^{n-i} \times [0,1]^i$ to a linear combination of its
  faces, which consist of points with the first $n-i$ coordinates
  equal to 0 and a single other coordinate equal to 0 for a front
  face, or 1 for a back face. As such, these faces are all in the
  $\mathfrak{S}_{n-i} \times \mathfrak{H}_i$ orbit of the face
  $\{0\}^{n-i} \times \{0\} \times [0,1]^{i-1}$, which spans a
  representation of
  $\mathfrak{S}_{n-i} \times \mathfrak{S}_{1} \times
  \mathfrak{H}_{i-1}$ isomorphic to
  $\{n-i\} \otimes \{1\} \otimes \{0;1^{i-1}\}$. In other words, the
  faces of $\{0\}^{n-i} \times [0,1]^i$ span a representation of
  $\mathfrak{S}_{n-i} \times \mathfrak{H}_i$ isomorphic to
  \begin{equation}\label{eq:30}
    \begin{split}
      &\operatorname{Ind}^{\mathfrak{S}_{n-i} \times \mathfrak{H}_i}
      _{\mathfrak{S}_{n-i} \times \mathfrak{S}_{1} \times \mathfrak{H}_{i-1}}
      (\{n-i\} \otimes \{1\} \otimes \{0;1^{i-1}\})\\
      \cong
      &\{n-i\} \otimes \operatorname{Ind}^{\mathfrak{H}_i}
      _{\mathfrak{S}_{1} \times \mathfrak{H}_{i-1}}
      (\{1\} \otimes \{0;1^{i-1}\})
    \end{split}
  \end{equation}
  and $\partial_i$ sends the cell $\{0\}^{n-i} \times [0,1]^i$ into
  this subspace. This shows the second condition holds.

  Now, we prove the uniqueness of the complex. Inducing from
  $\mathfrak{S}_{n-i} \times \mathfrak{H}_i$ to $\mathfrak{H}_n$ in
  \Cref{eq:30}, gives the $\mathfrak{H}_n$-representation
  \begin{equation*}
    \begin{split}
      \mathcal{G}_i
      &= \operatorname{Ind}^{\mathfrak{H}_n}_{\mathfrak{S}_{n-i} \times \mathfrak{H}_i}
        \operatorname{Ind}^{\mathfrak{S}_{n-i} \times \mathfrak{H}_i}
        _{\mathfrak{S}_{n-i} \times \mathfrak{S}_{1} \times \mathfrak{H}_{i-1}}
        (\{n-i\} \otimes \{1\} \otimes \{0;1^{i-1}\})\\
      &\cong \operatorname{Ind}^{\mathfrak{H}_n}
        _{\mathfrak{S}_{n-i} \times \mathfrak{S}_{1} \times \mathfrak{H}_{i-1}}
        (\{n-i\} \otimes \{1\} \otimes \{0;1^{i-1}\})
    \end{split}
  \end{equation*}
  after using transitivity of induction.  The second condition in the
  proposition implies that the differential
  $\partial_i \colon \mathcal{F}_i \to \mathcal{F}_{i-1}$ must factor
  as $\mathcal{F}_i \to \mathcal{G}_i \to \mathcal{F}_{i-1}$.  By
  \cite[Lemma III.6, Theorem III.2]{zbMATH03590472}, we have
  \begin{equation*}
    \begin{split}
      &\operatorname{Ind}^{\mathfrak{H}_i}
      _{\mathfrak{S}_{1} \times \mathfrak{H}_{i-1}}
        (\{1\} \otimes \{0;1^{i-1}\})\\
      \cong{}
      &\operatorname{Ind}^{\mathfrak{H}_i}
      _{\mathfrak{H}_{1} \times \mathfrak{H}_{i-1}}
      \left( (\{1;0\} \oplus \{0;1\}) \otimes \{0;1^{i-1}\} \right)\\
      \cong{}
      &\{1;1^{i-1}\} \oplus \{0;(21^{i-2})\} \oplus \{0;1^i\}.
    \end{split}
  \end{equation*}
  Using Schur's Lemma for algebraically closed fields \cite[Corollary
  2.3.10]{zbMATH05943514}, there is, up to multiplication by a nonzero
  scalar, a unique nonzero map
  \begin{equation*}
    \{n-i\} \otimes \{0;1^i\} \longrightarrow
    \{n-i\} \otimes \operatorname{Ind}^{\mathfrak{H}_i}
    _{\mathfrak{S}_{1} \times \mathfrak{H}_{i-1}}
    (\{1\} \otimes \{0;1^{i-1}\})
  \end{equation*}
  of $\mathfrak{S}_{n-i} \times \mathfrak{H}_i$-representations. This
  map is injective because the domain is 1-dimensional.  Inducing from
  $\mathfrak{S}_{n-i} \times \mathfrak{H}_i$ to $\mathfrak{H}_n$ gives
  an injection $\phi_i \colon \mathcal{F}_i \to \mathcal{G}_i$ of
  $\mathfrak{H}_n$-representations because induction between finite
  groups is an exact functor. Similarly, using Pieri's rule \cite[Section
  2.2]{zbMATH01001729}, we have
  \begin{equation*}
    \operatorname{Ind}^{\mathfrak{S}_{n-i+1}}
    _{ \mathfrak{S}_{n-i} \times \mathfrak{S}_{1}}
    (\{n-i\} \otimes \{1\})
    \cong \{n-i+1\} \oplus \{(n-i,1)\}.
  \end{equation*}
  By Schur's Lemma for algebraically closed fields, there is, up to
  multiplication by a nonzero scalar, a unique nonzero map
  \begin{equation*}
    \operatorname{Ind}^{\mathfrak{S}_{n-i+1}}
    _{ \mathfrak{S}_{n-i} \times \mathfrak{S}_{1}}
    (\{n-i\} \otimes \{1\})
    \otimes \{0;1^{i-1}\}
    \longrightarrow  \{n-i+1\} \otimes \{0;1^{i-1}\}
  \end{equation*}
  of $\mathfrak{S}_{n-i+1} \times \mathfrak{H}_{i-1}$-representations.
  This map is surjective because the codomain is 1-dimensional.
  Inducing from $\mathfrak{S}_{n-i+1} \times \mathfrak{H}_{i-1}$ to
  $\mathfrak{H}_n$ gives a surjection
  $\psi_i \colon \mathcal{G}_i \to \mathcal{F}_{i-1}$ of
  $\mathfrak{H}_n$-representations because induction between finite
  groups is exact. Altogether, this shows that
  $\partial_i = \psi_i \phi_i$ where $\psi_i$ and $\phi_i$ are
  uniquely determined up to multiplication by a nonzero
  scalar. Therefore, $\partial_i$ is also uniquely determined up to
  multiplication by a nonzero scalar, which makes the complex
  $(\mathcal{F}_\bullet, \partial_\bullet)$ unique up to isomorphism.
\end{proof}

\section{Computational methods}
\label{sec:comp-meth}

Some of the results in this paper depend on software computations. We
provide here a brief explanation for the setup behind these
computations. All code is available at
\url{https://github.com/galettof/VietorisRipsHypercube}. We report
approximate compute times for our reference machines, which include:
\begin{itemize}
\item a Dell XPS laptop with an Intel Core i9-12900HK processor and 64
  GB of RAM, and
\item an HP EliteDesk 800 G3 desktop with an Intel Core i5-7500
  processor and 64 GB of RAM.
\end{itemize}

\subsection{Macaulay2 package}
\label{sec:macaulay2-package}

Macaulay2 \cite{M2} is a software system devoted to supporting
research in algebraic geometry and commutative algebra. Its core
functionality includes Gröbner basis algorithms to solve problems in
commutative and homological algebra. Although the primary focus of
Macaulay2 is not on algebraic or discrete topology, we chose it for
its convenience. In particular:
\begin{itemize}
\item it handles simplicial complexes via the
  \texttt{SimplicialComplexes} package
  \cite{zbMATH07771378,SimplicialComplexesSource};
\item it handles finite group characters via the
  \texttt{BettiCharacters} package
  \cite{zbMATH07771377,BettiCharactersSource};
\item we are particularly familiar with it.
\end{itemize}
The \texttt{VietorisRipsHypercube} package for Macaulay2 is designed
to experiment with homology computations for Vietoris-Rips complexes
of hypercubes, especially using the filtration introduced in
\Cref{sec:suppl-filtr}. It offers the following functionality.
\begin{itemize}
\item Constructs the Vietoris-Rips complex $X^{n,r}$ as a simplicial
  complex via its Stanley-Reisner ideal.
\item Computes the complex $C^{p,r}_{\bullet,p}$ of simplices with
  cubic dimension $p$ in $X^{p,r}$.
\item Computes the homology of $C^{p,r}_{\bullet,p}$, giving an
  explicit presentation when nonzero.
\item Computes the character of the hyperoctahedral group
  $\mathfrak{H}_p$ on the homology of $C^{p,r}_{\bullet,p}$.
\item Constructs generators for the image of the differentials
  $d^1_{i,0} \colon E^1_{i,0} \to E^1_{i-1,0}$ in the first page of
  the spectral sequence for the homology of $X^{n,r}$.
\end{itemize}
The reader may consult the package documentation in our Github
repository.

This Macaulay2 package can perform many computations for $X^{n,r}$
with $n\leqslant 7$ at scales $r=2$ and $r=3$. The complex
$C^{7,3}_{\bullet,7}$ took approximately 11 hours to compute on one of
our reference machines. However, the use of this package becomes
impractical for $n\geqslant 8$ or at larger scales, even on more
powerful machines. This led us to develop specialized Python code to
handle some remaining cases needed for the proof of \Cref{thm:2}.

\subsection{Python code}
\label{sec:python-code}

We outline the algorithms behind the computations used in
\Cref{sec:cubic-dimens-eight}. A Python implementation is available in
our Github repository. First, we list all simplices in $X^{n,3}$
having diameter 3 and cubic dimension $p$, with
$8\leqslant p\leqslant 9$. The strategy depends on the dimension.

\begin{description}
\item[Dimension $p-2$] Start with a simplex
  $\Delta_k = [z,u,v_4,\dots,v_k,w_{k+1},\dots,w_p]$ from
  \Cref{pro:3}, \Cref{item:6}. Act with the left cosets of the
  subgroup
  $G_k = \mathfrak{S}_{\{3,\dots,k\}} \times
  \mathfrak{S}_{\{k+1,\dots,p\}}$ from \Cref{pro:3}, \Cref{item:8}, in
  $\mathfrak{H}_p$ to obtain all other simplices in the
  $\mathfrak{H}_p$-orbit of $\Delta_k$. Since $G_k$ may be smaller
  than the stabilizer of $\Delta_k$, there may be some redundancy, but
  this is still more efficient than acting with the entire group
  $\mathfrak{H}_p$. Repeat for all $\Delta_k$ in \Cref{pro:3},
  \Cref{item:6}.
\item[Dimension $p-1$] Start with a simplex $\Delta_k$ and add a
  vertex $x$ so that $\Delta_k \cup \{x\}$ is still in $X^{n,3}$. The
  stabilizer of $\Delta_k \cup \{x\}$ may not contain the subgroup
  $G_k$, so we also take all simplices $\Delta_k \cup \{y\}$ in the
  $G_k$-orbit of $\Delta_k \cup \{x\}$. Act with the left cosets of
  $G_k$ in $\mathfrak{H}_p$
  on all the simplices $\Delta_k \cup \{y\}$ obtained this way to
  recover all simplices in the $\mathfrak{H}_p$-orbit of
  $\Delta_k \cup \{x\}$. Repeat for all $\Delta_k$ in \Cref{pro:3},
  \Cref{item:6}.
\item[Dimension $p$ and $p+1$] Follow the steps for dimension
  $p-1$. Our code adds two or three vertices to $\Delta_k$ in all
  possible ways so that the resulting simplex is still in
  $X^{n,3}$. For $p=9$, we only go up to dimension $p$, since this is
  all that is needed for \Cref{thm:2}.
\end{description}

Next, we compute matrices for the restriction of the boundary map to
the simplices found above; these are the matrices $d'_1$, $d'_2$, and
$d'_3$ of \Cref{sec:cubic-dimens-eight}.  Working over $\mathbb{Z}/2$
allows us to ignore orientation, which simplifies the code. The
matrices are saved in SciPy Sparse NPZ format \cite{2020SciPy-NMeth}.

Finally, we compute the ranks of the matrices above. Since we could
not find a Python library to compute the rank of a sparse matrix over
$\mathbb{Z}/2$, we implemented a simple Gaussian elimination
algorithm, although this does not account for sparseness.  On one of
our reference machines, it took about 2 minutes to construct the three
matrices for $p=8$, and about 17 minutes to construct the two matrices
for $p=9$. Computing their ranks took a total of about 3 hours.

We also provide an option to save the matrices in a sparse format
compatible with Magma \cite{MR1484478}, which implements a more
efficient algorithm to compute the rank of sparse matrices over finite
fields. The same rank computations take only a few minutes in Magma;
however, Magma requires a license.

\subsection{A Polymake computation}
\label{sec:polym-comp}

As reported in \cite[Section 6.4.4]{zbMATH07819241}, Ziqin Feng
computed the homology of $X^{6,4}$ with $\mathbb{Z}/2$-coefficients
using the Ripser software package \cite{zbMATH07421252}. We computed
the homology of $X^{6,4}$ with $\mathbb{Q}$-coefficients using
Polymake \cite{zbMATH01538120}. Our results show that
\begin{equation*}
  \widetilde{H}_i (X^{6,4}; \mathbb{Q}) \cong
  \begin{cases}
    \mathbb{Q}^{239}, &i=7,\\
    \mathbb{Q}^{14}, &i=15,\\
    0, &i\neq 7,15.\\
  \end{cases}
\end{equation*}
In particular, $X^{6,4}$ has the same Betti numbers over
$\mathbb{Z}/2$ and $\mathbb{Q}$. Our computation was performed on the
Viking Cluster at Cleveland State University; we estimate it took
about 18 hours and at least 800 GB of RAM. The code used for this
computation is available in our Github repository.


\end{document}